\documentclass[a4paper,10pt]{article}
\usepackage[T1]{fontenc}
\usepackage{lmodern}
\usepackage{amsmath,amsthm,amssymb,amsfonts}
\usepackage[dvipsnames,svgnames,table]{xcolor}
\usepackage{pgf,tikz}
\usetikzlibrary{calc}
\usetikzlibrary{arrows.meta}
\usepackage{enumitem}
\usepackage{fullpage}
\usepackage{mathtools}
\usepackage{nicematrix}
\usepackage{bigdelim}
\usepackage{booktabs}
\usepackage[title]{appendix}
\usepackage{thmtools,thm-restate}
\usepackage[colorlinks=true,linkcolor=RoyalBlue,urlcolor=RoyalBlue,citecolor=PineGreen,hypertexnames=false]{hyperref}
\usepackage[nameinlink]{cleveref}
\usepackage[export]{adjustbox}
\usepackage{authblk}

\usepackage{imakeidx}
\makeindex[]

\usepackage{multicol}
\usepackage{relsize}
\usepackage{float}
\usepackage{tikz,pgf}
\usepackage{tikz-cd}
\usepackage{wasysym}
\usepackage{graphicx}
\usepackage{fancyhdr}
\usepackage{verbatim}

\usepackage{pgfplots}
\pgfplotsset{compat=1.15}
\usepackage{mathrsfs}
\usetikzlibrary{arrows}
\usepackage{algorithm}
\usepackage{algpseudocode}

\newcommand{\Ccal}{\mathcal{C}}

\numberwithin{equation}{section}

\theoremstyle{plain}
\newtheorem{theorem}{Theorem}[section]
\newtheorem{proposition}[theorem]{Proposition}
\newtheorem{corollary}[theorem]{Corollary}
\newtheorem{lemma}[theorem]{Lemma}
\newtheorem{conjecture}[theorem]{Conjecture}
\newtheorem{claim}[theorem]{Claim}
\newtheorem{introthm}{Theorem}

\theoremstyle{definition}
\newtheorem{definition}[theorem]{Definition}

\newtheorem{remark}[theorem]{Remark}

\newtheorem{example}[theorem]{Example}
\newtheorem{question}[theorem]{Question}
\newenvironment{claimproof}[1][\proofname]
	{
		\proof[#1]%
			
	}
	{
		\endproof
	}

 \newcommand\restr[2]{{
   \left.\kern-\nulldelimiterspace 
   #1 
   \vphantom{\small|} 
   \right|_{#2} 
   }}

\newcommand{\vf}{\mathbf{v}}

\newcommand{\F}{\mathcal{F}}
\newcommand{\HH}{\mathcal{H}}

\makeatletter
\def\mathcenterto#1#2{\mathclap{\phantom{#1}\mathclap{#2}}\phantom{#1}}
\let\old@widetilde\widetilde
\def\widetildeto#1#2{\mathcenterto{#2}{\old@widetilde{\mathcenterto{#1}{#2\,}}}}
\let\old@widehat\widehat
\def\widehatto#1#2{\mathcenterto{#2}{\old@widehat{\mathcenterto{#1}{#2\,}}}}
\makeatother

\newcommand{\OW}{\delta_{\textup{OW}}(\vf)}

\newcommand{\size}[1]{\left| #1 \right|} 

\DeclareMathOperator{\rank}{rk}

\def\dim{\operatorname{dim}}
\DeclareMathOperator{\supp}{supp}

\DeclareMathOperator{\Lat}{Lat}

\DeclareMathOperator{\CC}{\mathbb{C}}

\DeclareMathOperator{\QQ}{\mathbb{Q}}
\DeclareMathOperator{\RR}{\mathbb{R}}

\setlist[itemize]{noitemsep}

\tikzstyle{vertex}=[fill=black,circle,inner sep=0pt, minimum size=4pt]
\tikzstyle{edge}=[line width=1.5pt,black]
\tikzstyle{bedge}=[line width=1.5pt,RoyalBlue]
\tikzstyle{gedge}=[line width=1.5pt,black!15]
\tikzstyle{dedge}=[line width=1.5pt,black,dashed]

\definecolor{myorange}{RGB}{239,128,108}
\definecolor{mypurple}{RGB}{117,0,235}
\tikzstyle{oedge}=[line width=1.5pt,myorange]
\tikzstyle{pedge}=[line width=1.5pt,mypurple]

 \usepackage[maxbibnames=99]{biblatex}
\title{Matroid liftability and Oxley--Wang derived matroids, with connections to scene analysis}
\author[1]{Sean Dewar}
\author[2]{Emiliano Liwski}
\author[1]{Signe Lundqvist}
\author[1,2]{Fatemeh Mohammadi}
\date{}

\affil[1]{Department of Computer Science, KU Leuven}
\affil[2]{Department of Mathematics, KU Leuven}

\begin{document}
\maketitle
\begin{abstract}
Liftability was first introduced for paving matroids, where it was used to systematically construct polynomials in the ideals of the associated matroid 
varieties.
In this paper, we generalise the notion of liftability to matroids with realisable truncations. 
We additionally introduce the notion of generic realisability of a matroid; a strong notion of realisability that implies both realisability and liftability of the matroid. 
We study the behaviour of liftability and generic realisability under the weak order of matroids within the lattice of matroids with fixed truncation. 
We develop a unified framework for studying liftability of matroids through Oxley--Wang derived matroid of generic matroid realisations.
We characterise liftability and generic realisability in terms of Oxley--Wang derived matroids. 
Moreover, we give an algebraic description of the bases of the generic Oxley--Wang derived matroids in terms of coordinate sections of matroid realisation spaces and relate our constructions to slack realisation spaces of matroids. 

As an application of our framework, we obtain a combinatorial characterisation of liftable and generically realisable paving matroids and connect it with Whiteley's theory of scene analysis. In particular, our approach gives a new proof, via Oxley--Wang derived matroids, of Whiteley's characterisation of hypergraphs admitting non-trivial (sharp) scenes. 
Finally, we compute the generic Oxley--Wang derived matroid for a special class of sparse paving matroids.
\end{abstract}

\section{Introduction}\label{sec:intro}

Determining when a rank-$d$ matroid $M$ with ground set $[n]$ has a \emph{realisation} -- a tuple of vectors $(v_1, \ldots, v_n)$ in $\mathbb{C}^d$ which represent $M$ -- is a central problem in matroid theory.
The space of such realisations of $M$ forms a quasi-affine variety $\Gamma_M \subset (\mathbb{C}^d)^n$ called the \emph{realisation space of $M$}. The {\em matroid variety} $V_{M}$ is defined as the Zariski closure of $\Gamma_{M}$ within $(\mathbb{C}^d)^n$.
The equations defining the realisation space can be relaxed to form the \emph{circuit variety} $V_{\mathcal{C}(M)} \subset (\mathbb{C}^d)^n$: the set of vector tuples $(v_1,\ldots, v_n)$ for which a subfamily of vectors $(v_e : e \in E)$ are linearly dependent if $E$ is dependent in $M$.
From an algebraic geometry perspective, it is often easier to work with circuit variety: it is a non-empty Zariski-closed set, its defining equations are given by the circuits of $M$, and any irreducible component of $V_{M}$ is an irreducible component of $V_{\mathcal{C}(M)}$.

Matroid and circuit varieties are well-studied objects lying at the intersection of 
algebraic geometry and combinatorics~\cite{gelfand1987combinatorial,clarke2021matroid,
sidman2021geometric,Fatemeh4,brandt2019slack}. Their complexity is shown by the 
Mn\"{e}v Universality Theorem~\cite{mnev1988universality}, which asserts that every 
singularity type of a semi-algebraic set can be realised within a matroid variety. 
The ideals of these varieties arise naturally as primary components in the decomposition 
of several families of ideals, including conditional independence 
ideals~\cite{clarke2020conditional,clarke2022conditional,
alexandersson2026decomposing} and determinantal 
ideals~\cite{bruns2003determinantal,herzog2010binomial}. The notion of \emph{matroid 
liftability} was introduced in~\cite{Fatemeh4} for paving matroids,
which formalised and generalised the concepts explored in \cite{clarke2024liftable} for rank-3 matroids.
There liftability served as a systematic method to construct polynomials in the ideals of the associated matroid 
varieties. In this paper, we generalise liftability beyond the paving setting to 
arbitrary matroids.

\begin{definition}
    Let $M$ be a rank-$d$ matroid with ground set $[n]$ whose \emph{truncation} -- the rank $(d-1)$ matroid $T(M)$ whose independent sets are exactly the independent sets of $M$ with size $d-1$ or less -- is realisable over $\mathbb{C}$.
    A \emph{lifting} of a realisation $\vf = (v_1, v_2, \ldots, v_n) \in \Gamma_{T(M)}$ of $T(M)$ to a point in $V_{\mathcal{C}(M)}$ is a tuple 
    \begin{equation*}
        \widetilde{\vf} = \big( (v_1 ,z_1), (v_2 ,z_2) \ldots,  (v_n ,z_n) \big) \in (\mathbb{C}^d)^n
    \end{equation*}
    which is contained in the circuit variety $V_{\mathcal{C}(M)}$;
    moreover, a lifting $\widetilde{\vf}$ is \emph{non-trivial} if the linear span of $\widetilde{\vf}$ is $d$-dimensional.
    The matroid $M$ is said to be {\em liftable} if every realisation of $T(M)$ has a non-trivial lifting to a point in $V_{\mathcal{C}(M)}$.
\end{definition}

In the case where $M$ is a rank-$d$ \emph{paving matroid} -- a matroid where every non-spanning circuit has size $d$ -- the truncation of $M$ is the rank-$(d-1)$ uniform matroid.
In this case, the realisation space $\Gamma_{T(M)}$ consists of all $n$-tuples of vectors in $\mathbb{C}^{d-1}$ that are in \emph{general position} (i.e., any $d-1$ vectors of the tuple are linearly independent).
Our definition thus aligns with the concept of liftability for paving matroids previously defined in \cite{clarke2024liftable,Fatemeh4}.

To explore the links between liftability and realisability, we additionally introduce the following stronger notion of realisability.
Here a realisation $\vf$ of a matroid $M$ is \emph{generic} if, roughly speaking, $\vf$ does not satisfy additional rational algebraic constraints than what are required to be a point in the irreducible component of $\Gamma_M$ containing it (see \Cref{def:generic} for a rigorous definition).

\begin{definition}\label{def:genericallyrealisable}
    Let $M$ be a rank-$d$ matroid with ground set $[n]$.
    The matroid $M$ is said to be {\em generically realisable} 
    if every generic realisation of $T(M)$ has a 
    lifting that represents $M$.
\end{definition}

\subsection{Our contribution}

\subsubsection{1-elevation lattice structure}

A \emph{1-elevation} of a matroid $N$\footnote{Also known in the literature as an \emph{erection} of $N$.} is any matroid $M$ with $T(M)=N$.
Crapo proved that the set of 1-elevations along with the matroid $N$ forms a lattice $\Lat(N)$ of matroids under the \emph{weak order of matroids}; the partial ordering of matroids on a shared ground set where $M_1 \leq M_2$ if and only if every independent set of $M_1$ is an independent set of $M_2$.
Each realisation $\mathbf{v}$ of $N$ generates two subsets of $\Lat(N)$:
\begin{enumerate}[label=(\roman*)]
    \item $L_{\mathbf{v}} \subseteq \Lat(N)$ is the collection of matroids where for each $M \in L_{\mathbf{v}}$ there exists a non-trivial lifting $\widetilde{\mathbf{v}}$ of $\mathbf{v}$.
    \item $R_{\mathbf{v}} \subseteq \Lat(N)$ is the collection of matroids where for each $M \in R_{\mathbf{v}}$ there exists a non-trivial lifting $\widetilde{\mathbf{v}}$ of $\mathbf{v}$ that is also a realisation of $M$.
\end{enumerate}
It is clear that $R_{\mathbf{v}} \subseteq L_{\mathbf{v}}$.
Additionally, we can characterise the liftable, realisable and generically realisable matroids in $\Lat(N)$ using these sets as follows (see \Cref{prop:genrealimplieslift}):
\begin{center}
\begin{tikzcd}
\begin{array}{c}
    \text{Realisable} \\
     \bigcup_{\mathbf{v} \in \Gamma_N} R_{\mathbf{v}}
\end{array}
& 
\begin{array}{c}
    \text{Generically realisable} \\
     \bigcap_{\mathbf{v} \in \Gamma_N \text{generic}} R_{\mathbf{v}}
\end{array} 
\arrow[l, Rightarrow] \arrow[r, Rightarrow] & 
\begin{array}{c}
    \text{Liftable} \\
     \bigcap_{\mathbf{v} \in \Gamma_N} L_{\mathbf{v}}
\end{array}
\end{tikzcd}
\end{center}
We additionally show that neither liftability nor realisability imply generic realisability (even when combined), and neither realisability nor liftability imply the other;
see \Cref{subsec:rvl}.

We prove that the sets of liftable and generically realisable matroids are `closed' in a meaningful sense.

\begin{introthm}[Lattice structure for liftability]\label{thm:A}
    Let $N$ be a realisable rank $(d-1)$-matroid on $[n]$.
    Then the set of liftable matroids in $\Lat(N)$ is upwards-closed,
    and the set of generically realisable matroids in $\Lat(N)$ is join-closed.
    
    Moreover, if $L_{\mathbf{v}} = L_{\mathbf{w}}$ for all generic $\mathbf{v},\mathbf{w} \in \Gamma_N$ (e.g., $V_N$ is irreducible),
    then the minimal liftable matroids in $\Lat(N)$ are generically realisable.
\end{introthm}

\subsubsection{Linking liftability to Oxley--Wang matroids}

We prove the following link between liftability and \emph{generic Oxley--Wang derived matroids} \cite{oxley2019dependencies,freij2023combinatorial} (see \Cref{def: generic oxley wang} for full terminology):

\begin{introthm}[Liftability via Oxley--Wang matroid]\label{thm:B}
Let $M$ be a rank-$d$ matroid on $[n]$ such that $T(M)$ is realisable.
Suppose that $T(M)$ has generic Oxley--Wang derived matroids $\delta_1 T(M),\ldots, \delta_k T(M)$\footnote{This should not be confused with the \emph{derived sequence} introduced in \cite{oxley2019dependencies}.}. 
Then, $M$ is liftable if and only if
\[
\max_{i=1,\ldots,k}\rank_{\delta_{i}T(M)}
   ( \mathcal{C}^{\text{ns}}(M) )
   \leq n-d, 
\]
where $\mathcal{C}^{\text{ns}}(M)$ denotes the collection of non-spanning circuits of $M$.
\end{introthm}

We can also translate the property of generic realisability of $M$ to a condition on the generic Oxley--Wang matroids of $T(M)$.

\begin{introthm}[Generic realisability via Oxley--Wang matroid]
    \label{thm:C}
    Let $M$ be a rank-$d$ matroid on $[n]$ such that $T(M)$ is realisable.
Suppose that $T(M)$ has generic Oxley--Wang derived matroids $\delta_1 T(M),\ldots, \delta_k T(M)$. Then, $M$ is generically realisable if and only if 
\[\text{$\mathcal{C}^{\text{ns}}(M)$ is a flat on each of the matroids $\delta_1 T(M),\ldots, \delta_k T(M)$.}\]
\end{introthm}

Moreover, in~\Cref{thm: fixed support} we give an algebraic description of the bases of the generic Oxley--Wang derived matroids in terms of coordinate sections of matroid realisation spaces and relate our constructions to slack realisation spaces of matroids in~\Cref{sec: slack}.

\subsubsection{Linking liftability to scene analysis}\label{subsub:lifttoscene}

We then prove the following link between liftability and \emph{scene analysis} for any \emph{paving matroid}; a matroid $M$ where every circuit contains at least $\rank(M)$ elements.
Using Whiteley's characterisation of non-trivial and sharp scenes for generic pictures \cite{whiteley1989matroid},
each result can be expressed using the $(1,d\!-\!1)$-sparsity matroid for $d$-uniform hypergraphs (see \Cref{def: sparse}), which we denote here by $\mathcal{S}_{(1,d-1)}$.
More background on scene analysis can be found in \Cref{sec:sceneanalysisbackground}.

We prove the following result in two different ways:
our first proof uses an equivalence between liftability and scene analysis and then employs Whiteley's characterisation of non-trivial scenes for generic pictures (\Cref{thm:indepA});
for our second proof, we compute the unique generic Oxley--Wang matroid of a uniform matroid using \Cref{thm:B} and Hall's marriage theorem (\Cref{thm:hall}).
This second method additionally provides an alternative proof of Whiteley's picture theorem (\Cref{thm:exists_sharp_scene}).

\begin{introthm}[Liftability as scene analysis]\label{thm:D}
    Let $M$ be a rank-$d$ paving matroid on $[n]$, and let $\mathcal{C}_{d}(M)$ denote the collection of circuits of size $d$ of $M$. Then the matroid $M$ is liftable if and only if
  \[
  \rank_{\mathcal{S}_{(1,d-1)}}\!\bigl(\mathcal{C}_{d}(M)\bigr) \leq n-d,
  \]
  where $\mathcal{S}_{(1,d-1)}$ denotes the $(1,d\!-\!1)$-sparsity matroid, as defined in~\Cref{def: sparse}.
\end{introthm}

We also provide an exact combinatorial characterisation of the generically realisable paving matroids, 
including a stronger characterisation for \emph{sparse paving matroids}: paving matroids whose dual is also a paving matroid.

\begin{introthm}[Generic realisability for paving matroids]\label{thm:E}
Let $M$ be a rank-$d$ paving matroid on $[n]$. Then the matroid $M$ is generically realisable if and only if $\mathcal{C}_d(M)$ is a flat of the matroid $\mathcal{S}_{(1,d-1)}$.

Furthermore, if $M$ is sparse and $d\geq 2$ then $M$ is generically realisable if and only if the associated hypergraph $H = ([n],C_d(M))$ is $(1,d)$-sparse.
\end{introthm}

It was proven by Streinu and Theran \cite{streinu2009sparse_pebble_game} that the rank of a $d$-uniform hypergraph with $n$ vertices and $m$ hyperedges can be computed in $O((2d-1)dnm)$ time using $O(n+m)$ space.
Thus it can be determined in polynomial time using polynomial space (both with respect to $n$ and $d$) whether or not a rank-$d$ paving matroid on $[n]$ is liftable.
It is currently open whether there exists a polynomial-time algorithm (in $n$ and $d$) for determining if the same such matroid is generically realisable.
As determining whether or not a $d$-uniform hypergraph is $(1,d)$-sparse is NP-Hard with respect to $n$ and $d$ (as shown by Berg and Jord\'{a}n \cite[Theorem 5]{berg2003algorithms}),
it seems likely that no such algorithm exits.

We additionally show that the following implications\footnote{Our implication diagram uses `$\exists$ generic (sharp) scene' as shorthand for `there exists a non-trivial (sharp) scene that lifts a generic picture' and `$\exists$ general position (sharp) scene' as shorthand for `there exists a non-trivial (sharp) scene that lifts some general position picture'.} hold for any sparse paving matroid $M$ on $[n]$ and its associated hypergraph $H = ([n],E)$ with hyperedges given by the non-spanning circuits of $M$:

\vspace{3mm}
\adjustbox{center,valign=m,scale=0.95}{%
\begin{tikzcd}
\text{$M$ realisable} \arrow[d, Rightarrow]             & \text{$\exists$ generic sharp scene of $H$} \arrow[d, Rightarrow] \arrow[ld, Rightarrow] \arrow[l, Rightarrow] & \text{$M$ generically realisable} \arrow[d, Rightarrow] \arrow[l, Leftrightarrow] \\
\text{$\exists$ general position sharp scene of $H$} \arrow[rd, Rightarrow] & \text{$\exists$ generic scene of $H$} \arrow[d, Rightarrow] \arrow[r, Leftrightarrow]                  & \text{$M$ liftable} \arrow[d, Rightarrow]                                     \\
                                                    & \text{$\exists$ general position scene of $H$}                                                                      & \text{$\exists$ non-trivial lifting of $M$} \arrow[l, Leftrightarrow]               
\end{tikzcd}%
}
\vspace{3mm}

Any implication not shown in the above diagram does not hold;
see \Cref{rem:counterexamplelist} for counter-examples.

\subsubsection{Characterising 
bases of the Oxley-Wang derived matroid 
for a special family of sparse paving matroids}

The next most simple family of generic Oxley--Wang derived matroids to study are those constructed from sparse paving matroids.
We study the following special family of sparse paving matroids:

\begin{definition}
    A sparse paving matroid has \emph{maximally intersecting circuits} 
    if the union of any two distinct non-spanning circuits is the entire ground set.
    Equivalently,
    a sparse paving matroid has maximally intersecting circuits if and only if its dual has pairwise-disjoint non-spanning circuits.
\end{definition}

Every sparse paving matroid $M$ with maximally intersecting circuits has a unique generic Oxley--Wang derived matroid, here denoted $\delta M$ (see \Cref{prop:irreduciblematroidvar}).
We are able to exactly compute the bases of $\delta M$ using our following result.

\begin{introthm}\label{thm:F}
Let $M$ be a rank-$d$ sparse paving matroid with maximally intersecting circuits.
    Then $A \subseteq \mathcal{C}(M)$ is a basis of $\delta M$ if and only if $|A|=n-d$ and all proper subsets $B \subsetneq A$ with $|B| \geq 2$ we satisfy
    \[
    \size{B} \le
    \begin{cases}
        \size{V(B)} - d -1 &\text{if $V(B)$ contains a non-spanning circuit $c$ of $M$ not contained in $B$}, \\
        \size{V(B)} - d &\text{otherwise},
    \end{cases}
    \]
    where $V(B) \subseteq [n]$ denotes the union of the circuits in $B$.
\end{introthm}

\subsection{Structure of paper}

In \Cref{section 2},
we recall the necessary background on matroids, realisation spaces, truncations, and scene analysis. In~\Cref{sec:liftabilitynonpaving}, we introduce the notion of circuit-matrices for studying liftability and generic realisability. In~\Cref{sec:matroidlattice}, we study the $1$-elevation lattice of a matroid $N$, consisting of all matroids that truncate to $N$, and investigate how liftability and generic realisability behave under the weak order. In~\Cref{sec:liftabilityviaOW}, we connect liftability and generic realisability with Oxley--Wang derived matroids. We characterise liftability and generic realisability in terms of Oxley--Wang derived matroids, and give an algebraic description of their bases through coordinate sections of matroid realisation spaces and relate our constructions to slack realisation spaces of matroids. In~\Cref{sec:applications}, 
we apply our framework to paving matroids and scene analysis, obtaining a combinatorial characterisation of liftability and generic realisability for paving matroids and 
a new proof of Whiteley's characterisation of hypergraphs that admit non-trivial scenes. Finally, in~\Cref{sec:simple_paving_OW}, we compute the generic Oxley--Wang derived matroids for a class of sparse paving matroids with maximally non-intersecting circuits.

\section{Preliminaries}\label{section 2}

\subsection{Matroid terminology}

A \emph{matroid} $M$ consists of a \emph{ground set} $[n]$ together with a collection $\mathcal{I}(M)$ of subsets of $[n]$, called \emph{independent sets}, satisfying the following axioms:
\begin{enumerate}[label=(\roman*)]
\item $\emptyset \in \mathcal{I}(M)$;
\item if $I \in \mathcal{I}(M)$ and $I' \subseteq I$, then $I' \in \mathcal{I}(M)$;
\item if $I_1, I_2 \in \mathcal{I}(M)$ with $|I_1| < |I_2|$, then there exists $e \in I_2 \setminus I_1$ such that $I_1 \cup \{e\} \in \mathcal{I}(M)$.
\end{enumerate}
Any set that is not independent is said to be \emph{dependent}.
We also recall the following matroid definitions.
\begin{itemize}
\item A \emph{circuit} is a minimally dependent subset of $[n]$. The set of all circuits is denoted by $\mathcal{C}(M)$.
\item The \emph{rank} of a subset $F \subseteq [n]$, denoted $\rank(F)$, is the size of the largest independent subset contained in $F$. The \emph{rank of the matroid} $M$, denoted $\rank(M)$, is defined as $\rank([n])$.
\item A \emph{basis} is a maximal independent subset of $[n]$. The collection of all bases is denoted by $\mathcal{B}(M)$.
\item A set $F \subseteq [n]$ is a \emph{flat} if $\rank(F') > \rank (F)$ for all proper supersets $F \subsetneq F' \subseteq [n]$.
Any flat of rank $\rank(M)-1$ is a \emph{hyperplane}.
We denote the set of hyperplanes of $M$ by $\mathcal{H}(M)$.
\item A set is \emph{spanning} if it has rank $\rank(M)$, and \emph{non-spanning} otherwise.
\item A {\em non-spanning circuit} is any circuit $c$ with $\rank(c) \leq \rank(M) -1$, or equivalently with $\size{c}\leq \rank(M)$. We denote the collection of non-spanning circuits of $M$ by $\Ccal^{\text{ns}}(M)$.
\end{itemize}

There are two important operations that can be applied to a matroid $M$ with ground set $[n]$ and rank $d \geq 1$ to produce a new matroid with ground set $[n]$:
\begin{itemize}
    \item The {\em truncation} of $M$, here denoted $T(M)$, is the matroid of rank $d-1$ on $[n]$ whose independent sets are those of $M$ of size at most $d-1$.
    \item The {\em dual} of $M$, here denoted $M^*$, is the matroid whose basis sets are exactly the complements of basis sets of $M$.
\end{itemize}

Let $M$ and $N$ be matroids on $[n]$. We write $M\geq N$ if $\mathcal{I}(M)\supseteq \mathcal{I}(N)$. This defines a partial order on the family of matroids with common ground set, called the {\em weak order}.
We observe here that, if $N=T(M)$, then $M \geq N$.

\subsection{Matroid realisation spaces and circuit varieties}

Here we introduce the notions of matroid realisation spaces and circuit varieties; see \cite{clarke2021matroid,Fatemeh4} for further details.

Let $M$ be a matroid of rank $d$ on $[n]$. A \emph{realisation} of $M$ is a tuple of vectors $\vf = (v_{1}, \ldots, v_{n}) \in (\CC^{d})^{n}$ such that
\[
\{i_{1}, \ldots, i_{k}\} \text{ is dependent in } M 
\;\Longleftrightarrow\;
\{v_{i_{1}}, \ldots, v_{i_{k}}\} \text{ are linearly dependent}.
\]
Equivalently, $\vf$ may be viewed as a matrix in $\CC^{d \times n}$ whose columns are $v_{1}, \ldots, v_{n}$.
The \emph{realisation space} of $M$ is defined by
\[
\Gamma_{M} = \left\{ \vf \in (\CC^{d})^{n} \;\middle|\; \vf \text{ is a realisation of } M \right\}.
\]
The matroid $M$ is said to be \emph{realisable} if $\Gamma_{M} \neq \emptyset$. 
The \emph{matroid variety} $V_{M}$ is defined as the Zariski closure of $\Gamma_{M}$ in $\CC^{dn}$.

It is often difficult to work with realisation spaces or matroid varieties, and so we instead wish to work with a larger and more-easily definable set.
This motivates the following definition.
Let $M$ be a matroid of rank $d$ on $[n]$. 
A tuple of vectors $\vf = (v_1, \ldots, v_n) \in (\CC^d)^n$ is said to \emph{include the dependencies} of $M$ if
\begin{equation}\label{includes dependencies}
\{i_1, \ldots, i_k\} \text{ is dependent in } M 
\;\Longrightarrow\;
\{v_{i_1}, \ldots, v_{i_k}\} \text{ is linearly dependent}.
\end{equation}
The \emph{circuit variety} of $M$ is defined by
\[
V_{\Ccal(M)}
= \left\{ \vf \in (\CC^d)^n \;\middle|\; \vf \text{ includes the dependencies of } M \right\}.
\]
Each irreducible component of the matroid variety is also an irreducible component of the circuit variety,
and so the matroid variety sits nicely inside the circuit variety.

We will later require the following two technical results. For a tuple of vectors $\vf=(v_{1},\ldots,v_{n})\in (\CC^{d})^{n}$, we denote by $\mathcal{M}(\vf)$ the unique matroid on $[n]$ that $\mathbf{v}$ is a realisation of.

\begin{lemma}\label{lem:circvar1}
    Let $M$ be a matroid and let $\vf=(v_{1},\ldots,v_{n})\in (\CC^{d})^{n}$. Then, $\mathbf{v} \in V_{\Ccal(M)}$ if and only if $\mathcal{M}(\mathbf{v}) \leq M$.
\end{lemma}

\begin{proof}
It follows from~\eqref{includes dependencies} that $\vf\in V_{\Ccal(M)}$ if and only if every dependent set in $M$ is dependent in $\mathcal{M}(\vf)$, which is equivalent to $\mathcal{I}(\mathcal{M}(\vf))\subseteq \mathcal{I}(M)$, which is equivalent to $\mathcal{M}(\mathbf{v}) \leq M$ by definition.
\end{proof}

\begin{lemma}\label{lem:circvar2}
    Let $M, N$ be matroids on the ground set $[n]$.
    If $M \leq N$, then $V_{\Ccal(M)} \subseteq V_{\Ccal(N)}$.
\end{lemma}

\begin{proof}
Let $\vf\in V_{\Ccal(M)}$. Applying \Cref{lem:circvar1} with $\vf$ and $M$, it follows that $\mathcal{M}(\vf)\leq M\leq N$.
Applying the lemma again with $\vf$ and $N$, we obtain $\vf\in V_{\Ccal(N)}$. This proves the inclusion.
\end{proof}

As we observe in the following remark,
matroid varieties and realisation spaces are well-behaved with respect to matroid duality.

\begin{remark}\label{rem:dual}
    If $\vf$ is a realisation of a rank-$d$ matroid $M$ on $[n]$,
    then any matrix $\mathbf{w} \in \mathbb{C}^{(n-d) \times d}$ with rank $n-d$ satisfying $\vf \mathbf{w}^\top = \mathbf{0}_{d \times (n-d)}$ is a realisation of the dual matroid $M^*$.
    Since the family of rank-($n-d$) matrices satisfying $\vf \mathbf{w}^\top = \mathbf{0}_{d \times (n-d)}$ for a given realisation $\vf$ of $M$ all form realisations of $M^*$, it follows that each irreducible component $V$ of $\Gamma_M$ has a unique corresponding irreducible component $V^*$ of $\Gamma_{M^*}$ for which the following holds:
    for each $\vf \in V$, there exists $\mathbf{w} \in V^*$ satisfying $\vf \mathbf{w}^\top = \mathbf{0}_{d \times (n-d)}$. In particular, if $M$ is irreducible, then so is $M^*$
\end{remark}

\subsection{Generic matroid realisations}

There are many different definitions for a `generic' point of an algebraic variety.
In this paper we opt for one of the stronger definitions.

\begin{definition}\label{def:genericpoint}
    Let $V \subset \mathbb{C}^m$ be a \emph{rationally-defined} variety; that is, $V$ is the locus of a finite family of rational-coefficient polynomials.
    If $V$ is irreducible, a point $x \in V$ is said to be a \emph{generic point of $V$} if for every rational-coefficient polynomial $f \in \mathbb{Q}[X_1,\ldots,X_{m}]$ we have $f(x)=0$ if and only if $f(y)=0$ for all $y \in V$.
    If $V$ is reducible, then a point in $V$ is a \emph{generic point of $V$} if it is a generic point of any irreducible component of $V$ containing it.
\end{definition}

It is a standard exercise in algebraic geometry to show that every rationally-defined variety contains a dense subset of generic points, and a generic point of a variety is contained within exactly one irreducible component. As the circuit variety, matroid variety and realisation space of a matroid are all defined by rational-coefficient polynomials,
this strong definition of `genericity' can be applied to our setting.

\begin{definition}\label{def:generic}
    A realisation $\vf$ of a matroid $M$ is \emph{generic} if 
    it is a generic point of $V_M$.
\end{definition}

\subsection{The fundamentals of scene analysis}
\label{sec:sceneanalysisbackground}
In this section we describe \emph{scene analysis}, a classical topic in combinatorial geometry.

\begin{remark}
  All the results of this section are presented over $\mathbb{R}$, as is usually done in the literature about scene analysis. However, all results remain valid over $\CC$.  
\end{remark}

For a $d$-uniform hypergraph $H=(V,E)$, $d \geq 3$, with at least 1 hyperedge, we define the following:
\begin{itemize}
    \item A \emph{picture} of $H$ is a map $x: V \to \mathbb{R}^{d-2}$, $v \mapsto x_v = (x_{v,1},\ldots,x_{v,d-2})$. 
    We say $x$ is \emph{generic} if the multiset $(x_{v,i} : v \in V,~ i \in [d-2])$ is algebraically independent over $\mathbb{Q}$.
    Similarly, $x$ is a \emph{general position picture} if no $(d-1)$-tuple of points $x_{v_1},\ldots,x_{v_{d-1}}$ are contained in an affine hyperplane of $\mathbb{R}^{d-2}$.
    \item A \emph{scene} of $H$ is a map $y: V \to \mathbb{R}^{d-1}$, $v \mapsto y_v$, such that if $\{v_1,\ldots,v_d\}$ is a hyperedge of $H$, then the points $y_{v_1}, \ldots, y_{v_d}$ are contained in an affine hyperplane of $\mathbb{R}^{d-1}$.
    \item A scene $y$ of $H$ is \emph{sharp} if the points $\{y_{v_1},\ldots,y_{v_d}\}$ corresponding to a hyperedge $e=\{v_1,\ldots,v_d\}$ are contained in a unique affine hyperplane $L_e$,
    and $L_e \neq L_f$ for any two distinct hyperedges $e,f \in E$.
    \item A scene $y$ of $H$ is \emph{trivial} if the points $\{y_v: v \in V\}$ are contained in an affine hyperplane.
    \item A scene $y$ of $H$ \emph{lifts} a picture $x$ if for each $v \in V$ there exists $z_v \in \RR$ such that $y_{v} = (x_v,z_v)$.
\end{itemize}
The set of scenes that lift a picture $x$ form an affine subspace of $(\mathbb{R}^{d-1})^V$.
Moreover, if the points $\{x_v: v \in V\}$ are not contained in any affine hyperplane of $\mathbb{R}^{d-2}$, then the set of trivial scenes that lift $x$ form a $(d-1)$-dimensional affine subspace of the space of scenes that lift $x$.
We say that a picture $x$ of $H$ is \emph{flat} if every lift of $x$ is trivial.

\begin{figure}[H]
    \centering
    \includegraphics[width=1\textwidth, trim=0 0 0 0, clip]{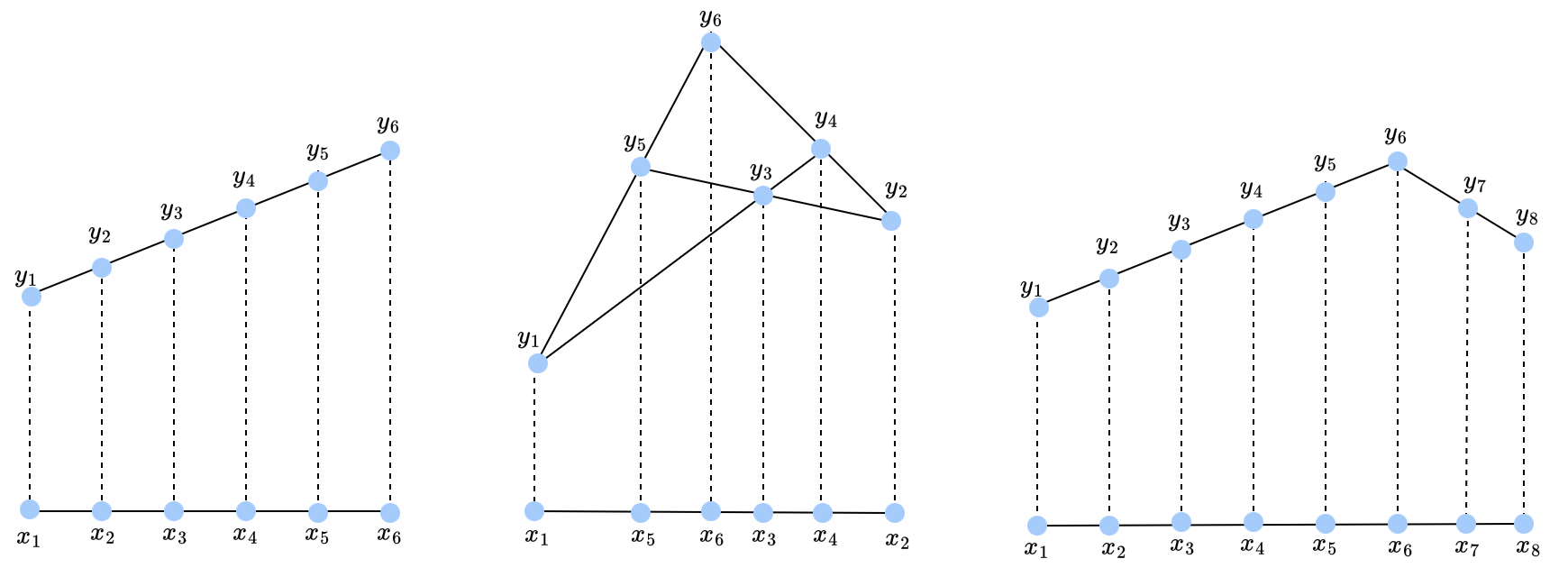}
    \caption{(a) A trivial scene; (b) A sharp scene; (c) A scene which is neither trivial nor sharp.}
    \label{figure: pictures and scenes}
\end{figure}

\begin{example}
    Consider the $3$-uniform hypergraph $H$ on the vertex set $[6]$ with hyperedge set $E:=\{\{1,5,6\},\{1,3,4\},\{2,3,5\},\{2,4,6\}\}$. \Cref{figure: pictures and scenes} (a) and (b) show two lifts of a picture of $H$, corresponding respectively to a trivial scene and a sharp scene. \Cref{figure: pictures and scenes} (c) shows a scene that is neither trivial nor sharp with respsect to the hypergraph $H'$ on $[8]$ with edge set $E\cup\{6,7,8\}$.
\end{example}

The main question of scene analysis is the following:

\begin{question}
    Suppose that $x$ is a picture of a $d$-uniform hypergraph $H=(V,E)$. Is there a sharp scene that lifts $x$?
    \label{q:scene_analysis}
\end{question}

To better understand \Cref{q:scene_analysis}, 
we recall the following observation from \cite{whiteley1989matroid}.
Let $H=(V,E)$ be a $d$-uniform hypergraph with picture $x$.
Then a map $y :V \rightarrow \mathbb{R}^{d-1}$ with $y_v=(x_v,z_v)$ is a scene that lifts $x$ if and only if for each hyperedge $e=\{v_1,\ldots,v_d\} \in E$ we have
\begin{equation*}
    \det \begin{pmatrix}
x_{v_{1},1} & \cdots & x_{v_{d},1} \\
\vdots & \vdots & \vdots \\
x_{v_{1},d-2} & \cdots & x_{v_{d},d-2} \\
z_{v_{1}} & \cdots & z_{v_{d}} \\
1 & \cdots & 1
\end{pmatrix} = 0.
\end{equation*}
Alternatively, if we define for each hyperedge $e=\{v_1,\ldots,v_d\}$ with an arbitrary ordering $v_1, \ldots , v_d$ the matrix
\begin{equation*}
    X(e) :=    
    \begin{pmatrix}
    x_{v_{1},1} & \cdots & x_{v_{d},1} \\
    \vdots & \vdots & \vdots \\
    x_{v_{1},d-2} & \cdots & x_{v_{d},d-2}\\
    1 & \cdots & 1
    \end{pmatrix}
\end{equation*}
and set $X_j(e )$ to be the matrix formed from $X(e)$ by deleting the column $j$,
then $y$ is a scene that lifts $x$ if and only if 
\begin{equation}\label{eq:scenelinear}
    \sum_{j=1}^d \left((-1)^j\det (X_j(e )) \right)  z_{v_j} = 0 \qquad \text{for each $e=\{v_1,\ldots,v_d\} \in E$}.
\end{equation}
We observe that the left hand side of each equation in \eqref{eq:scenelinear} is a linear polynomial that is independent of our choice of ordering of $e$ (up to multiplication of the linear polynomial by $-1$).

Now let $H=([n],E)$ be a $d$-uniform hypergraph $H=(V,E)$ with hyperedges $e_1,\ldots,e_m$.
We now consider each hyperedge to be an ordered $d$-tuple $(v_1,\ldots,v_d)$ with $v_1 < v_2 < \cdots < v_d$.
Given a picture $x$ of $H$,
let $A(H,x)$ denote the $|E| \times n$-coefficient matrix of the system of linear equations given in \eqref{eq:scenelinear}. 
Explicitly, $A(H,x)$ is defined by
\begin{equation*}
A(H,x)_{i,j} =
\begin{cases}
(-1)^k \, \det(X_k(e_{i})) & \text{if } j = v_k \in e_{i} = (v_1,\ldots,v_d), \\
0 & \text{otherwise.}
\end{cases}
\end{equation*}
As noted above, the affine space of lifts of $x$ are in one-to-one correspondence with the kernel of $A(H,x)$.
Thus $x$ has a non-trivial lift if and only if the kernel of $A(H,x)$ has dimension $d$ or greater.

Whiteley proved the following characterisation of when the rows of $A(H,x)$ are independent for a generic picture $x$:

\begin{theorem}[\cite{whiteley1989matroid}]\label{thm:indepA}
Let $H=(V,E)$ be a $d$-uniform hypergraph, and let $x$ be a generic $(d-1)$-picture of $H$. 
The rows of $A(H,x)$ are independent if and only if 
\[
|E'| \leq V(E')-(d-1)
\]
for all non-empty subsets $E' \subseteq E$.
\end{theorem}

\Cref{thm:indepA} directly implies the following characterisation of when generic pictures of $d$-uniform hypergraphs lift to sharp scenes, answering \Cref{q:scene_analysis} for generic pictures:

\begin{theorem}[The picture theorem for uniform hypergraphs \cite{whiteley1989matroid}]
    Let $H=(V,E)$ be a $d$-uniform hypergraph, and let $x$ be a generic picture of $H$.
    There is a sharp scene that lifts $x$ if and only if
    \[
    |E'| \leq V(E')-d
    \]
    for all subsets $E' \subseteq E$ with $|E'| \geq 2$.
    \label{thm:exists_sharp_scene}
\end{theorem}

\begin{example}
    Consider the hypergraph $H$ pictured in \Cref{figure: pictures and scenes} (b). It has $4$ hyperedges and $6$ vertices, so $|E|=|V|-2$. By \Cref{thm:exists_sharp_scene} generic pictures of $H$ (b) do not admit sharp scenes. Hence the picture in \Cref{figure: pictures and scenes} (b) has to be non-generic.
\end{example}

We extend \Cref{thm:exists_sharp_scene} to non-trivial lifts by introducing the following matroid.

\begin{definition}\label{def: sparse}
A $d$-uniform hypergraph $H=(V,E)$ is \emph{$(k,l)$-sparse} if 
\[
|E'| \leq k|V(E')|-l
\]
for all non-empty subsets of edges $E' \subseteq E$. We say that $H$ is \emph{$(k,l)$-tight} if $H$ is $(k,l)$-sparse, and $|E|=k|V|-l$. The $(k,l)$-sparse hypergraphs with vertex set $[n]$ are the independent sets of a matroid on $\binom{[n]}{k}$, which is non-empty provided that $0 \leq l \leq kd-1$ (see~\cite{whiteley1996some}). We refer to this matroid as the \emph{$(k,l)$-sparsity matroid}, and denote it by $\mathcal{S}_{(k,l)}$. 
\end{definition}

\begin{corollary}
    Let $H=(V,E)$ be a $d$-uniform hypergraph on $n$ vertices. Then generic pictures of $H$ have liftings to non-trivial scenes if and only if $\rank_{\mathcal{S}_{(1,d-1)}}(E) \leq n-d $.
    \label{cor:non-triv.scene}
\end{corollary}

\begin{proof}
    This is an easy corollary of \Cref{thm:indepA}, using that a generic picture $x$ of $H$ has a lifting to a non-trivial scene if and only if $\rank(A(H,x)) \leq n-d$.
\end{proof}

\begin{remark}
\Cref{thm:indepA} says that the rows of $A(H,x)$, where $H$ a $d$-uniform hypergraph and $x$ is a generic picture, are independent if and only if $H$ is $(1,d-1)$-sparse. The parameters $k=1$ and $l=d-1$ from \Cref{thm:indepA} are in the range where $(k,l)$-sparsity defines a matroid. The parameters $k=1$ and $l=d$ from \Cref{thm:exists_sharp_scene} are not, which is why the sparsity condition is required for all subsets of at least two edges. As a slight abuse of notation, we will also refer to the $d$-unform hypergraphs that satisfy the counting condition in \Cref{thm:exists_sharp_scene} as $(1,d)$-sparse, even though they are not the independent sets of a matroid.
\end{remark}

\section{Liftability of matroids}
\label{sec:liftabilitynonpaving}

The notion of {\em liftability} of matroids, has already appeared and been studied in several other papers, see \cite{clarke2024liftable,Fatemeh4}. It has been studied only in the case of paving matroids, for which it has been proven to be a useful tool to construct polynomials within the ideals of matroid varieties of paving matroids, see \cite{Fatemeh4}. Liftability provides a systematic way of constructing such polynomials, called {\em liftability polynomials}. In this section, we generalise these notion to matroids with realisable truncations.

\subsection{Circuit matrices}\label{subsec:circuitmatrices}

Let $N$ be a realisable matroid of rank $d-1$ on $[n]$, and let $\vf = (v_1, \ldots, v_n)$ be a tuple of vectors in $\CC^{d-1}$ realising the matroid $N$. 
For each circuit $c$ of $N$, let $r_c = (r_{c,1}, \ldots,r_{c,n}) \in \CC^n$ be the unique vector, up to a nonzero scalar, satisfying
\begin{equation}\label{eq:rcrm}
\operatorname{supp}(r_c) = c \qquad \text{and} \qquad \sum_{i \in c} r_{c,i} v_i = \mathbf{0}.
\end{equation}
For a set $S$ of circuits of $N$,
we define the {\em circuit matrix} $\mathcal{R}_{\vf}(N;S)$ associated with $N$ and $\vf$ to be the matrix whose columns and rows are indexed by $[n]$ and $S$ respectively, where where the row indexed by the circuit $c \in S$ is the row vector $r_c$.  

If we consider $\vf$ as a $(d-1) \times n$-matrix,
then, by construction, we have that $\mathcal{R}_{\vf}(N;S) \vf^\top = \mathbf{0}_{|S| \times (d-1)}$.
Since our choice of $\vf$ will always be a realisation of $N$ (and hence $\vf$ has rank $d-1$ when considered as a matrix),
it follows that $\rank \mathcal{R}_{\vf}(N;S) \leq n - d+1$.

\begin{lemma}\label{lem: rmatrix kernel and liftings trivial}
Let $N$ be a rank-$(d-1)$ matroid with ground set $[n]$, and let $\vf \in \Gamma_{N}$.
Then the linear space
\begin{equation}\label{eq: rmatrix kernel and liftings trivial}
    \Big\{ ~ z = (z_1,\ldots,z_n) ~ : ~ (v_1,z_1),\ldots,(v_n,z_n) \text{ are linearly dependent}  ~ \Big\}
\end{equation}
is $(d-1)$-dimensional and contained in $\ker \mathcal{R}_{\vf}(N;S)$.
\end{lemma}

\begin{proof}
    The vectors contained in the linear space described in \eqref{eq: rmatrix kernel and liftings trivial} are exactly the vectors contained in the row space of $\vf$ when the latter is considered to be a $(d-1) \times n$ matrix.
    Thus linear space described in \eqref{eq: rmatrix kernel and liftings trivial} is $(d-1)$-dimensional and contained in $\ker \mathcal{R}_{\vf}(N;S)$.
\end{proof}

We showcase the geometric properties of row matroids for circuit matrices with the following two results.

\begin{lemma}\label{lem:cont}
    Let $N$ be a realisable rank-$(d-1)$ matroid with ground set $[n]$.
    Then, given the projective variety
    \begin{equation*}
        \mathbb{P}(\Gamma_{N}) := \left\{ [\vf] := \big([v_1], \ldots, [v_n] \big)_{i \in [n]} ~ : ~ \mathbf{v} = (v_1,\ldots,v_n) \in \Gamma_{N}  \right\} \subset \prod_{i \in [n]} \mathbb{P}(\mathbb{C}^{d-1}),
    \end{equation*}
    and a set $S$ of circuits of $N$,
    the map
    \begin{equation*}
        f_S : \mathbb{P}(\Gamma_{N}) \longrightarrow \prod_{c \in S} \mathbb{P}(\mathbb{C}^{n}), ~ [\mathbf{v}] \longmapsto \big([r_c(\mathbf{v})] \big)_{c \in S},
    \end{equation*}
    where each $r_c(\mathbf{v})$ is the unique-up-to-scalar vector satisfying \eqref{eq:rcrm},
    is continuous with respect to the Zariski topology.
    Moreover, the preimage of any quasi-projective set defined by rational-coefficient polynomials is a quasi-projective set defined by rational-coefficient polynomials.
\end{lemma}
\begin{proof}
    It suffices to show that for a circuit $c$, we can choose a vector $r_c(\mathbf{v})$ for each $\mathbf{v} \in \Gamma_{N}$ so that for all $\mathbf{w} \in \Gamma_{N}$ the map $\mathbf{v} \mapsto r_c(\mathbf{v})$ is continuous on a Zariski-open neighbourhood of $\mathbf{w}$ in $\Gamma_{N}$.

    Fix $\mathbf{w} \in \Gamma_{N}$.
    By reordering the elements of $n$,
    we may suppose that $c= [k]$ for some $k \leq d$.
    Fix the matrix $A_{\mathbf{v}} = ( v_1 | \cdots | v_{k-1})$ for each $\mathbf{v} \in \Gamma_{N}$.
    As $c$ is a circuit with rank $k-1$,
    each matrix $A_{\mathbf{v}}$ has rank $k-1$.
    By reordering rows, we may suppose that the first $k-1$ rows of $A_{\mathbf{w}}$ are linearly independent.
    With this, fix $B_{\mathbf{v}}$ to be the submatrix formed from $A_{\mathbf{v}}$ by taking only the first $k-1$ rows.
    Here we note that there exists a Zariski-open subset $U \subset \Gamma_{N}$ containing $\mathbf{w}$ for which $B_{\mathbf{v}}$ is invertible for all $\mathbf{v}$.

    Choose any $\mathbf{v} \in U$.
    As $B_{\mathbf{v}}$ is invertible, there exists a unique solution $x = x_{\mathbf{v}}$ to the equation $B_{\mathbf{v}}x = -(v_{k,i})_{i \in [k-1]}$.
    We observe here that $x_{\mathbf{v}}$ can construct $r_c(\mathbf{v})$ as follows:
    \begin{equation*}
        r_c(\mathbf{v})_i =
        \begin{cases}
            x_{\mathbf{v},i} &\text{if } i \in [k-1],\\
            1 &\text{if } i=k,\\
            0 &\text{otherwise.}
        \end{cases}
    \end{equation*}
    If we set $B_{\mathbf{v}}^{(i)}$ to be the matrix formed by replacing column $i$ with the vector $-(v_{k,i})_{i \in [k-1]}$,
    then Cramer's rule implies that
    \begin{equation*}
        x_{\mathbf{v}} = \frac{1}{\det B_{\mathbf{v}}} \Big( \det B_{\mathbf{v}}^{(1)} ~ ,~ \ldots ~,~ B_{\mathbf{v}}^{(k-1)} \Big)
    \end{equation*}
    Since $B_{\mathbf{v}}$ is non-singular on $U$,
    the map $\mathbf{v} \mapsto x_{\mathbf{v}}$, $\mathbf{v} \in U$, is continuous in the Zariski topology. This now gives rise to our desired continuous map on a Zariski-open neighbourhood of $\mathbf{v}$.
    Moreover, as the coordinates of $x_\vf$ are defined using rational-coefficient polynomials,
    the preimage of any quasi-projective set defined by rational-coefficient polynomials is a quasi-projective set defined by rational-coefficient polynomials.
\end{proof}

\begin{lemma}\label{lem:genericliftmatroid}
    Let $N$ be a realisable rank-$(d-1)$ matroid with ground set $[n]$.
    If $\mathbf{v}, \mathbf{w}$ are realisations of $N$ in the same irreducible component of $\Gamma_{N}$ and $\mathbf{w}$ is generic,
    then for any set of circuits $S$ of $N$, the row matroid of $\mathcal{R}_{\vf}(N;S)$ is less than or equal to the row matroid of $\mathcal{R}_{\mathbf{w}}(N;S)$ in the weak order of matroids.
\end{lemma}

\begin{proof}
    For an $m$-tuple $\mathbf{x} = ([x_i])_{i \in [m]}$ with $[x_i] \in \mathbb{P}(\mathbb{C}^{k})$,
    we fix $\mathcal{M}(\mathbf{x})$ to be the matroid generated by the linear dependencies of the projective vectors $[x_1],\ldots,[x_m]$.
    With this, we have that the row matroid of $\mathcal{R}_{\mathbf{w}}(N;S)$ is the same as the matroid $\mathcal{M}\big([r_c(\mathbf{w})] :c \in S \big)$ for any $\mathbf{w} \in \Gamma_N$.
    
    Choose any realisation $\vf$ of $N$, and fix $V$ to be the irreducible component of $\Gamma_M$ that contains $\vf$.
    Now choose any independent set $I = \{c_1,\ldots,c_k\} \subset S$ in the row matroid of $\mathcal{R}_{\vf}(N;S)$.
    Then there exists a subset $J = \{j_1,\ldots,j_k\} \subset [n]$ for which the determinant of the submatrix of $\mathcal{R}_{\vf}(N;S)$ with rows $I$ and columns $J$ is non-zero.
    By \Cref{lem:cont},
    the set of realisations $\mathbf{w} \in V$ for which
    \begin{equation*}
        \det 
        \begin{pmatrix}
            r_{c_1}(\mathbf{w})_{j_1} & \cdots & r_{c_1}(\mathbf{w})_{j_k}\\
            \vdots & \ddots & \vdots\\
            r_{c_{k}}(\mathbf{w})_{j_1} & \cdots & r_{c_k}(\mathbf{w})_{j_k}
        \end{pmatrix}
        = 0
    \end{equation*}
    forms a Zariski closed subset of $V$ that is defined by rational-coefficient polynomials.
    As the row matroid of $\mathcal{R}_{\mathbf{w}}(N;S)$ is the same as the matroid $\mathcal{M}\big([r_c(\mathbf{\mathbf{w}})] :c \in S \big)$ for any $\mathbf{w} \in \Gamma_N$, 
    it follows that $I$ is an independent of the row matroid of $\mathcal{R}_{\mathbf{w}}(N;S)$ for generic $\mathbf{w} \in V$. 
\end{proof}

\begin{remark}\label{rem:rowvcolmatroid}
    If we choose $S$ to be the set of all circuits of a matroid $N$ (or indeed a set for which $\rank \mathcal{R}_\vf(N;S) = n-d$),
    then the column matroid of $\mathcal{R}_\vf(N;S)$ is exactly the dual matroid $N^*$;
    in particular, $\rank \mathcal{R}_\vf(N;S) = \rank N^*$.
\end{remark}

\subsection{Liftability matrices and matroids}

We begin with the following observation: if $c$ is a non-spanning circuit of a matroid $M$,
then it is also a circuit (now possibly spanning) of the truncation $T(M)$.
This simple observation implies that the following is well-defined.

\begin{definition}\label{def: liftability matrix general}
    Let $M$ be a rank-$d$ matroid on $[n]$.
    Suppose that the truncation of $M$ has a realisation $\vf \in \Gamma_{T(M)}$.
    We define the {\em liftability matrix} associated with $M$ and $\vf$ to be the matrix $\mathcal{L}_\vf(M) := \mathcal{R}_\vf(T(M);\mathcal{C}^{\mathrm{ns}}(M))$.
\end{definition}

The following proposition establishes a connection between the liftability matrix and the space of liftings of a given tuple of vectors.

\begin{proposition}\label{prop: kernel and liftings 3}
Let $M$ be a rank-$d$ matroid with ground set $[n]$.
Then the following equivalence holds for any realisation $\vf \in \Gamma_{T(M)}$:
\[
(z_{1}, \ldots, z_{n}) \in \ker \mathcal{L}_{\vf}(M)
\quad \Longleftrightarrow \quad
\widetilde{\mathbf{v}} = \big( (v_{1},z_1), \ldots, (v_{n} , z_{n}) \big) \in V_{\Ccal(M)}.
\]
\end{proposition}

\begin{proof}
    For each $c \in \Ccal^{\mathrm{ns}}(M)$,
    choose $r_c = (r_{c,1},\ldots,r_{c,n}) \in \CC^n$ satisfying the properties described in \eqref{eq:rcrm}.
    Since $\vf$ is a realisation of $T(M)$,
    the restriction of $r_c$ to indices in $c$ describes the unique (up to scaling) non-zero linear dependence between the points $(v_i)_{i \in c}$ for each non-spanning circuit $c$ of $M$.
    Thus for each $c \in \Ccal^{\mathrm{ns}}(M)$,
    the following equivalent conditions hold:
    \begin{equation}\label{eq:equiv kernel}
        (v_i,z_i)_{i \in c} \, \, \text{ linearly dependent}\quad  \iff \quad \sum_{i\in c} r_{c,i} (v_i,z_i) = \mathbf{0} \quad \iff \quad \sum_{i\in c} r_{c,i} z_i = 0.
    \end{equation}
    The left hand side of \eqref{eq:equiv kernel} holds for all $c \in \Ccal^{\mathrm{ns}}(M)$ if and only if $\widetilde{\vf} \in V_{\mathcal{C}(M)}$,
    since the vectors $(v_i,z_i)_{i \in c}$ are automatically linearly dependent when $|c| = d+1$ (i.e.~$c$ is spanning).
    The right hand side of \eqref{eq:equiv kernel} holds for all $c \in \Ccal^{\mathrm{ns}}(M)$ if and only if $z \in \ker \mathcal{L}_\vf(M)$.
    This completes the proof.
\end{proof}

\begin{corollary}\label{cor:realisationliftable}
    Let $M$ be a rank-$d$ matroid with ground set $[n]$, and let $\vf \in \Gamma_{T(M)}$.
    Then $\rank \mathcal{L}_{\vf}(M) \leq n-d + 1$.
    Moreover, $\vf$ has a non-trivial lifting if and only if $\rank \mathcal{L}_{\vf}(M) \leq n-d$.
\end{corollary}

\begin{proof}
    Since the vectors $z$ that induce a trivial lifting of $\vf$ form a $(d-1)$-dimensional subspace of $\ker \mathcal{L}_{\vf}(M)$ by \Cref{lem: rmatrix kernel and liftings trivial}, it follows that $\rank \mathcal{L}_{\vf}(M) \leq n-d + 1$.
    It now follows from \Cref{prop: kernel and liftings 3} that $\vf$ has a non-trivial lifting if and only if $\rank \mathcal{L}_{\vf}(M) \leq n-d$.
\end{proof}

By \Cref{lem:genericliftmatroid},
if $\vf$ and $\mathbf{w}$ are both generic realisations of $T(M)$ in the same irreducible component of $V_{T(M)}$, then the row matroids of $\mathcal{L}_\vf(M)$ and $\mathcal{L}_{\mathbf{w}}(M)$ coincide.
This now motivates the following definition.

\begin{definition}\label{def: generic liftability matroids}
Let $M$ be a matroid with a realisable truncation. A \emph{generic liftability matroid} of $M$ is defined as the row matroid of the liftability matrix $\mathcal{L}_{\mathbf{v}}(M)$ corresponding to a generic point $\mathbf{v}$ of one of the irreducible components $\Gamma_{T(M)}$.
\end{definition}

Many matroids have a unique generic liftability matroid.
For example, if $\Gamma_{T(M)}$ is irreducible (e.g., $T(M)$ is uniform), then there exists a unique generic liftability matroid.
It is currently open whether any matroid has more than one generic liftability matroid.

\begin{conjecture}\label{conj:genlift}
    Every realisable matroid has a unique generic liftability matroid.
\end{conjecture}

We are now ready to state the following result.

\begin{theorem}\label{thm:genericlifting}
    Let $M$ be a rank-$d$ matroid with ground set $[n]$. If $T(M)$ is realisable, then $M$ is liftable if and only if every generic liftability matroid has rank at most $n-d$.
\end{theorem}

\begin{proof}
    If $M$ is liftable then $\rank \mathcal{L}_{\mathbf{v}}(M) \leq n-d$ for all $\mathbf{v}$ by \Cref{cor:realisationliftable},
    which implies each generic liftability matroid has rank at most $n-d$.
    Suppose instead that every generic liftability matroid has rank at most $n-d$.
    Choose any $\mathbf{v} \in \Gamma_{T(M)}$. 
    Now choose any realisation $\mathbf{w}$ that is a generic point in the same irreducible component of $\Gamma_{T(M)}$ as $\mathbf{v}$.
    Then by \Cref{lem:genericliftmatroid} we have 
    \begin{equation*}
        \rank \mathcal{L}_{\mathbf{v}}(M) \leq \rank \mathcal{L}_{\mathbf{w}}(M) \leq n-d.
    \end{equation*}
    It now follows from \Cref{cor:realisationliftable} that $M$ is liftable.
\end{proof}

\subsection{Realisability vs liftability}\label{subsec:rvl}

We first observe that realisability and liftability do not imply each other.

\begin{figure}[H]
    \centering
    \includegraphics[width=1\textwidth, trim=0 0 0 0, clip]{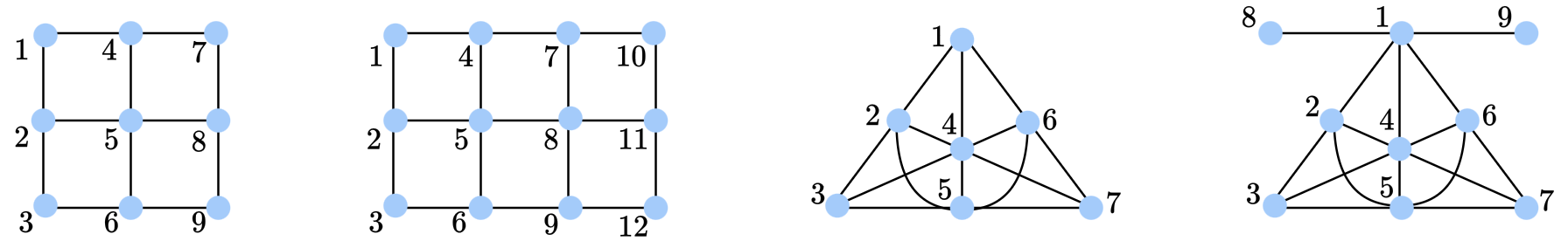}
    \caption{(a) $3\times 3$ grid; (b) $3\times 4$ grid; (c) Fano plane; (d) Fano plane with extra line.}
    \label{Figure1}
\end{figure}

\begin{example}
Consider the four rank-$3$ paving matroids depicted in \Cref{Figure1}, with points representing ground set elements and lines representing circuits. 

It is easy to check that any matroid stemming from a grid (e.g., (a) and (b) of \Cref{Figure1}) is realisable.
As the Fano plane is not realisable, neither (c) nor (d) of \Cref{Figure1} are realisable.
The liftability or non-liftability of these matroids can be determined using \Cref{thm:D}. 
 
More precisely:
\begin{itemize}
\item The matroid in \Cref{Figure1} (a) is realisable and liftable.
\item The matroid in \Cref{Figure1} (b) is realisable but non-liftable.
\item The matroid in \Cref{Figure1} (c) is non-realisable and non-liftable.
\item The matroid in \Cref{Figure1} (d) is liftable but non-realisable.
\end{itemize}
\end{example}

Fortunately, generic realisability (as introduced in \Cref{def:genericallyrealisable}) implies both realisability and liftability.

\begin{proposition}\label{prop:genrealimplieslift}
    Every generically realisable matroid is liftable.
\end{proposition}

\begin{proof}
    Let $M$ be a generically realisable matroid,
    and let $\vf$ be a generic realisation of $T(M)$.
    As $M$ is generically realisable,
    there exists a lifting $\widetilde{\vf}$ of $\vf$ for which $\mathcal{M}(\widetilde{\vf}) = M$,
    and thus $\widetilde{\vf} \in V_{\mathcal{C}(M)}$.
    As this holds for arbitrary generic realisations of $T(M)$,
    the result now follows from \Cref{thm:genericlifting}.
\end{proof}

It is tempting to believe that any matroid which is both realisable and liftable is also generically realisable.
However, this is not the case.

\begin{figure}[H]
    \centering
    \includegraphics[scale=0.5]{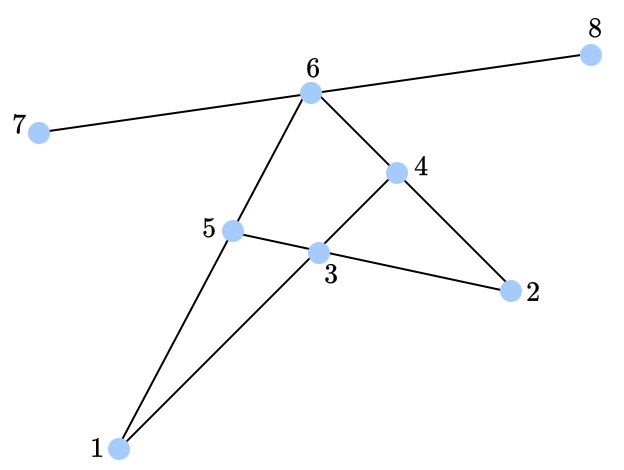}
    \caption{A matroid that is realisable and liftable but not generically realisable.}
    \label{Figure2}
\end{figure}

\begin{example}
We now present an example of a matroid which is liftable and realisable, but not generically realisable. 
Let $M$ be a matroid on $[8]$ depicted in \Cref{Figure2}. A realisation of $M$ can be seen in \Cref{Figure2}, so $M$ is realisable. 
For any realisation $\vf$ of $T(M)$, 
the matrix $\mathcal{L}_{\vf}(M)$ has size $5\times 8$, and hence its kernel is of dimension at least $3$. 
Thus, $M$ is liftable.

To show that $M$ is not generically realisable, we first note that the restriction of $M$ to the elements $\{1,2,3,4,5,6\}$ is the graphic matroid $M(K_4)$.
However, the matroid $M(K_4)$ is not liftable by \Cref{thm:D}. Alternatively, the non-liftability of $M(K_4)$ can be deduced from \cite[\S 4]{clarke2024liftable}. Hence, in any lifting of a generic realisation $\vf \in \Gamma_{T(M)}$, the points $\{1,2,3,4,5,6\}$ lie on a common line, so such a lifting cannot be a realisation of $M$.
\end{example}

\section{Lattice properties of liftability and realisability}
\label{sec:matroidlattice}

Instead of starting with a matroid $M$, passing to its truncation $T(M)$, and then lifting realisations of $T(M)$ to points of the circuit variety of $M$, it is natural to reverse this approach. In this section, we instead fix a matroid $N$ and ask how its realisations can be lifted to realisations of \emph{some} matroid $M$ satisfying $T(M)=N$. We focus on the elevation lattice, which provides a natural framework to study such liftings.

\subsection{The 1-elevation lattice and its subsets}

We begin with the following result of Crapo.

\begin{theorem}[Crapo \cite{crapo1970erecting}]\label{thm:crapo}
    For any matroid $N$ on ground set $[n]$,
    the set
    \begin{equation*}
        \Lat(N) := \{N\} \cup \{M \text{ matroid on } [n] : T(M)=N \}
    \end{equation*}
    forms a lattice with respect to the weak ordering of matroids.
\end{theorem}

We say that every element of $\Lat(N)$ is a \emph{1-elevation} of $N$ (with $N$ itself being the \emph{trivial 1-elevation}). We call $\Lat(N)$ the \emph{1-elevation lattice} of $N$. 

Let $N$ be a realisable rank-$(d-1)$ matroid on $[n]$ with realisation $\mathbf{v}$.
Given $N$ and $\mathbf{v}$, we define the following two sets of matroids.
\begin{enumerate}[label=(\roman*)]
    \item $L_{\mathbf{v}} \subseteq \Lat(N)$ is the collection of matroids where for each $M \in L_{\mathbf{v}}$ there exists a non-trivial lifting $\widetilde{\mathbf{v}} \in V_{\mathcal{C}(M)}$ of $\mathbf{v}$.
    \item $R_{\mathbf{v}} \subseteq \Lat(N)$ is the collection of matroids where for each $M \in R_{\mathbf{v}}$ there exists a non-trivial lifting $\widetilde{\mathbf{v}} \in \Gamma_{M}$ of $\mathbf{v}$.
\end{enumerate}
By \Cref{thm:genericlifting},
$L_{\vf} = L_{\mathbf{w}}$ for any two generic realisations $\vf , \mathbf{w}$ of $T(M)$ that are contained in the same irreducible component.
We will later prove that $R_{\vf} = R_{\mathbf{w}}$ also for any two generic realisations $\vf , \mathbf{w}$ of $T(M)$ that are contained in the same irreducible component;
see \Cref{cor:rv_is_generic}.

To better understand the collections $L_{\vf}$ and $R_{\vf}$, we define the \emph{elevation matrix} $\mathcal{E}_{\mathbf{v}}(N) := \mathcal{R}_{\mathbf{v}}(N;\Ccal^{\text{ns}}(N))$ for any realisation $\mathbf{v}$ of a matroid $N$ which is not uniform. Note that $\Ccal^{\text{ns}}(N)$ is non-empty if and only if $N$ is not uniform, which is why we only define $\mathcal{E}_{\mathbf{v}}(N)$ for matroids which are not uniform. If $N$ is the uniform $U_{d,n}$, we set $\mathcal{R}_{\mathbf{v}}(N;\Ccal^{\text{ns}}(N))=[0,...,0] \in \mathbb{R}^{n}$.

\begin{lemma}\label{lem:elevation}
    Let $N$ be a rank-$(d-1)$ matroid on $[n]$, and suppose that $\mathbf{v}=(v_1,\ldots,v_n)$ is a realisation of $N$.
    Then the following properties are equivalent for any 1-elevation $M$ of $N$:
    \begin{enumerate}[label=(\roman*)]
        \item \label{lem:elevation:item1} $M \in R_{\mathbf{v}}$;
        \item \label{lem:elevation:item2} there exists $z \in \ker \mathcal{E}_{\mathbf{v}}(N)$ such that $\mathcal{M}(\widetilde{\mathbf{v}}) = M$, where $\widetilde{\mathbf{v}}=\big((v_1,z_1),\ldots,(v_n,z_n) \big)$.
    \end{enumerate}
\end{lemma}

\begin{proof}
    Note that the statement is trivial for uniform matroids, so suppose that $M$ is not uniform.
    \ref{lem:elevation:item2} implies \ref{lem:elevation:item1} by definition.
    Suppose \ref{lem:elevation:item1} holds for realisation $\widetilde{\mathbf{v}}=\big((v_1,z_1),\ldots,(v_n,z_n) \big)$ of $M$.
    For each $c \in \Ccal^{\mathrm{ns}}(N)$,
    choose $r_c = (r_{c,1},\ldots,r_{c,n}) \in \CC^n$ satisfying the properties described in \eqref{eq:rcrm}.
    Since $\vf$ is a realisation of $N$,
    the restriction of $r_c$ to indices in $c$ describes the unique (up to scaling) non-zero linear dependence between the points $(v_i)_{i \in c}$ for each non-spanning circuit $c$ of $N$.
    Thus for each $c \in \Ccal^{\mathrm{ns}}(N)$,
    we have
    \begin{equation}\label{eq:equiv kernel 2}
        (v_i,z_i)_{i \in c} \, \, \text{ linearly dependent}\quad  \iff \quad \sum_{i\in c} r_i (v_i,z_i) = \mathbf{0} \quad \iff \quad \sum_{i\in c} r_i z_i = 0.
    \end{equation}
    The left hand side of \eqref{eq:equiv kernel 2} holds for all $c \in \Ccal^{\mathrm{ns}}(N)$ as $\widetilde{\mathbf{v}} \in \Gamma_M$.
    Thus right hand side of \eqref{eq:equiv kernel 2} holds for all $c \in \Ccal^{\mathrm{ns}}(N)$,
    which is equivalent to $z \in \ker \mathcal{E}_\vf(N)$.
    This completes the proof.
\end{proof}

\begin{corollary}\label{cor: lifting from N}
    Let $N$ be a rank-$(d-1)$ matroid on $[n]$ which is not uniform, and let $\vf \in \Gamma_{N}$.
    Then $\rank \mathcal{E}_{\vf}(N) \leq n-d + 1$.
    Moreover, $\vf$ has a non-trivial lifting $\widetilde{\vf}$ realising a matroid $M$ such that $T(M)=N$ if and only if $\rank \mathcal{E}_{\vf}(N) \leq n-d$.
\end{corollary}

\begin{proof}
    By \Cref{lem: rmatrix kernel and liftings trivial}, $\rank \mathcal{E}_{\vf}(N) \leq n-d + 1$, since the vectors $z$ that induce a trivial lifting of $\vf$ form a $(d-1)$-dimensional subspace of $\ker \mathcal{E}_{\vf}(N)$.
    It now follows from \Cref{lem:elevation} that $\vf$ has a non-trivial lifting $\widetilde{\mathbf{v}} \in \Gamma_{M}$ of $\mathbf{v}$, for some matroid $M$ such that $T(M)=N$, if and only if $\rank \mathcal{E}_{\vf}(N) \leq n-d$.
\end{proof}

It is immediate that the set of liftable matroids in $\Lat(N)$ is exactly the set $\bigcap_{\vf \in \Gamma_N} L_\vf$ and the set of generically realisable matroids in $\Lat(N)$ is exactly the set $\bigcap_{\vf \in \Gamma_N \text{ generic}} R_\vf$.
To describe the realisable matroids in $\Lat(N)$,
we first require the following lemma.

\begin{lemma}\label{lem:projweakorder}
    Let $\widetilde{\mathbf{v}}$ be a realisation of a rank-$d$ matroid $M$ on $[n]$, let $A \in \mathbb{C}^{(d-1)\times d}$ be any matrix, and let $\mathbf{v} = (A\widetilde{v}_i)_{i \in [n]}$.
    Then $\mathcal{M}(\mathbf{v}) \leq T(M)$,
    with equality for almost all matrices $A \in \mathbb{C}^{(d-1)\times d}$.
\end{lemma}

\begin{proof}
    Choose any independent set $I$ of $\mathcal{M}(\mathbf{v})$.
    Then by definition, the vectors $\{A\widetilde{v}_i: i\in I\}$ are linearly independent, and thus the vectors $\{\widetilde{v}_i: i\in I\}$ are independent.
    Since $M=\mathcal{M}(\widetilde{\mathbf{v}})$,
    $I$ is an independent set of $M$.
    As $\rank A \leq d-1$, it follows that $\size{I}\leq d-1$.
    Given that $I$ is independent in $M$ and $\size{I}\leq d-1$, it follows that $I$ is an independent set for $T(M)$ as desired.
    
    The equality between the matroids $\mathcal{M}(\mathbf{v})$ and $T(M)$ for almost all matrices $A$ now follows from the observation that, for every subset $I\subseteq [n]$ with $\size{I}\leq d-1$, the vectors $\{A\widetilde{v}_i: i\in I\}$ are linearly independent if and only if the vectors $\{\widetilde{v}_i: i\in I\}$ are linearly independent.
\end{proof}

Since the truncation of a realisable matroid is realisable,
it follows that there are no realisable matroids in $\Lat(N)$ if $N$ is not itself realisable. Given a realisable matroid $N$, the following result describes the realisable matroids in $\Lat(N)$ exactly.

\begin{proposition}
    Let $N$ be a realisable rank-$(d-1)$ matroid on $[n]$.
    Then the set of realisable matroids in $\Lat(N)$ is exactly the set $\bigcup_{\vf \in \Gamma_N} R_\vf$.
    \label{prop:R_v}
\end{proposition}

\begin{proof}
    If $M \in \bigcup_{\vf \in \Gamma_N} R_\vf$ then $M$ is realisable by definition of $R_\vf$.
    Suppose that $M$ is realisable with realisation $\mathbf{w} = (w_1,\ldots,w_n)$.
    By \Cref{lem:projweakorder},
    we can choose a matrix $A \in \mathbb{C}^{(d-1) \times d}$ so that the tuple $\mathbf{v}$ with $v_i = Aw_i$ for each $i \in [n]$ is a realisation of $N$.
    Now let $B \in \mathbb{C}^{d \times d}$ be any matrix for which $AB = (I_{d-1} | \mathbf{0})$ (the existence of $B$ stemming from $A$ having rank $d-1$),
    and fix $\widetilde{\mathbf{v}} = (B^{-1}w_i)_{i \in [n]}$.
    Then $\mathcal{M}(\widetilde{\mathbf{v}}) = \mathcal{M}(\mathbf{w}) = M$ and
    $\widetilde{\mathbf{v}} = \big((v_i,z_i)\big)_{i \in [n]}$ for some $z=(z_1,\ldots,z_n) \in \mathbb{C}^n$.
    Thus $M \in R_{\vf}$ as required.
\end{proof}

 Since the truncation of a realisable matroid is realisable, it follows from \Cref{prop:R_v} that any realisable matroid $M$ belongs to $R_{\vf}$ for some choice of realisation $\vf \in \Gamma_{T(M)}$.

We recall the following terminology for a finite lattice $(L,\prec)$. The \emph{join} of two elements $a,b$ in a lattice is the unique minimal element $c$ satisfying $a \preceq c$ and $b \preceq c$.
A subset $S \subset L$ is said to be \emph{join-closed} if the join of any two elements in $S$ is contained in $S$.
A subset $S \subset L$ is said to be \emph{upwards-closed} if $a \in S$ and $a \preceq b$ implies $b \in S$.
Every upwards-closed set is join-closed, but the converse is not necessarily true.

\begin{lemma}\label{lem:realisablelift}
    Let $N$ be a realisable rank-$(d-1)$ matroid on $[n]$ with realisation $\mathbf{v}$.
    Then $L_{\mathbf{v}}$ is an upwards-closed subset of $\Lat(N)$.
    Moreover, any minimal element of $L_{\mathbf{v}}$ is realisable (and hence is contained in $R_{\mathbf{v}}$).
\end{lemma}

\begin{proof}
    Choose any $M \in L_{\mathbf{v}}$, and let $\widetilde{\mathbf{v}} \in V_{\mathcal{C}(M)}$ be a non-trivial lifting of $\mathbf{v}$.
    Then by \Cref{lem:circvar2},
    $\widetilde{\mathbf{v}} \in V_{\Ccal (M')}$ for all $M' \geq M$.
    Hence $L_{\mathbf{v}}$ is upwards-closed.
    
    Now additionally assume that $M$ is minimal in $L_{\mathbf{v}}$.
    The matroid $\mathcal{M}(\widetilde{\mathbf{v}})$ is realisable and distinct from $T(M)$.
    As the projection of $\widetilde{\mathbf{v}}$ is $\mathbf{v}$, we have that $T(\mathcal{M}(\widetilde{\mathbf{v}})) \geq N$ by \Cref{lem:projweakorder}.
    By \Cref{lem:circvar1},
    $\mathcal{M}(\widetilde{\mathbf{v}}) \leq M$.
    By comparing rank functions we see that $T(\mathcal{M}(\widetilde{\mathbf{v}})) \leq T(M) = N$,
    and thus $T(\mathcal{M}(\widetilde{\mathbf{v}})) = N$.
    Since $\mathcal{M}(\widetilde{\mathbf{v}})$ is contained in $L_{\mathbf{v}}$, we have now that $M = \mathcal{M}(\widetilde{\mathbf{v}})$ by the minimality of $M$.
    Hence $M$ is realisable.
\end{proof}

\begin{figure}[H]
    \centering
    \includegraphics[scale=0.5]{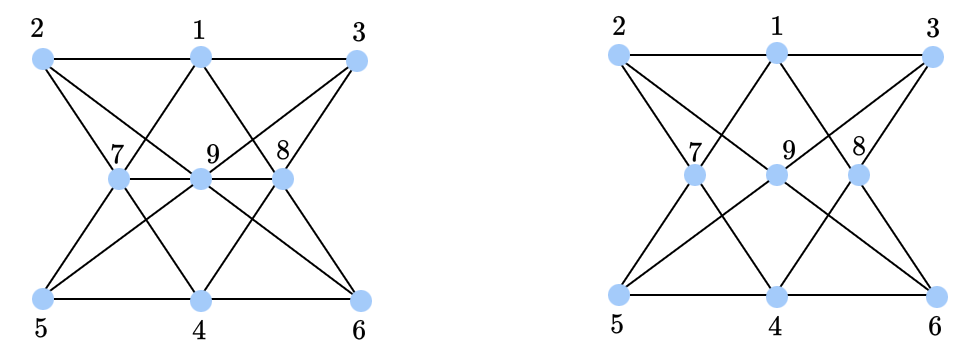}
    \caption{(a) Pappus matroid; (b) Non-Pappus matroid.}
    \label{Figure: Pappus-NonPappus}
\end{figure}

The set $R_{\mathbf{v}}$ is not upwards-closed in general. For instance, consider the Pappus matroid $M$ shown in~\Cref{Figure: Pappus-NonPappus} (a). Choose any realisation of $M$ and project it to obtain a tuple $\vf$ realising $T(M)$. Then, by construction, $M\in R_{\vf}$. On the other hand, the non-Pappus matroid $M'$ depicted in~\Cref{Figure: Pappus-NonPappus} (b), is not realisable, so it can not belong to $R_{\vf}$. Since $M<M'$, this shows that $R_{\mathbf{v}}$ is not upwards-closed.
However, the next lemma says that $R_{\mathbf{v}}$ is still somewhat well-behaved with respect to the weak order.

\begin{lemma}\label{lem:realisablereal}
    Let $N$ be a realisable rank-$(d-1)$ matroid on $[n]$ with realisation $\mathbf{v}$.
    Then $R_{\mathbf{v}}$ is join-closed.
\end{lemma}

\begin{proof}
    Choose any two matroids $M_1,M_2 \in R_{\mathbf{v}}$.
    Take two non-trivial liftings $\widetilde{\mathbf{v}}_1 = (\mathbf{v},z_1),\widetilde{\mathbf{v}}_2=(\mathbf{v},z_2)$ of $\mathbf{v}$ where $\mathcal{M}(\widetilde{\mathbf{v}}_1) = M_1$ and $\mathcal{M}(\widetilde{\mathbf{v}}_2)=M_2$.
    With this, we set $\mathbf{w}_t := 
    (\mathbf{v},tz_1 +(1-t)z_2)$ for each $t \in \mathbb{C}$.
    By \Cref{lem:elevation}, $z_1,z_2 \in \ker \mathcal{E}_{\mathbf{v}}(N)$, and thus $tz_1 +(1-t)z_2 \in \ker \mathcal{E}_{\mathbf{v}}(N)$.
    Hence by \Cref{lem:elevation},
    $\mathcal{M}(\mathbf{w}_t) \in R_{\mathbf{v}}$ for each $t \in \mathbb{C}$.

For each circuit $C$ of $N$ and $z=(z_1,\ldots,z_n)\in \CC^{n}$, the vectors $(v_i,z_i)_{i \in C}$ are linearly dependent if and only if $z\perp r_{C}$. Choose any circuit $C$ of $N$. If $C$ is a circuit in both $M_1$ and $M_2$, and since $\widetilde{\vf}_{1}$ and $\widetilde{\vf}_{2}$ realize these matroids, it follows that $z_1\perp r_{C}$ and $z_2\perp r_{C}$ and hence $tz_1+(1-t)z_2\perp r_{C}$ for every $t\in \CC$. Hence, $C$ is dependent in $\mathcal{M}(\mathbf{w}_t)$ for every $t\in \CC$. Now suppose that $C$ is a circuit of $N$ that is independent in $M_1$ or $M_2$. This implies that $z_1\not \perp r_{C}$ or $z_2\not \perp r_{C}$. Consequently, for all but finite choices of $t$ we have $tz_{1}+(1-t)z_{2}\not \perp r_c$, and for these values of $t$ we have that $C$ is independent in $\mathcal{M}(\mathbf{w}_t)$.
    
    Applying the above argument to each circuit of $N$, we see that we can choose $t^* \in \mathbb{C}$ so that $C$ is a circuit of $M:=\mathcal{M}(\mathbf{w}_{t^*})$ if and only if $C$ is a circuit of both $M_1$ and $M_2$.
    Hence $M$ is the join of $M_1,M_2$ and $M \in R_{\mathbf{v}}$ as required.
\end{proof}

We are now ready to prove \Cref{thm:A},
which we recall to be the following:

\begin{theorem}\label{thm: minimal liftables are generic realisable}
    Let $N$ be a realisable rank $(d-1)$-matroid on $[n]$.
    Then the set of liftable matroids in $\Lat(N)$ is upwards-closed,
    and the set of generically realisable matroids in $\Lat(N)$ is join-closed.
    
    Moreover, if $L_{\mathbf{v}} = L_{\mathbf{w}}$ for all generic $\mathbf{v},\mathbf{w} \in \Gamma_N$ (e.g., $V_N$ is irreducible),
    then the minimal liftable matroids in $\Lat(N)$ are generically realisable. 
\end{theorem} 

\begin{proof}
    Choose generic realisations $\mathbf{v}_1, \ldots,\mathbf{v}_k$ of $N$, one from each irreducible component of $\Gamma_N$.
    By \Cref{thm:genericlifting}, 
    the set of liftable matroids in $\Lat(N)$ is exactly the intersection $\bigcap_{i=1}^k L_{\mathbf{v}_i}$.
    Likewise, the set of generically realisable matroids in $\Lat(N)$ is $\bigcap_{i=1}^k R_{\mathbf{v}_i}$.
    As the intersection of upwards-closed sets is upwards-closed, the set of liftable matroids is upwards-closed by \Cref{lem:realisablelift}. 
    Similarly, as the intersection of join-closed sets is join-closed, the set of generically realisable matroids is join-closed by \Cref{lem:realisablereal}. 

    Now suppose that $L_{\mathbf{v}_i} = L_{\mathbf{v}_j}$ for all $i,j \in [k]$.
    Importantly, the minimal liftable matroids in $\Lat(N)$ are exactly the minimal elements of $L_{\mathbf{v}_1}$.
    Choose any minimal element $M$ of $L_{\mathbf{v}_1}$.
    By \Cref{lem:realisablelift},
    $M \in R_{\mathbf{v}_i}$ for each $i \in [k]$.
    Thus $M$ is generically realisable as is required.
\end{proof}

The authors believe that the requirement in \Cref{thm:A} that generic realisations provide the same subset $L_{\mathbf{v}}$ is always satisfied.
We observe here that \Cref{conj:genlift} being true would, in conjugation with \Cref{thm:A}, imply the following holds.

\begin{corollary}
    Let $N$ be a realisable matroid.
    If \Cref{conj:genlift} is true, then $L_{\mathbf{v}} = L_{\mathbf{w}}$ for all generic realisations $\vf,\mathbf{w}$ of $N$,
    and the minimal liftable matroids in $\Lat(N)$ are generically realisable.
\end{corollary}

We close the subsection with the following sufficient condition for generic realisability.

\begin{corollary}
\label{cor:unique_implies_realisable}
    Let $N$ be a realisable matroid on $[n]$. If $M$ is the unique 1-elevation of $N$, then $M$ is generically realisable if and only if $M$ is liftable.
\end{corollary}

\begin{proof}
    Suppose that $M$ is the unique liftable $1$-elevation of $N$. If $M$ is liftable, then in particular $M$ is a minimal liftable matroid of $\Lat(N)$. Hence $M$ is generically realisable, by \Cref{thm: minimal liftables are generic realisable}. On the other hand, if $M$ is not liftable, then $M$ is not generically realisable, by \Cref{prop:genrealimplieslift}.
\end{proof}

\begin{corollary}
\label{cor:elevation_matrix_characterisation}
    Let $N$ be a realisable matroid of rank $d-1$ on $[n]$ which is not uniform. Then the freest $1$-elevation of $N$ is liftable if and only if $\rank \mathcal{E}_{\vf}(N) \leq n-d$ for all generic $\vf \in \Gamma_{N}$.
\end{corollary}

\begin{proof}
    If $\rank \mathcal{E}_{\vf}(N) \leq n-d$, then some matroid in the $1$-elevation lattice of $N$ is liftable by \Cref{cor: lifting from N}. Since liftability is upwards-closed in the $1$-elevation lattice of $N$ by \Cref{thm: minimal liftables are generic realisable}, it follows that also the freest $1$-elevation is liftable. Conversely, if $\rank \mathcal{E}_{\vf}(N) = n-(d-1)$ then no $1$-elevation of $N$ is liftable. In particular, the freest $1$-elevation of $N$ is not liftable.
\end{proof}

\subsection{Generic realisability of \texorpdfstring{$M(K_5)$}{M(K5)} }
\label{ex:K5}

We now consider $M(K_5)$, the graphic matroid of the complete graph on $5$ elements. We let $N$ denote the truncation of $M(K_5)$, which is the rank-3 paving matroid on $E(K_5)$ with triangles being the circuits of rank two. As in \Cref{fig:K_5}, label the vertices $V(K_5)$ by $\{v_1,\ldots,v_5\}$ and label the edges $E(K_5)$ as follows:
\begin{equation}\label{labeling of vertices}
(\{v_1,v_2\},\{v_2,v_3\},\{v_3,v_4\},\{v_4,v_5\},\{v_1,v_5\},\{v_1,v_3\},\{v_3,v_5\},\{v_1,v_4\},\{v_2,v_4\},\{v_2,v_5\})=(1,2,\ldots,10).
\end{equation}
Under this labelling, $N$ is the rank-3 paving matroid on $\{1,\ldots,10\}$ with non-spanning circuits given by
\begin{equation*}
\mathcal{C}^{\text{ns}}(N)=
\{\{1,2,6\},\{3,4,7\},\{5,6,7\},\{2,7,10\},\{1,5,10\},\{2,3,9\},\{4,9,10\},\{3,6,8\},\{1,8,9\},\{4,5,8\}\}.
\end{equation*}

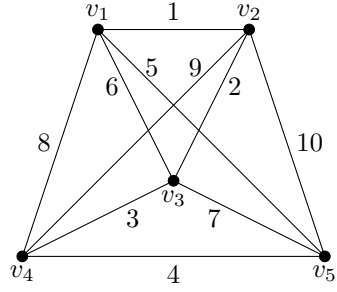
\begin{figure}[h]
\centering
\begin{tikzpicture}
\filldraw[black] (-2,-1) circle (2pt);
\filldraw[black] (0,0) circle (2pt);
\filldraw[black] (1,2) circle (2pt);
\filldraw[black] (2,-1) circle (2pt);
\filldraw[black] (-1,2) circle (2pt);

\draw (0,0)--(-2,-1);
\draw (0,0)--(1,2);
\draw (0,0)--(2,-1);
\draw (0,0)--(-1,2);
\draw (1,2)--(-2,-1);
\draw (2,-1)--(-2,-1);
\draw (-1,2)--(-2,-1);
\draw (1,2)--(2,-1);
\draw (-1,2)--(2,-1);
\draw (-1,2)--(1,2);

\node[anchor=south] at (-1,2) {$v_1$};
\node[anchor=south] at (1,2) {$v_2$};
\node[anchor=north] at (0,0) {$v_3$};
\node[anchor=north] at (-2,-1) {$v_4$};
\node[anchor=north] at (2,-1) {$v_5$};

\node[anchor=south] at (0,2) {$1$};
\node[anchor=east] at (-0.6,1.25) {$6$};
\node[anchor=east] at (-1.5,0.5) {$8$};
\node[anchor=west] at (-0.5,1.5) {$5$};

\node[anchor=west] at (0.6,1.25) {$2$};
\node[anchor=east] at (0.5,1.5) {$9$};
\node[anchor=west] at (1.5,0.5) {$10$};

\node[anchor=east] at (0.75,-0.5) {$7$};
\node[anchor=west] at (-0.75,-0.5) {$3$};

\node[anchor=north] at (0,-1) {$4$};
\end{tikzpicture}
\caption{$K_5$ with the labeling of the edges used in \Cref{ex:K5}}
\label{fig:K_5}
\end{figure}

Following \textup{\cite[Procedure~1]{Fatemeh6}}, we find that, up to projective transformation, every realisation of $N$ has the form:

\begin{equation}\label{matrix of N}
\begin{array}{ccccccccccccc}
&1&2&3&4&6& 7& 5& 10& 9 & 8 \\
\\
\ldelim({3}{0.5em} &1 & 0 & 0 & 1 & 1 & 1 & z & 1 & 0 & z-1 & \rdelim){3}{0.5em} \\
& 0 & 1 & 0 & 1 & x & 1 & x(z-1)+1 & x(z-1)+1 & x(z-1) & x(z-1)\\
& 0 & 0 & 1 & 1 & 0 & y & y & y & y-1 & y-1
\end{array}
\end{equation}

where the columns correspond to the elements $(1,2,3,4,6,7,5,10,9,8)$, and $x,y,z\in \CC$ are any scalars such that all minors associated with bases are nonzero. In particular, it follows from this that $\Gamma_N$ is irreducible. 

A generic realisation of $N$ can be obtained from choosing $x,y,z\in \CC$ algebraically independent over $\QQ$. If $\vf$ denotes the corresponding realisation of $N$ for such a choice of parameters, a direct computation shows that
\begin{equation}\label{ex:EK5}
\mathcal{E}_{\mathbf{v}}(N)=
\begin{array}{cccccccccccccccccccccccccc}
&1&2&3&4&6& 7& 5& 10& 9 & 8 \\
\\
\ldelim({10}{0.5em} & 1 & x & 0 & 0 & -1 & 0 & 0 & 0 & 0 & 0 & \rdelim){10}{0.5em}\\
&0 & 0 & y-1 & 1 & 0 & -1 & 0 & 0 & 0 & 0\\
&0 & 0 & 0 & 0 & z-1 & 1 & -1 & 0 & 0 & 0\\
&0 & x(z-1) & 0 & 0 & 0 & 1 & 0 & -1 & 0 & 0\\
&z-1 & 0 & 0 & 0 & 0 & 0 & -1 & 1 & 0 & 0\\
&0 & x(z-1) & y-1 & 0 & 0 & 0 & 0 & 0 & -1 & 0\\
&0 & 0 & 0 & 1 & 0 & 0 & 0 & -1 & 1 & 0\\
&0 & 0 & y-1 & 0 & z-1 & 0 & 0 & 0 & 0 & -1\\
&z-1 & 0 & 0 & 0 & 0 & 0 & 0 & 0 & 1 & -1\\
&0 & 0 & 0 & 1 & 0 & 0 & -1 & 0 & 0 & 1
\end{array},
\end{equation}
where the columns correspond to the elements $(1,2,3,4,6,7,5,10,9,8)$ and the rows correspond to the non-spanning circuits of $N$. 
Direct computation shows that the kernel of $\ker(\mathcal{E}_{\mathbf{v}}(N))$ is spanned by the rows of the matrix \eqref{matrix of N} and the vector $(0,0,0,1,0,1,1,1,0,0)$. Hence, $\dim(\ker(\mathcal{E}_{\mathbf{v}}(N))=4$, which implies that $\rank(\mathcal{E}_{\mathbf{v}}(N))=6$. 

$M(K_5)$ is the unique element in the $1$-elevation lattice of $N$ (see \Cref{app:kn} for a proof). Hence $M(\widetilde{\vf})= M(K_5)$ is generically realisable if and only if it is liftable, by \Cref{cor:unique_implies_realisable}. Since $\rank(\mathcal{E}_{\mathbf{v}}(N))=6=10-4$ and $M(K_5)$ is the freest $1$-elevation of $N$, $M(K_5)$ is liftable by \Cref{cor:elevation_matrix_characterisation}, and therefore generically realisable.

\subsection{Constructing \texorpdfstring{$R_{\vf}$}{Rv} with liftability matrices}

We first note the following simple necessary condition for when a matroid $M$ does not belong to $R_{\vf}$.

\begin{lemma}
    Let $N$ be a realisable matroid on $[n]$ with rank $d-1$. 
    Let $\mathbf{v}$ be a realisation of $N$, and suppose that $M_1 ,M_2$ are 1-elevations of $N$ with $M_1 < M_2$.
    If $\rank \mathcal{L}_{\mathbf{v}}(M_1) =  \rank \mathcal{L}_{\mathbf{v}}(M_2)$,
    then $M_2 \notin R_{\vf}$.
\label{lem:realisability_wo}
\end{lemma}

\begin{proof}
    As $M_1 < M_2$ and $M_1$ and $M_2$ are 1-elevations of $N$,
    we have that $\Ccal^{\text{ns}}(M_2) \subseteq \Ccal^{\text{ns}}(M_1)$.
    Since furthermore $\rank \mathcal{L}_{\mathbf{v}}(M_1) =  \rank \mathcal{L}_{\mathbf{v}}(M_2)$, it follows that $\ker \mathcal{L}_{\mathbf{v}}(M_1) = \ker \mathcal{L}_{\mathbf{v}}(M_2)$.
    Suppose for the sake of contradiction that there exists $z = (z_1,\ldots,z_n) \in \mathbb{C}^n$ for which $\widetilde{\vf} = \big( (v_i,z_i) \big)_{i \in [n]}$ is a realisation of $M_2$.
    As $\ker \mathcal{L}_{\mathbf{v}}(M_1) = \ker \mathcal{L}_{\mathbf{v}}(M_2)$,
    it follows that $\widetilde{\vf} \in V_{\mathcal{C}(M_1)}$.
    Thus $\mathcal{M}(\widetilde{\vf}) \leq M_1$ by \Cref{lem:circvar1}.
    However,
    this now implies $\mathcal{M}(\widetilde{\vf}) < M_2$,
    contradicting that $\widetilde{\vf}$ is a realisation of $M_2$.
\end{proof}

Our next result describes a technique for constructing parts of $R_{\mathbf{v}}$ using the rank of liftability matrices.

\begin{proposition}
    Let $N$ be a realisable matroid on $[n]$ with rank $d-1$. 
    Let $\mathbf{v}$ be a realisation of $N$, and suppose that $M_1 ,M_2$ are 1-elevations of $N$ such that:
    \begin{enumerate}[label=(\roman*)]
        \item $M_1 < M_2$;
        \item $M_1 \in R_{\vf}$; and
        \item if $M \in R_{\vf}$ satisfies $M_1 \leq M < M_2$, then $M =M_1$.
    \end{enumerate}  
    Then $M_2 \in R_{\vf}$ if and only if $\rank \mathcal{L}_{\mathbf{v}}(M_2) <  \rank \mathcal{L}_{\mathbf{v}}(M_1)$.
\label{prop:realisability_wo}
\end{proposition}

\begin{proof}
    As $M_1 < M_2$ and $M_1$ and $M_2$ are 1-elevations of $N$,
    we have that $\Ccal^{\text{ns}}(M_2) \subseteq \Ccal^{\text{ns}}(M_1)$.
    Thus $\rank \mathcal{L}_{\mathbf{v}}(M_2) \leq  \rank \mathcal{L}_{\mathbf{v}}(M_1)$ and $\ker \mathcal{L}_{\mathbf{v}}(M_1) \subseteq \ker \mathcal{L}_{\mathbf{v}}(M_2)$.
    If $\rank \mathcal{L}_{\mathbf{v}}(M_1)  =  \rank \mathcal{L}_{\mathbf{v}}(M_2)$ then $M_2 \notin R_{\vf}$ by \Cref{lem:realisability_wo}.
    
    Suppose instead that $\rank \mathcal{L}_{\mathbf{v}}(M_2) <  \rank \mathcal{L}_{\mathbf{v}}(M_1)$.
    As $\rank \mathcal{L}_{\mathbf{v}}(M_2) <  \rank \mathcal{L}_{\mathbf{v}}(M_1)$,
    we have that $\ker \mathcal{L}_{\mathbf{v}}(M_1)$ is a strict subspace of $\ker \mathcal{L}_{\mathbf{v}}(M_2)$.
    Choose any non-zero $z = (z_1,\ldots,z_n) \in \ker \mathcal{L}_{\mathbf{v}}(M_2) \setminus \ker \mathcal{L}_{\mathbf{v}}(M_1)$,
    and fix $\widetilde{\mathbf{w}} = (w_1,\ldots,w_n)$ with $w_i = (v_i,z_i)$ for each $i \in [n]$.
    By \Cref{prop: kernel and liftings 3},
    $\widetilde{\mathbf{w}} \in V_{\Ccal(M_2)} \setminus V_{\Ccal(M_1)}$,
    and the matroid $\mathcal{M}(\widetilde{\mathbf{w}})$ is contained in $R_{\vf}$.
    By \Cref{lem:circvar1},
    $\mathcal{M}(\widetilde{\mathbf{w}}) \leq M_2$ and $\mathcal{M}(\widetilde{\mathbf{w}}) \in \Lat(N) \setminus \{N\}$.
    It now suffices to show that we can choose $z^*\in \ker \mathcal{L}_{\mathbf{v}}(M_2) \setminus \ker \mathcal{L}_{\mathbf{v}}(M_1)$ so that every independent subset of $M_1$ is independent in $\mathcal{M}(\widetilde{\mathbf{w}})$ when $z=z^*$;
    indeed, if we can pick such a $z^*$, then $M_1 < \mathcal{M}(\widetilde{\mathbf{w}}) \leq M_2$,
    which by assumption implies that $\mathcal{M}(\widetilde{\mathbf{w}}) = M_2$.

    As $M_1 \in R_{\vf}$,
    there exists $z' \in \ker \mathcal{L}_{\vf}(M_1)$ such that $\widetilde{\mathbf{v}}' = \big((v_i,z_i') \big)_{i \in [n]}$ is a realisation of $M_1$.
    Choose any $z\in \ker \mathcal{L}_{\mathbf{v}}(M_2) \setminus \ker \mathcal{L}_{\mathbf{v}}(M_1)$,
    and for each $\mu \in \mathbb{C}$ fix $\widetilde{\mathbf{w}}_\mu := \big((v_i,z_i + \mu z_i') \big)_{i \in [n]}$.
    Since $z' \in \ker \mathcal{L}_{\vf}(M_1)$ and $z\in \ker \mathcal{L}_{\mathbf{v}}(M_2) \setminus \ker \mathcal{L}_{\mathbf{v}}(M_1)$,
    we have $z+ \mu z' \in \ker \mathcal{L}_{\mathbf{v}}(M_2) \setminus \ker \mathcal{L}_{\mathbf{v}}(M_1)$.

    \begin{claim}
        Let $I$ be an independent set in $M_1$.
        Then for all but at most one choice of $\mu$,
        $I$ is independent in $\mathcal{M}(\widetilde{\mathbf{w}}_\mu)$.
    \end{claim}

    \begin{claimproof}
        By relabelling, we may suppose that $I=\{1,\ldots,k\}$ is an independent set of $M_1$. If $k \leq d-1$, then $\{1,\ldots,k\}$ is independent in $N$, and therefore also in $\mathcal{M}(\widetilde{\mathbf{w}}_\mu)$ for any choice of $\mu$. Suppose that $k=d$, and that $\{1,\ldots,d\}$ is an independent set of $M_1$, but not in $\mathcal{M}(\widetilde{\mathbf{w}}_m)$ for some $m \in \mathbb{C}$. 
        Then $\{1,\ldots,d\}$ must be a circuit of $M$, and $\sum_{i=1}^d \lambda_i (v_i,z_i) = \mathbf{0}$, where $\lambda_1,\ldots,\lambda_d$ are the unique non-zero constants such that $\sum_{i=1}^d \lambda v_i = \mathbf{0}$.
        This now implies that for $\mu \neq m$,
        \begin{equation*}
            \sum_{i=1}^d \lambda_i(v_i,z_i +\mu z_i') = \sum_{i=1}^d \lambda_i(v_i,z_i +(\mu - m) z_i') + \sum_{i=1}^d \lambda_i(\mathbf{0},z_i +m z_i') =  \sum_{i=1}^d \lambda_i (v_i, (\mu-m) z_i') \neq \mathbf{0},
        \end{equation*}
        with the inequality following from $\{1,\ldots,k\}$ being independent in $\mathcal{M}(\widetilde{\mathbf{v}}') = M_1$.
        Thus $\{1,\ldots,k\}$ is independent in $\mathcal{M}(\widetilde{\mathbf{w}}_\mu)$ whenever $\mu \neq m$.
    \end{claimproof}

    We now choose $\mu$ so that every independent set in $M_1$ is independent in $\widetilde{\mathbf{w}}_\mu$ to conclude the proof.
\end{proof}

We now present the following algorithm for constructing $R_{\vf}$ for a rank-$(d-1)$ matroid $N$ on $[n]$ with realisation $\vf$:
\begin{algorithm}[H]
\begin{algorithmic}[1]\caption{Algorithm for constructing $R_{\vf}$}\label{algorithm Rv}
    \State Compute $\rank \mathcal{L}_{\vf}(M)$ for every $M \in \Lat(N)$.
    \State $A \gets \{ M \in \Lat(N) : \rank \mathcal{L}_{\vf}(M) \leq n-d , ~ \rank \mathcal{L}_{\vf}(M') = n-d +1 ~ \forall M' < M\}$
    
    \State $B \gets \{M \in \Lat(N) \setminus A : \rank \mathcal{L}_{\vf}(M) \leq n-d\}$.
     \Statex \hspace{\algorithmicindent} \textit{(Observe here that for each matroid $M \in B$ the set $\{X \in A : X < M\}$ is non-empty.)}
     \While{$B \neq \emptyset$}
    \State Choose a minimal element $M \in B$, and choose any maximal element $M'$ of $\{X \in A : X < M\}$. 
    \If{$\rank \mathcal{L}_{\vf}(M) < \rank \mathcal{L}_{\vf}(M')$}
    \State $A \gets A\cup \{M\}$.
    \EndIf
    \State $B \gets B \setminus \{M\}$.
    \EndWhile
    \State\Return $A$.
\end{algorithmic}
\end{algorithm}
It now follows from \Cref{prop:realisability_wo} that $R_{\vf} = A$.
\Cref{prop:realisability_wo} also informs us that \Cref{conj:genlift} being true would guarantee the following.

\begin{corollary}
    Let $N$ be a realisable matroid.
    If \Cref{conj:genlift} is true, then $R_{\vf} = R_{\mathbf{w}}$ for all generic realisations $\vf,\mathbf{w}$ of $N$.
\end{corollary}

\begin{proof}
    Choose any generic $\vf,\mathbf{w}$ of $N$.
    As \Cref{conj:genlift} is true,
    $\rank \mathcal{L}_\vf(M) = \rank \mathcal{L}_{\mathbf{w}}(M)$ for all $M \in \Lat(N)$.
    By applying \Cref{algorithm Rv},
    we see that $R_{\vf} = R_{\mathbf{w}}$.
\end{proof}

We can additionally use \Cref{prop:realisability_wo} to prove the (long overdue) following corollary:

\begin{corollary}\label{cor:rv_is_generic}
    Let $M$ be a rank-$d$ matroid with ground set $[n]$. 
    If $T(M)$ is realisable, then $M$ is generically realisable if and only if each irreducible component of $\Gamma_{T(M)}$ contains a generic realisation that lifts to $M$.
\end{corollary}

\begin{proof}
    The result follows immediately from the following consequences to \Cref{prop:realisability_wo} and \Cref{thm:genericlifting}:
    whether or not a realisation of $T(M)$ has a lifting to $M$ is entirely determined by the ranks of some family of liftability matrices, and all liftability matrices of the same matroid have equal rank for generic points in the same irreducible component of $\Gamma_{T(M)}$.
\end{proof}

\subsection{Construction of \texorpdfstring{$\Lat(N)$}{M(K5)} for a family of paving matroids}

In this subsection, we compute $\Lat(N)$ for a sparse paving matroid $N$ of rank $d-1$ with disjoint non-spanning circuits, i.e. its size-$(d-1)$ circuits are pairwise disjoint. We assume $d\geq 4$. Throughout this subsection, we let $N$ be a sparse paving matroid of rank $d-1$ on $[n]$ with disjoint non-spanning circuits given by $\{c_1,\ldots,c_{r}\}\subset [n]$.

The aim of this section is to compute the $1$-elevation lattice of $N$. However, we begin with the following proposition, which in particular asserts that matroids with disjoint non-spanning circuits are realisable. Hence we can always consider liftability of matroids in the $1$-elevation lattice.

\begin{proposition}
    Let $N$ be a matroid with pairwise disjoint non-spanning circuits. Then the realisation space of $N$ is irreducible and non-empty. In particular, $N$ has a unique generic Oxley-Wang derived matroid.
    \label{prop:disj-circuits-irr}
\end{proposition}

\begin{proof}
    Since the non-spanning circuits of $N$ are pairwise disjoint, it follows that the dependent hyperplanes of $N$ coincide with its non-spanning circuits. In particular the dependent hyperplanes of $N$ are pairwise disjoint. Hence, every element $i \in [n]$ belongs to at most one dependent hyperplane, so $N$ is nilpotent in the sense of \textup{\cite[Definition~3.2]{LIWSKI2026102484}}. The claim now follows from \textup{\cite[Theorem~A]{LIWSKI2026102484}}, which asserts that nilpotent matroids are realisable and have irreducible realisation spaces.
\end{proof}

\begin{definition}
Let $\mathcal{H}=\{H_{1},\ldots,H_k\}$ be a collection of subsets of $[n]$.  We say that $\mathcal{H}$ is $N$-{\em admissible} if it satisfies the following conditions:
\begin{enumerate}[label=(\roman*)]
\item\label{ad1} $\size{H_{i}}\geq d$ for all $i\in [k]$.
\item\label{ad2} $H_i\cap H_j$ contains no basis of $N$ for all distinct $i,j\in[k]$.
\item\label{ad3} For every $i\in[k]$ and $j\in[r]$, if $c_j\setminus\{x\}\subseteq H_i$ for some $x\in c_j$, then $x\in H_i$.
\item\label{ad4} No $H_i$ is of the form $c_j\cup\{x\}$ for any $j\in[r]$ and $x\notin c_j$.
\end{enumerate}
\end{definition}

Let $\HH=\{H_1,\ldots,H_k\}$ be an $N$-admissible subset. Given a family $\mathcal{A}$ of subsets of $n$, let $\min(\mathcal{A})$ denote the collection of inclusion-minimal subsets of $A$. Consider the following collection of subsets:
\begin{equation}\label{mathcalC}\mathcal{C}
:=
\min\left(
\{c_1, \ldots, c_r\} \cup
\binom{H_1}{d} \cup \cdots\cup
\binom{H_k}{d}
\cup
\binom{[n]}{d+1}
\right).\end{equation}

We show in \Cref{Sec:AppB} that $\mathcal{C}$ is the set of circuits of a matroid $M(\HH)$ if and only if $\HH$ is $N$-admissible. For a matroid $N$ with disjoint non-spanning circuits, we get the following explicit description of $\Lat(N)$ in terms of $N$-admissible collections.

\begin{theorem}\label{cor: Lat N}
Let $N$ be a sparse paving matroid with disjoint non-spanning circuits. Then
\begin{equation}\label{latn}\Lat(N)=\{M(\HH): \text{$\HH$ satisfies Conditions~\textup{\ref{ad1}-\ref{ad4}}}\}.\end{equation}
\end{theorem}

The proof of \Cref{cor: Lat N} can be found in \Cref{Sec:AppB}.\Cref{cor: Lat N} reduces the problem of computing and characterizing $\Lat(N)$ to the more tractable combinatorial problem of determining all $N$-admissible collections, that is, all collections $\HH$ of subsets of $[n]$ satisfying Conditions~\textup{\ref{ad1}--\ref{ad4}}. We now present an algorithm for computing all $N$-admissible collection.

\begin{algorithm}
\caption{Algorithm for computing all $N$-admissible collections}
\label{alg:admissible}
\begin{algorithmic}[1]
\State {\bf Input:} A sparse paving matroid $N$ of rank $d-1$ on $[n]$, with disjoint non-spanning circuits $c_1,\ldots,c_r$.
\State {\bf Output.} The set of all $N$-admissible collections.
\State $\mathscr{S}_N \gets \emptyset$.
\For{each $H\subseteq [n]$ with $\size{H}\geq d$}
  \If{$H$ satisfies Conditions~\ref{ad3} and~\ref{ad4}}
   \State $\mathscr{S}_N \gets \mathscr{S}_N\cup \{H\}$.
   \EndIf
   \EndFor

\State $\mathfrak{A}_N\gets \emptyset$.

\For{each $\mathcal{H}\subseteq \mathscr{S}_N$}
  \If{$H_i\cap H_j$ contains no basis of $N$ for all distinct $H_i,H_j\in \mathcal{H}$}
  \State $\mathfrak{A}_N\gets \mathfrak{A}_N\cup \mathcal{H}$.
  \EndIf
  \EndFor

\State Return $\mathfrak{A}_N$.
\end{algorithmic}
\end{algorithm}

\begin{example}
Let $N$ be a sparse paving matroid of rank $[n]$ with pairwise disjoint non-spanning circuits $\mathcal{C}^{ns}(N)=\{c_1,...,c_r\}$. Two extremal choices of the $N$-admissible collection $\HH$ recover particularly natural elements of $\Lat(N)$. Namely, if $\HH={[n]}$, then the associated matroid $M(\HH)$ is precisely $N$, viewed as the trivial elevation of itself. At the other extreme, if $\HH=\emptyset$, then $M(\emptyset)$ is the \emph{freest $1$-elevation} of $N$, that is, the maximal element of $\Lat(N)$.

\end{example}

\begin{example}\label{ex:1-lattice_disjoint}

    Let $N$ be the matroid of rank $3$ on $[6]$ with a single non-spanning circuit $c=\{1,2,3\}$. To construct the $1$-elevation lattice of $N$ we need to find the $N$-admissible collections $\mathcal{H}$. We start by considering the case that $\mathcal{H}$ consists of a single set. 
    
    The $N$-admissible $4$-element sets are $H_1=\{1,4,5,6\}$, $H_2=\{2,4,5,6\}$ and $H_3=\{3,4,5,6\}$. Any other subset $H$ of $[6]$ with $4$ elements must contain $c$, and therefore be of the form $c \cup \{x\}$ for some $x \in [n]$. 
 The $N$-admissible $5$-element subsets are $H_4=\{1,2,3,4,5\}$, $H_5=\{1,2,3,4,6\}$, and $H_6=\{1,2,3,5,6\}$. Any other subset $H$ of $[6]$ with $5$ elements must contain exactly two of the elements of $c$, and therefore not be $N$-admissible. Finally, $H=\emptyset$ and $H_7=[6]$ are $N$-admissible. Finally, note that the intersection $H_i \cap H_j$ contains a basis of $N$ for all $i \neq j$, so all $N$-admissible collections $\mathcal{H}$ must consist of a single set. The $1$-elevation lattice of $N$ can be found in \Cref{fig:lattice}.

    Now, $M(H_1)$ has non-spanning circuits $\{\{1,2,3\}, \{1,4,5,6\}\}$. Hence $\rank \mathcal{L}_\vf(M(H_1)))\leq 2$ for all realisations $\vf$ of $N$. In fact, since only the row corresponding to $\{1,2,3\}$ is supported in columns $2$ and $3$ and only the row corresponding to $\{1,4,5,6\}$ is supported in columns $4$, $5$ and $6$, it is easy to see that $\rank \mathcal{L}_\vf(M(H_1))) = 2$. Hence $M(H_1)$ is liftable, by \Cref{cor:realisationliftable}. Now, $\mathcal{C}^{ns}(M(H_4))=\{\{1,2,3\}, \{1,2,4,5\}, \{1,3,4,5\}, \{2,3,4,5\}\}$. Since $|\mathcal{C}^{ns}(M(H_4))| \geq 3$, we need to compute $\rank \mathcal{L}_\vf(M(H_4)))$ in order to determine whether $M(H_4)$ is liftable. Suppose that $\widetilde{\vf}$ is a generic realisation of $N$. By applying a projective transformation, we obtain a realisation $\vf$ of $N$ where $\vf_1=(0,1,1)^T$, $\vf_2=(0,-1,1)^T$, $\vf_4=(1,1,1)^T$ and $\vf_5=(1,-1,1)^T$. We then get that  
    \[\vf= 
    \begin{pmatrix}
        0 & 0 & 0 & 1 & 1 & y \\
        1 & -1 & x & 1 & -1 & z \\
        1 & 1 & 1 & 1 & 1 & 1
    \end{pmatrix}    
    \]
    where the set $\{x,y,z\}$ is algebraically independent over $\mathbb{Q}$. With the representation $\vf$ of $N$
    \[\mathcal{L}_\vf(M(H_4)))= 
    \begin{pmatrix}
        x+1 & 1-x & -2 & 0 & 0 & 0 \\
        1 & -1 & 0 & -1 & 1 & 0\\
        2 & 0 & -2 & x-1 & 1-x & 0 \\
        0 & 2 & -2 & x+1 & -x-1 & 0
    \end{pmatrix} .   
    \]
    It is now easy to check that, as $x,y,z$ are algebraically independent over $\mathbb{Q}$, we have $\rank \mathcal{L}_\vf(M(H_4)))=2$, so $\rank \mathcal{L}_\textbf{w}(M(H_4)))=2$ for any generic realisation $\textbf{w} \in \Gamma_N$. Hence $M(H_4)$ is liftable. Now, note that $M(H_5)$ and $M(H_6)$ are isomorphic to $M(H_4)$, and therefore also liftable.

    Hence $M(H_i)$ is a minimal liftable element of $\Lat N$ for all $1 \leq i \leq 6$. Therefore $M(H_i)$ is generically realisable for all $1 \leq i \leq 6$ by \Cref{thm:A}. By \Cref{prop:realisability_wo} it now also follows that $M(\emptyset)$ is generically realisable.
\end{example} 

\begin{figure}
    \centering
        \begin{tikzpicture}
        \node[anchor=south] at (0,0) {$M(\emptyset)$};
    \node at (1,-2) {$M(H_3)$};
    \node at (3,-2) {$M(H_2)$};
    \node at (5,-2) {$M(H_1)$};
    \node at (-1,-2) {$M(H_4)$};
    \node at (-3,-2) {$M(H_5)$};
    \node at (-5,-2) {$M(H_6)$};
    \node at (0,-4) {$N$};
    \draw (0,-0.2)--(1,-1.8);
    \draw (0,-0.2)--(3,-1.8);
    \draw (0,-0.2)--(5,-1.8);
    \draw (0,-0.2)--(-1,-1.8);
    \draw (0,-0.2)--(-3,-1.8);
    \draw (0,-0.2)--(-5,-1.8);
    \draw (0,-3.8)--(1,-2.2);
    \draw (0,-3.8)--(3,-2.2);
    \draw (0,-3.8)--(5,-2.2);
    \draw (0,-3.8)--(-1,-2.2);
    \draw (0,-3.8)--(-3,-2.2);
    \draw (0,-3.8)--(-5,-2.2);
\end{tikzpicture}
    \caption{The $1$-elevation lattice of the matroid from \Cref{ex:1-lattice_disjoint}.}
    \label{fig:lattice}
\end{figure}
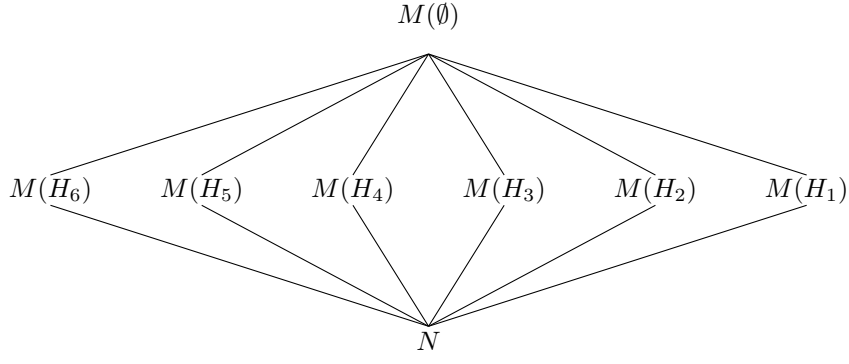

\section{Describing liftability via Oxley-Wang matroids}
\label{sec:liftabilityviaOW}

In this section, we connect the properties of liftability and generic realisability with the notion of {\em Oxley--Wang derived} matroids introduced in \cite{oxley2019dependencies}. We introduce {\em generic Oxley--Wang} matroids, and characterise liftability and generic realisability in terms of them. Furthermore, we give an algebraic description of their bases through coordinate sections of realisation spaces, and we relate these constructions to slack realisation spaces of matroids from \cite{brandt2019slack}.

\subsection{Generic Oxley--Wang matroids and liftability}
\label{sec:OWmatroids}

We begin with the section with the following concept introduced by Oxley and Wang in \cite{oxley2019dependencies}.
We define it here using the terminology introduced in \Cref{subsec:circuitmatrices}.

\begin{definition}\label{def: generic oxley wang}
    Let $M$ be a realisable matroid of rank $d$ on $[n]$, and let $\vf = (v_{1}, \ldots, v_{n})$ be a realisation of $M$.
    The \emph{Oxley--Wang derived matroid} $\delta_{\mathrm{OW}}(\vf)$ is defined as the row matroid of $\mathcal{R}_\vf(M;\mathcal{C}(M))$.
Suppose $\Gamma_M$ has irreducible components $V_1,\ldots,V_k$. For each $i\in [k]$, the matroid $\delta_{\mathrm{OW}}(\vf)$ is the same for all generic $\vf\in V_i$ (see \Cref{lem:genericliftmatroid}). We denote this matroid by $\delta_i M$ and call it a {\em generic Oxley--Wang derived matroid} of $M$.
\end{definition}

It follows from \Cref{lem:genericliftmatroid} that a matroid $M$ has finitely many generic Oxley--Wang derived matroids,
which we label $\delta_1 M, \ldots, \delta_k M$.
If $M$ has a single unique generic Oxley--Wang derived matroid,
we say it is \emph{the} generic Oxley--Wang derived matroid and denote it by $\delta M$.
As noted in \Cref{rem:rowvcolmatroid},
every generic Oxley--Wang derived matroid of a rank-$d$ matroid on $[n]$ will have rank $n-d$.

We now prove our second main theorem, \Cref{thm:B},
which we recall is the following:

\begin{theorem}\label{thm: liftable from Oxley wang}
Let $M$ be a rank-$d$ matroid on $[n]$ such that $T(M)$ is realisable.
Suppose that $T(M)$ has generic Oxley--Wang derived matroids $\delta_1 T(M),\ldots, \delta_k T(M)$.
Then, $M$ is liftable if and only if
\[
\max_{i=1,\ldots,k}\rank_{\delta_{i}T(M)}
   ( \mathcal{C}^{\text{ns}}(M) )
   \leq n-d.
\]
\end{theorem}

\begin{proof}
    The generic liftability matroids of $M$ are exactly the restrictions of the generic Oxley--Wang derived matroids of $T(M)$ to $\mathcal{C}^{\text{ns}}(M)$.
    The result now follows from \Cref{thm:genericlifting}.
\end{proof}

Thus, deciding whether $M$ is liftable reduces to computing the rank of $\mathcal{C}^{\mathrm{ns}}(M)$ in the generic Oxley--Wang derived matroids of $T(M)$. The following lemma gives a necessary condition for independence in the Oxley-Wang derived matroid.

\begin{lemma}\label{lem: necessary for independent}
Let $M$ be a matroid of rank $d$ on $[n]$, and let $\vf = (v_1,\ldots,v_n)$ be a realisation of $M$. If $A \subseteq \mathcal{C}(M)$ is independent in $\delta_{\mathrm{OW}}(\vf)$, then for every non-empty $B \subseteq A$ we have
\[
\size{B} \le \size{V(B)} - \rank_M(V(B)),
\]
where $V(B) \subseteq [n]$ denotes the union of the circuits in $B$.
\end{lemma}

\begin{proof}
See \cite[\S~4, Equation~(2)]{freij2023combinatorial}.
\end{proof}

\begin{example}
    Let $M$ be the unique rank-3 paving matroid on $\{1,\ldots,8\}$ with collection of non-spanning circuits: 
    \begin{equation*}
        \mathcal{C}^{\text{ns}}(M)=
        \{\{1,2,8\},\{2,3,5\},\{3,4,6\},\{1,4,7\},\{1,5,6\},\{2,6,7\},\{3,7,8\},\{4,5,8\}\}.
    \end{equation*}
    The realisation space of $M$ has exactly two irreducible components, see \textup{\cite[Table~4.1]{corey2023singular}}. Following the procedure described in \textup{\cite[\S~2.4]{corey2023singular}}, we see that every realisation of $M$ is obtained from one of the following, up to row operations and scaling columns by non-zero scalars:
    \begin{equation*}
        \vf_{1}:=
        \begin{pmatrix}
        1 & 0 & 0 & 1 & 0 & 1 & 1 &1 \\
        0 & 1 & 0 & 1 & 1 & 1 & x_{1} & x_{1}\\
        0 & 0 & 1 & 1 & x_{1} & x_{1} & x_{1} & 0
        \end{pmatrix}, \qquad 
        \vf_{2}:=
        \begin{pmatrix}
        1 & 0 & 0 & 1 & 0 & 1 & 1 &1 \\
        0 & 1 & 0 & 1 & 1 & 1 & x_{2} & x_{2}\\
        0 & 0 & 1 & 1 & x_{2} & x_{2} & x_{2} & 0,
        \end{pmatrix}
    \end{equation*}
    where $x_{1}$ and $x_{2}$ are the roots of the quadratic polynomial $x^{2}-x+1$, and the columns correspond to the elements $\{1,2,3,\ldots,8\}$.
    The realisations $\vf_1,\vf_2$ generate the two irreducible components of $\Gamma_{M}$,
    and so the generic Oxley--Wang matroids of $M$ are $\delta_{\mathrm{OW}}(\vf_{1})$ and $\delta_{\mathrm{OW}}(\vf_{2})$. 
    
    A direct computation shows that 
    \[\mathcal{R}_{\vf_i}(M;\mathcal{C}^{\text{ns}}(M))=
    \begin{pmatrix}
1 & x_i & 0 & 0 & 0 & 0 & 0 & -1\\
0 & 1 & x_i & 0 & -1 & 0 & 0 & 0\\
0 & 0 & x_{i}-1 & 1 & 0 & -1 & 0 & 0\\
1-x_i & 0 & 0 & x_i & 0 & 0 & -1 \\
1 & 0 & 0 & 0 & 1 & -1 & 0 & 0\\
0 & x_{i}-1 & 0 & 0 & 0 & 1 & -1 & 0\\
0 & 0 & x_i & 0 & 0 & 0 & -1 & 1\\
0 & 0 & 0 & 1 & x_{i}-1 & 0 & 0 & -1
    \end{pmatrix}
    \]
for $i=1,2$. We observe that the polynomials appearing in the entries of the matrices $\mathcal{R}_{\vf_1}(M;\mathcal{C}^{\text{ns}}(M))$ and $\mathcal{R}_{\vf_2}(M;\mathcal{C}^{\text{ns}}(M))$ are the same. The same holds for the matrices $\mathcal{R}_{\vf_1}(M;\mathcal{C}(M))$ and $\mathcal{R}_{\vf_2}(M;\mathcal{C}(M))$. Since $x^{2}-x+1$ is irreducible over $\QQ$, it is the minimal polynomial of both $x_1$ and $x_2$ over $\QQ$. Therefore
\[\{f\in \QQ[x]: f(x_1)=0\}=(x^{2}-x+1)=\{f\in \QQ[x]: f(x_2)=0\},\]
and the set of minors that vanish are the same for $\mathcal{R}_{\vf_1}(M;\mathcal{C}(M))$ and $\mathcal{R}_{\vf_2}(M;\mathcal{C}(M))$.
Therefore, the matrices
\[
\mathcal{R}_{\vf_1}(M;\mathcal{C}(M))
\quad\text{and}\quad
\mathcal{R}_{\vf_2}(M;\mathcal{C}(M))
\]
have the same row matroid. Hence the corresponding Oxley--Wang derived matroids satisfy
\[
\delta_{\mathrm{OW}}(\vf_1)=\delta_{\mathrm{OW}}(\vf_2).
\]
Hence, the two generic Oxley--Wang derived matroids of $M$ coincide.

\end{example}

Similar to generic liftability matroids (see \Cref{conj:genlift}), the authors would conjecture that there exists a unique generic Oxley--Wang derived matroid for any realisable matroid.

\begin{conjecture}\label{conj:gom}
    Every realisable matroid has a unique generic Oxley--Wang derived matroid.
\end{conjecture}

\begin{remark}
    Since the generic liftability matroids of $M$ are exactly the restrictions of the generic Oxley--Wang derived matroids of $T(M)$,
    \Cref{conj:gom} being true would imply \Cref{conj:genlift} is true.
    \Cref{conj:gom} is a weaker conjecture than \cite[Conjecture 45]{freij2023combinatorial}.
    The latter conjecture would not only imply that there exists a unique generic Oxley--Wang derived matroid,
    but that it is equivalent to the \emph{combinatorial derived matroid} defined by Freij-Hollanti, Jurrius and Kuznetsova in \cite{freij2023combinatorial}.
\end{remark}

\subsection{Generic Oxley--Wang matroids and generic realisability}

Our aim in this section is to prove \Cref{thm:C}. We begin with the following lemma.

\begin{lemma}\label{lem: flat in oxley wang}
Let $M$ be a matroid of rank $d$ on $[n]$ and let $\vf=(v_1,\ldots,v_n)$ be a realisation of $T(M)$. Then, there exists a lifting $\widetilde{\vf}$ of $\vf$ that is a realisation of $M$ if and only if $\mathcal{C}^{\text{ns}}(M)$ is a flat of $\OW$.
\end{lemma}

\begin{proof}
    We first make the following observation:
    if $\widetilde{\mathbf{v}} = \big( (v_{1},z_1), \ldots, (v_{n} , z_{n})\big)$ is a lifting of $\vf$,
    then the following equivalence holds for any $c \in \mathcal{C}(T(M))$,
    \begin{equation}\label{eq:lindepiff}
        (v_i,z_i)_{i \in c} \text{ are linearly dependent } \iff z\perp r_{c}.
    \end{equation}
    
    Now suppose that $\widetilde{\mathbf{v}} = \big( (v_{1},z_1), \ldots, (v_{n} , z_{n})\big)$ is a lifting of $\vf$ such that $\widetilde{\vf}$ is a realisation of $M$. Since $\widetilde{\vf}$ is a realisation of $M$, it follows from \eqref{eq:lindepiff} that $z\in \textup{span}\{r_{c}: c\in \mathcal{C}^{\text{ns}}(M)\}^{\perp}$ but $z\not \perp r_{c}$ for each circuit contained in $c\in \mathcal{C}(T(M))\setminus \mathcal{C}^{\text{ns}}(M)$, as the latter set is exactly the set of bases of $M$. 
    Hence 
    \begin{equation*}
        \rank \mathcal{R}_{\vf} \Big(T(M); \mathcal{C}^{\text{ns}}(M) \cup \{c\} \Big) > \rank \mathcal{R}_{\vf}\Big(T(M); \mathcal{C}^{\text{ns}}(M) \Big)
    \end{equation*}
    for any $c \in \mathcal{C}(T(M)) \setminus \mathcal{C}^{\text{ns}}(M)$.
    Since $\rank \mathcal{R}_{\vf}(T(M);S)$ is the rank of $S$ in $\OW$ for any collection of circuits $S$ of $T(M)$, it follows that $\mathcal{C}^{\text{ns}}(M))$ is a flat of $\OW$.

    Suppose now that $\mathcal{C}^{\text{ns}}(M)$ is a flat of $\OW$, or equivalently, that 
    \begin{equation}\label{c and c'}
\rank\big\{r_{c'}:c'\in \mathcal{C}^{\text{ns}}(M) \cup \{c\}\big\}>\rank\big\{r_{c'}: c'\in \mathcal{C}^{\text{ns}}(M)\big\}
    \end{equation}
 for every $c \in \mathcal{C}(T(M)) \setminus \mathcal{C}^{\text{ns}}(M)$. 

\begin{claim}\label{claim: flatness}
    There is some $z \in \mathbb{C}^n$ such that 
    \begin{equation}\label{eq: perp z}\text{$z\perp r_{c}$ for all $c\in \mathcal{C}^{\text{ns}}(M)$}\quad \text{and} \quad \text{$z\not \perp r_{c}$ for all $c\in \mathcal{C}(T(M))\setminus \mathcal{C}^{\text{ns}}(M)$}.\end{equation}
\end{claim}

    \begin{claimproof}
    Let $W=\operatorname{span}\{r_c\in\mathcal C^{\mathrm{ns}}(M)\}$.
From~\eqref{c and c'}, it follows that every remaining circuit vector lies outside $W$. One then chooses
\[z\in W^\perp\setminus\bigcup_{c\notin\mathcal C^{\mathrm{ns}}(M)}(W^\perp\cap r_c^\perp),\]
which is non-empty because it is the complement of finitely many proper hyperplanes. 
    \end{claimproof}

Now, let $z$ be as in~\Cref{claim: flatness}.
Since $z\perp r_{c}$ for all $c\in \mathcal{C}^{\text{ns}}(M)$ and $z\not \perp r_{c}$ for all $c\in \mathcal{C}(T(M))\setminus \mathcal{C}^{\text{ns}}(M)$ (which we recall is exactly the set of bases of $M$), it follows from \eqref{eq:lindepiff} that $\widetilde{\vf}$ is a realisation of $M$.
\end{proof}

We can now prove~\Cref{thm:C},
which we recall to be the following:

\begin{theorem}
    Let $M$ be a rank-$d$ matroid on $[n]$ such that $T(M)$ is realisable.
    Suppose that $T(M)$ has generic Oxley--Wang derived matroids $\delta_1 T(M),\ldots, \delta_k T(M)$. Then, $M$ is generically realisable if and only if 
    \[\text{$\mathcal{C}^{\text{ns}}(M)$ is a flat on each of the matroids $\delta_1 T(M),\ldots, \delta_k T(M)$.}\]
\end{theorem}

\begin{proof}
Let $\delta_1 T(M),\ldots, \delta_k T(M)$ be the generic Oxley--Wang derived matroids of $T(M)$ corresponding to the irreducible components $V_1,\ldots,V_k$ of $\Gamma_{T(M)}$, respectively. 
Choose a generic point $\vf_i$ of $V_i$ for each irreducible component.
By \Cref{lem: flat in oxley wang},
we have for each $i \in [k]$ that
$M \in R_{\vf_i}$ if and only if $\mathcal{C}^{\text{ns}}(M)$ is a flat of $\delta_i T(M)$.
The result now follows from \Cref{cor:rv_is_generic}.
\end{proof}

\subsection{Bases of Oxley--Wang matroids}
\label{sec:OWcoordinatesections}

For this subsection we introduce the following notation.
Given a realisable matroid $M$ with irreducible component $V$,
we denote the generic Oxley--Wang derived matroid generated from a generic point of $V$ by $\delta_V M$.

By \Cref{thm:B}, liftability of a matroid can be determined by computing the rank of its non-spanning circuits in the generic Oxley--Wang derived matroids of $T(M)$. 
The following result now provides an exact characterisation of the bases of $\delta M$ that facilitates this computation.

\begin{theorem}\label{thm: fixed support}
Let $M$ be a realisable rank-$d$ matroid on $[n]$. 
Let $V$ be an irreducible component of the matroid variety $V_M$ and let $V^{\ast}$ be the corresponding irreducible component of $V_{M^{\ast}}$ (as defined in \Cref{rem:dual}).
Then a collection of circuits $\{c_{1},\ldots,c_{n-d}\}\subseteq \mathcal{C}(M)$ forms a basis of $\delta_V M$ if and only if there exists a matrix $\mathbf{w}\in \CC^{(n-d)\times n}$
such that:
\begin{itemize}
    \item[\rm (i)] $\mathbf{w}$ is a realisation of $M^{\ast}$ contained in $V^{\ast}$;
    \item[\rm (ii)] $\supp(\mathbf{w}^{(i)})=\{(i,j): j\in c_{i}\}$ for each $i$, where $\mathbf{w}^{(i)}$ denotes the $i$-th row of $\mathbf{w}$.
\end{itemize}
\end{theorem}

\begin{proof}
Assume first that $\{c_{1},\ldots,c_{n-d}\}$ is a basis of $\delta_V M$.
Fix some generic realisation $\vf=(v_1,\ldots,v_n)\in V$.
Then the rows of $\mathcal{R}_\vf(M;\mathcal{C}(M))$ indexed by $c_{1},\ldots,c_{n-d}$ are linearly independent. Let $\mathbf{w}$ denote the submatrix of $\mathcal{R}_\vf(M;\mathcal{C}(M))$ consisting of these rows. 
Since these columns are independent and $\operatorname{rank}(\mathcal{R}_\vf(M;\mathcal{C}(M)))=n-d$, we have
\[
\operatorname{rowspan}(\mathbf{w})=\operatorname{rowspan}(\mathcal{R}_\vf(M;\mathcal{C}(M))).
\]
Thus, the column matroid of $\mathbf{w}$ coincides with the column matroid of $\mathcal{R}_\vf(M;\mathcal{C}(M))$, which, as noted in \Cref{rem:rowvcolmatroid}, is exactly $M^{\ast}$. Moreover, since $\vf\in V$, it follows that $\mathbf{w}\in V^{\ast}$ (as observed in \Cref{rem:dual}). 
Furthermore, the $i$-th row of $\mathbf{w}$ has support precisely $c_{i}$. 
Hence $\mathbf{w}$ satisfies (i) and (ii).

Conversely, suppose that $\mathbf{w}\in \CC^{(n-d)\times n}$ satisfies (i) and (ii). Let $\vf \in \CC^{d\times n}$ be any matrix whose rows form a basis for the right kernel of $\mathbf{w}$. It follows from \Cref{rem:dual} that $\mathbf{v}\in V$. 
For each circuit $c\in \mathcal{C}(M)$, let $r_{c}$ be the (unique up to scaling) vector in $\ker(\vf)=\operatorname{rowspan}(\mathbf{w})$ with support exactly $c$ described in \eqref{eq:rcrm}. Since the $i$-th row of $\mathbf{w}$ has support $c_{i}$, it coincides with $r_{c_{i}}$. 
Hence, up to scaling, $\mathbf{w}$ is the restriction of $\mathcal{R}_{\vf}(M;\mathcal{C}(M))$ to the rows indexed by $c_{1},\ldots,c_{n-d}$. Because $\rank \mathbf{w} = n-d$, these rows must be independent, and therefore $\{c_{1},\ldots,c_{n-d}\}$ is a basis of the row matorid of $\mathcal{R}_{\vf}(M;\mathcal{C}(M))$.
Thus by \Cref{lem:genericliftmatroid}, $\{c_{1},\ldots,c_{n-d}\}$ forms a basis of $\delta_V M$.
\end{proof}

\Cref{thm: fixed support} can be reformulated in the language of coordinate sections of varieties.
Here we recall that for a variety $V\subseteq \CC^{k}$ and subset $I\subseteq [k]$, the coordinate section of $V$ along $I$ is the variety
\begin{equation*}
    V_I := \big\{ x = (x_1,\ldots,x_k) \in V : x_i = 0 \text{ for all } i \in I \big\}.
\end{equation*}
Coordinate sections of determinantal varieties have been studied in various settings; see, for example, \cite{zbMATH03803604,eisenbud1988linear,conca2019lovasz,cunha2022coordinate}.

Now take our variety $V$ to be an irreducible component of the matroid variety $V_M$ and, given a set of circuits $\{c_{1},\ldots,c_{n-d}\}\subseteq \mathcal{C}(M)$, let 
\[
I =\bigcup_{s\in [n-d]}\{(s,j): j\not\in c_{s}\}\subseteq [n-d]\times[n].
\]
Then, by \Cref{thm: fixed support}, $\{c_{1},\ldots,c_{n-d}\}$ is a basis of $\delta_V M$ if and only if $(V^{\ast})_{I}\neq \emptyset$.

\subsection{Relation to the slack matrix of a matroid}\label{sec: slack}

In \cite{brandt2019slack}, Brandt and Wiebe introduced the notion of the {\em slack realisation space} of a matroid. This represents a new model for the realisation space of a matroid, which they used to certify non-realisability of matroids and to detect projectively unique matroids.

\begin{definition}\label{def:slack}
    Let $\vf$ be a realisation of a rank-$d$ matroid $M$ on $[n]$,
    here considered as a $d \times n$ matrix. 
    For each hyperplane $H \in \mathcal{H}(M)$, let $a_H \in \CC^d$ be a non-zero normal vector to the hyperplane spanned by the corresponding columns of $\vf$. 
    Writing $\mathcal{H}(M) = \{H_1, \ldots, H_h\}$, define $\mathbf{a}$ to be the matrix whose columns are the vectors $a_{H_1}, \ldots, a_{H_h}$.
    The matrix $S_\vf(M) = \vf^\top \mathbf{a}$ is now said to be a \emph{slack matrix} of $M$.
\end{definition}

Although there are infinitely many slack matrices for a give matrix realisation,
all such slack matrices have the same columns up to non-zero scaling.
Thus we can speak about the slack matrix $S_\vf(M)$ as if it is a single matrix (in the same way we speak about the matrices $\mathcal{R}_\vf(M;S)$ as being unique).

We now show that slack matrices are dual to the Oxley--Wang derived matroid construction.
As observed in \Cref{rem:dual}, for every realisation $\vf \in \mathbb{C}^{d \times n}$ of a rank-$d$ matroid $n$,
any matrix $\mathbf{w} \in \CC^{(n-d) \times d}$ satisfying $\mathbf{w}^\top \vf = \mathbf{0}$ is a realisation of the dual matroid $M^*$.

\begin{proposition}\label{lem: delta and slack}
    Let $M$ be a rank-$d$ matroid on $[n]$ with realisation $\vf \in \mathbb{C}^{d \times n}$.
    Let $\mathbf{w} \in \CC^{(n-d) \times n}$ be a realisation of $M^*$ satisfying $\mathbf{w} \vf^\top = \mathbf{0}$.
    Then the matrices $\mathcal{R}_\vf(M;\mathcal{C}(M))$ and $S_\mathbf{w}(M^*)^\top$ agree up to non-zero scalar multiplication and permutation of rows. In particular, the column matroid of $S_\mathbf{w}(M^*)$ equals $\delta_{\mathrm{OW}}(\vf)$.
\end{proposition}

\begin{proof}
    Let $\mathbf{b}\in \CC^{(n-d)\times h}$ be the unique matrix whose columns represent normal vectors to the hyperplanes for the realisation $\mathbf{w}$ for which $S_\mathbf{w}(M^*) = \mathbf{w}^\top \mathbf{b}$.
    Then 
    \begin{equation*}
        S_\mathbf{w}(M^*)^\top \vf^\top =  \mathbf{b}^\top\mathbf{w} \vf^\top= \mathbf{b}^\top \mathbf{0} = \mathbf{0}.
    \end{equation*}
    The hyperplanes of $M^{\ast}$ are precisely the complements of the circuits of $M$.
    Thus we can order the columns of the slack matrix $S_\mathbf{w}(M^*)$ so that the support of the $i$-th column of $S_\mathbf{w}(M^*)$ is the support of the $i$-th row of $\mathcal{R}_\vf(M;\mathcal{C}(M))$.
    Thus the $i$-th column of $S_\mathbf{w}(M^*)$ satisfies the conditions in \eqref{eq:rcrm} for the $i$-th circuit of $M$.
    This now implies the matrices $\mathcal{R}_\vf(M;\mathcal{C}(M))$ and $S_\mathbf{w}(M^*)^\top$ agree up to non-zero scalar multiplication of their rows.
\end{proof}

\section{Characterising liftable paving matroids with scene analysis}
\label{sec:applications}

In this section, we apply our framework to paving matroids, and prove the equivalence between a paving matroid being liftablie and the corresponding hypergraph admitting a non-trivial scene. We obtain a combinatorial characterisation of liftability and generic realisability for paving matroids via Oxley--Wang derived matroids, which yields a new proof of Whiteley's characterisation of hypergraphs that admit non-trivial scenes, see \cite{whiteley1989matroid}. 

\subsection{Equivalence of scene analysis and liftability for paving matroids}
\label{sec:sceneanalysisliftabilitysubsec}

In this section, we give a characterisation of liftability of paving matroids, using \Cref{thm:indepA}. 

For a rank-$d$ paving matroid $M$ on $[n]$,
we set $H_M$ to be the $d$-uniform hypergraph with vertex set $[n]$ and hyperedges given by the non-spanning circuits of $M$,
i.e.~the set $\mathcal{C}_d(M)$.
The following lemma gives the connection between liftability of $M$ and existence of scenes over generic pictures of $H_M$.

\begin{lemma}
    Let $M$ be a rank-$d$ paving matroid on $[n]$, and let $\vf = (v_{1}, \ldots, v_{n})$ be a tuple of vectors in $\CC^{d-1}$ of the form $v_{i} = (x_{i}, 1)\in \CC^{d}$, where $x_{i} = (x_{i,1},...,x_{i,d-2}) \in \CC^{d-2}$. Let  $x : [n] \to \CC^{d-2}$ be the picture given by $x(i) = x_{i}$. Then
\[
\mathcal{L}_{\vf}(M) \;=\; A(H_M, x),
\]
where the equality is understood up to multiplication of each row by a nonzero scalar.
\label{lem:matrices_coincide}
\end{lemma}

\begin{proof}
Fix a non-spanning circuit $c = (n_1,\ldots,n_d)$ of $M$ with $n_1 < \cdots < n_d$.
Set $a= (a_1,\ldots, a_n)$ to be the vector in $\mathbb{C}^n$ for which $a_i= (-1)^j \det (X_j(c))$ (see \Cref{eq:scenelinear}) if $i = n_j$ and $a_i=0$ otherwise.

Suppose that $z=(z_1,...,z_n)$ is a lifting of $\vf$. Then
\begin{equation*}
    \sum_{i}^n a_i z_i = \left(
    \det 
    \begin{pmatrix}
    x_{n_1,1} & \cdots & x_{n_{d},1} \\
    \vdots & \vdots & \vdots \\
    x_{n_{1},d-2} & \cdots & x_{n_{d},d-2} \\
    z_{n_1} & \cdots & z_{n_d} \\
    1 & \cdots & 1
    \end{pmatrix}
    \right) = \mathbf{0}
\end{equation*}
Thus $a$ satisfies \eqref{eq:rcrm}, so up to multiplication by a non-zero scalar $a=r_c$.
As $a$ is the row of $A(H_M,x)$ corresponding to the circuit $c$,
we can now apply the above to each non-spanning circuit of $M$ to obtain the desired equality.
\end{proof}

We now use \Cref{lem:matrices_coincide} to provide a link between non-trivial scenes and non-trivial liftings.

\begin{lemma}\label{lem:linkingscenesandliftings}
    Let $M$ be a rank-$d$ paving matroid on $[n]$.
    \begin{itemize}
        \item[\rm{(i)}]  $T(M)$ has a realisation with a non-trivial lifting to $V_{\mathcal{C}(M)}$ if and only if there exists a general position picture of $H_M$ that lifts to a non-trivial scene.
        \item[\rm{(ii)}]
        If $\vf$ is a generic realisation of $T(M)$ and $x$ is a generic picture of $H_M$,
        then
        \begin{equation*}
            \rank \mathcal{L}_{\vf}(M) = \rank A(H_M,x).
        \end{equation*}
        In particular,
        $M$ is liftable if and only if there exists a generic picture of $H_M$ that lifts to a non-trivial scene.
    \end{itemize}
\end{lemma}

\begin{proof}
    Since $T(M) = U_{d-1,n}$, an $n$-tuple $\mathbf{w} \in (\mathbb{C}^{d-1})^n$ is contained in $\Gamma_{T(M)}$ if and only if it is in general position.
    Let $f : (\mathbb{C}^{d-1})^n \dashrightarrow (\mathbb{C}^{d-2})^n$ and $g : (\mathbb{C}^{d-2})^n \rightarrow (\mathbb{C}^{d-2} \times \{1\})^n$ be the rational maps
    \begin{eqnarray*}
        f :& \mathbf{w} = \left( \big(w_1(i) \big)_{i=1}^{d-1}, \ldots, \big(w_n(i) \big)_{i=1}^{d-1} \right) & \longmapsto \quad \Big( \big(w_1(i)/w_{d-1}(i) \big)_{i=1}^{d-2}, \ldots, \big(w_n(i)/w_{d-1}(i) \big)_{i=1}^{d-2} \Big), \\
        g : & x = \left( x_1, \ldots, x_n \right) & \longmapsto \quad \Big( \big(x_1,1 \big), \ldots, \big(x_n,1\big)\Big).
    \end{eqnarray*}
    Here we observe the following facts:
    \begin{enumerate}
        \item\label{point2:linkingscenesandliftings} $g(x) \in \Gamma_{T(M)}$ if and only if $x$ is a general position picture.
        \item\label{point3:linkingscenesandliftings} If $\mathbf{w} \in (\mathbb{C}^{d-1})^n$ is generic, then $f(\mathbf{w})$ exists and is a generic picture.
        Similarly, for every generic picture $x$, there exists generic $\mathbf{w} \in \Gamma_{T(M)}$ for which $f(\mathbf{w})$ is generic.
        \item\label{point4:linkingscenesandliftings} If $f(\mathbf{w})$ exists and one of $\mathbf{w}, g \circ f(\mathbf{w})$ is a realisation of $T(M)$, then both are;
        moreover, if both are realisations then
        \begin{equation*}
            \rank \mathcal{L}_{g \circ f(\mathbf{w})}(M) = \rank \mathcal{L}_{\mathbf{w}}(M);
        \end{equation*}
        in particular, $g \circ f(\mathbf{w})$ has a non-trivial lifting if and only if $\mathbf{w}$ has a non-trivial lifting (\Cref{cor:realisationliftable}).
    \end{enumerate}
    We now prove (i) and (ii) using these facts and \Cref{lem:matrices_coincide}.
    \vspace{2mm}

    \noindent\textbf{(i):}
    If there exists a general position picture $x$ with a non-trivial scene, then $g(x)$ is a realisation of $T(M)$ (fact \ref{point2:linkingscenesandliftings}) with non-trivial lifting (\Cref{lem:matrices_coincide}).
    Conversely, suppose that $\mathbf{w} \in \Gamma_{T(M)}$ has a non-trivial lifting.
    When considered as a matrix, $\mathbf{w}$ has linearly independent rows.
    This implies that we can form a new realisation $\mathbf{w}' \in \Gamma_{T(M)}$ from $\mathbf{w}$ using only row operations so that, as a realisation, the $(d-1)$-th coordinate of each point $v_i$ is non-zero.
    Since $\mathbf{w}' = X\mathbf{w}$ for some invertible matrix $X \in \mathbb{C}^{d-1 \times d-1}$,
    $\mathbf{w}'$ also has a non-trivial lifting.
    As $f(\mathbf{w}')$ exists, it now follows from fact \ref{point4:linkingscenesandliftings} and \Cref{lem:matrices_coincide} that there exists a general position picture with a non-trivial scene.
    \vspace{2mm}

    \noindent\textbf{(ii):}
    Since the rank of the liftability matrix is constant over generic realisations in the same irreducible component of  $\Gamma_{T(M)}$ (\Cref{lem:genericliftmatroid}) and  $\Gamma_{T(M)}$ is irreducible,
    the rank equation follows from \Cref{lem:matrices_coincide} and facts \ref{point3:linkingscenesandliftings} and \ref{point4:linkingscenesandliftings}.
    It now follows from \Cref{thm:genericlifting} that $M$ is liftable if and only if there exists a generic picture of $H_M$ that lifts to a non-trivial scene.
\end{proof}

\Cref{thm:D} is now an immediate consequence of the following theorem:

\begin{theorem}\label{thm: characterisation of liftable}
Let $M$ be a paving matroid of rank $d$ on $[n]$. Then the rank of any generic liftability matroid of $M$ is exactly the rank of $\mathcal{C}_d(M)$ in the matroid $\mathcal{S}_{(1,d-1)}$.
In particular, 
$M$ is liftable if and only if
\begin{equation*}
    \rank_{\mathcal{S}_{(1,d-1)}}(\mathcal{C}_d(M)) \leq n - d.
\end{equation*}
\end{theorem}

\begin{proof}
By \Cref{lem:linkingscenesandliftings}(ii),
the rank of $\mathcal{L}_{\vf}(M)$ for any generic realisation $\vf$ of $T(M)$ is equal to the rank of $A(H_M,x)$ for any generic picture $x$ of $H_M$. 
Thus the rank of any generic liftability matroid of $M$ is exactly the rank of $\mathcal{C}_d(M)$ in the matroid $\mathcal{S}_{(1,d-1)}$ by \Cref{thm:indepA}. The result now follows from \Cref{cor:non-triv.scene}.
\end{proof}

\subsection{A new proof of Whiteley's picture theorem using Oxley--Wang derived matroids}
\label{sec:Uniform_OW}

Let $U_{d,n}$ be the rank-$d$ uniform matroid on $[n]$.
Since $V_{U_{d,n}} = (\mathbb{C}^d)^n$ is irreducible,
there is a unique generic Oxley--Wang matroid $\delta U_{d,n}$.
In this subsection, we give a description of the bases of $\delta U_{d,n}$, which we use to provide alternate proofs of both \Cref{thm:D} and \Cref{thm:indepA}.
As a consequence we obtain an alternative proofs for the picture theorem for uniform hypergraphs (\Cref{thm:exists_sharp_scene}) and \Cref{cor:non-triv.scene},
since both results are direct consequences of \Cref{thm:indepA}.

For this proof, we require the following famous result of Hall \cite{hall1987representatives}.

\begin{theorem}[Hall's marriage theorem]\label{thm:hall}
    Let $\mathcal{X} = \{X_1,\ldots,X_k\}$ be a collection of subsets of $[n]$.
    Then the following conditions are equivalent:
    \begin{itemize}
        \item[\rm (i)] there exists a transversal of $\mathcal{X}$ -- i.e.~distinct elements $x_1,\ldots,x_k \in [n]$ such that $x_i \in X_i$ for each $i \in [k]$;
        \item[\rm (ii)] for any $\mathcal{X}' \subseteq \mathcal{X}$,
        we have
        \begin{equation*}
            |\mathcal{X}'| \leq \left| \bigcup_{X_i \in \mathcal{X}'} X_i \right|.
        \end{equation*}
    \end{itemize}
\end{theorem}

\begin{theorem}\label{thm:picturethm_OW}
    Let $A\subseteq \textstyle \binom{[n]}{d}$ with $\lvert A\rvert = n-d$. Then $A$ forms a basis of $\delta U_{d,n}$ if and only if
    \begin{equation}\label{eq:uniform-condition}
        \lvert B\rvert \le \lvert V(B)\rvert - d \quad \text{for all non-empty } B\subseteq A,
    \end{equation}
    where $V(B) \subseteq [n]$ denotes the union of the circuits in $B$.

\end{theorem}

\begin{proof}
If $A$ is a basis of $\delta U_{d,n}$, then \Cref{lem: necessary for independent} implies that~\eqref{eq:uniform-condition} holds.
Conversely, assume~\eqref{eq:uniform-condition} holds. Write $A = \{c_1,\ldots,c_{n-d}\}$ and set
\[
I := \bigcup_{i\in [n-d]} \{(i,j) : j \in c_i\} \subset [n-d] \times [n].
\]
By \Cref{thm: fixed support}, $A$ is a basis of $\delta U_{d,n}$ if and only if we can construct a matrix $\mathbf{w} \in \mathbb{C}^{(n-d)\times n}$ with $\supp(\mathbf{w}) = I$ that is also a realisation of $U_{n-d,n} = U_{d,n}^{\ast}$.

Choose any generic point $(x_{i,j})_{(i,j) \in I}$ of $\mathbb{C}^I$.
Let $\mathbf{w}$ be the $(n-d) \times n$-matrix where the $(i,j)$ entry is $x_{i,j}$ if $(i,j) \in I$, and $0$ otherwise.
With this set-up, it is well-known that the bases of $\mathcal{M}(\mathbf{w})$ (the column matroid of $\mathbf{w}$) are exactly the transversals of $A$ (see for example \textup{\cite[\S 2.2]{bonin2010introduction}} or \cite{Edmonds1967}).
Thus $\mathcal{M}(\mathbf{w}) = U_{n-d,n}$ if and only if for every subset $D \subset [n]$ of size $n-d$, the collection of sets $A|_D := \{c_i \cap D : i \in [n-d]\}$ has a transversal.

Choose any set $D \subset [n]$ of size $n-d$.
By \Cref{thm:hall},
$A|_D$ has a transversal (and hence $D$ is a basis of $\mathcal{M}(\mathbf{w})$) if and only if for every subset $B \subseteq A$ we have
\begin{equation*}
    |B| \leq \left| \bigcup_{c_i \in B} (c_i \cap D)  \right| = \left| D \cap \bigcup_{c_i \in B} c_i    \right| = \left| V(B) \cap D \right|.
\end{equation*}
As $|[n]\setminus D| = d$, we have for each $B \subseteq A$ that
\begin{equation}\label{eq:picturethm_OW}
    \Big(d - \big\lvert V(B) \cap ([n] \setminus D) \big\rvert \Big) \geq \Big(d - \big\lvert [n] \setminus D \big\rvert \Big) \geq 0.
\end{equation}
Combining \eqref{eq:picturethm_OW} and \eqref{eq:uniform-condition} then shows that
\[
\lvert B \rvert \le \lvert V(B) \rvert - d = \lvert V(B) \cap D \rvert - \Big(d - \big\lvert V(B) \cap ([n] \setminus D) \big\rvert \Big) \le \lvert V(B) \cap D \rvert
\]
for every subset $B \subseteq A$.
Hence $A|_D$ has a transversal, which concludes the proof.
\end{proof}

We can now use \Cref{thm:picturethm_OW} to provide an alternative proof of \Cref{thm:D} that only uses Oxley--Wang derived matroids and does not require \Cref{thm:indepA}, \Cref{cor:non-triv.scene}, or any other result from scene analysis. We also prove~\Cref{thm:E}. We abuse terminology and say a hypergraph $H=(V,E)$ is $(1,d)$-sparse (with $d \geq 2$) if
\[
    |E'| \leq V(E')-d \qquad \text{for all $E' \subseteq E$ with $|E'| \geq 2$.}
\]

\begin{proof}[Alternative proof of \Cref{thm:D} and proof of~\Cref{thm:E}]
    By \Cref{thm:B}, $M$ is liftable if and only if $\rank_{\delta U_{d-1,n}}
   \bigl(\mathcal{C}_d(M) \bigr)
   \leq n-d$. Moreover, by~\Cref{thm:C}, $M$ is generically realisable if and only if $\mathcal{C}^{\text{ns}}(M)$ is a flat in $\delta U_{d-1,n}$.
   \Cref{thm:D} and the first part of~\Cref{thm:E} then follow from \Cref{thm:picturethm_OW}, which gives $\delta U_{d-1,n}=\mathcal{S}_{(1,d-1)}$.

Suppose now that $M$ is a sparse paving matroid of rank $d\geq 2$. We must show that $M$ is generically realisable if and only if $H_M$ is $(1,d)$-sparse. If $H_M$ is $(1,d)$-sparse, then for any non-hyperedge $f \in \binom{[n]}{d} \setminus E(H_M)$, the hypergraph $H_M+f$ is $(1,d-1)$-sparse.
    Thus $\mathcal{C}_d(M) = E(H_M)$ is a flat of $\mathcal{S}_{(1,d-1)}$ and $M$ is generically realisable by the first part of~\Cref{thm:E}.
    
    Suppose $H_M$ is not $(1,d)$-sparse.
    Then there exists a set of circuits $F = \{c_1,\ldots,c_k\}$ of $M$ for some $k \geq 2$ where $|F| = |V(F)| - d +1$ and $|F'| \leq |V(F')| - d$ for all $F' \subsetneq F$ with $|F'|\geq 2$.
    Then $F$ is a basis of $\mathcal{S}_{(1,d-1)}$.
    If $F = \binom{V(F)}{d}$ (i.e.~it induces a complete hypergraph),
    then, since $|V(F)| = |F| + d - 1 \geq d+1$,
    we have that
    \begin{equation*}
        |V(F)| - d + 1 = |F| = \binom{|V(F)|}{d} \geq |V(F)|,
    \end{equation*}
    with equality if and only if $d= 1$.
    As $d \geq 2$, it follows that $F \neq \binom{V(F)}{d}$.
    Choose any $f \in \binom{V(F)}{d} \setminus F$.
    Then
    \begin{equation*}
        \rank_{\mathcal{S}_{(1,d-1)}}(F \cup \{f\}) = \rank_{\mathcal{S}_{(1,d-1)}}(F).
    \end{equation*}

    For any $f \in \binom{V(F)}{d} \setminus C_d(M) \subseteq \binom{V(F)}{d} \setminus F$, it follows that $f\in\operatorname{cl}(F)\subseteq\operatorname{cl}(\mathcal C_d(M))$
    and so $C_d(M)$ is not a flat of the matroid $\mathcal{S}_{(1,d-1)}$.
    Thus by the first part of~\Cref{thm:E}, $M$ is not generically realisable.
\end{proof}

We can similarly provide an alternative proof of \Cref{thm:indepA}.

\begin{proof}[Alternative proof of \Cref{thm:indepA}]
    Let $H=([n],E)$ be a $d$-uniform hypergraph.
    We observe that the proof of \Cref{lem:matrices_coincide}(ii) can easily adapted to show the following:
    for any generic $\vf \in \Gamma_{U_{d-1,n}}$ and any generic picture $x$ of $H$,
    we have 
    \begin{equation*}
        \rank_{\delta U_{d-1,n}} (E) = \rank \mathcal{R}_\vf (U_{d-1,n};E) = \rank A(H,x).
    \end{equation*}
    The result now follows from \Cref{thm:picturethm_OW}.
\end{proof}

\begin{remark}\label{rem:counterexamplelist}
    \Cref{thm:E} is now the final piece for  constructing the implication diagram in \Cref{subsub:lifttoscene}.
    The implications that cannot be obtained via transitivity all do not hold.
    To see this, we observe that the following hypergraphs $H$ have the following properties:
    \begin{itemize}
        \item \Cref{figure: pictures and scenes}(b): $H$ has a general position picture with a sharp/non-trivial scene, but $H$ does not have a generic picture with a sharp/non-trivial scene.
        \item \Cref{figure: pictures and scenes}(c): $H$ has a generic picture with a non-trivial scene, but $H$ does not have a generic picture with a sharp scene.
        \item \Cref{Figure1}(d): $H$ has a generic picture with a non-trivial scene, but $H$ does not have a general position picture with a sharp scene; as such, if $H$ is associated to a paving matroid $M$, then $M$ cannot be realisable.
        \item \Cref{Figure: Pappus-NonPappus}(a): $H$ is associated to a realisable paving matroid $M$,
        but $H$ has no generic picture with a sharp scene.
        \item \Cref{Figure: Pappus-NonPappus}(b): $H$ has a general position picture with a sharp scene, but the matroid constructing $H$ is not realisable.
    \end{itemize}
    
\end{remark}

\section{The Oxley-Wang matroid for a family of paving matroids}
\label{sec:simple_paving_OW}

In this section we compute the generic Oxley--Wang derived matroids of 
a family of sparse paving matroids:
those with maximally intersecting circuits, meaning that the union of any two distinct non-spanning circuits is the entire ground set. Note that
a sparse paving matroid has maximally intersecting circuits if and only if its dual has pairwise-disjoint non-spanning circuits. 

\begin{proposition}\label{prop:irreduciblematroidvar}
    
    If $M$ has maximally intersecting circuits
    then the realisation space of $M$ is irreducible and non-empty.
    Moreover, $M$ has a unique generic Oxley--Wang derived matroid.
\end{proposition}

\begin{proof}
By~\Cref{prop:disj-circuits-irr}, we have that $\Gamma_{M^{\ast}}$ is irreducible and non-empty. The claim then follows from~\Cref{rem:dual}, 
as duals of realisable matroids are realisable, and if $M^{\ast}$ has an irreducible matroid realisation space, then so does $M$.
\end{proof}

We require the following two results regarding independent and dependent sets in Oxley--Wang derived matroids.

\begin{lemma}\label{lem: necessary for independent special case}
    Let $M$ be a rank-$d$ sparse paving matroid on $[n]$ with realisation $\vf$.
    If $A \subseteq \mathcal{C}(M)$ is independent in $\delta_{\mathrm{OW}}(\vf)$, then
    \begin{equation*}
        \size{B} \le \size{V(B)} - d  \quad \text{for all } B\subseteq A \text{ with } |B| \geq 2.
    \end{equation*}
\end{lemma}

\begin{proof}
    Choose any non-empty subset $B \subseteq A$.
    If $B$ contains a spanning circuit,
    then $\rank_M(V(B)) = d$.
    If $B$ contains two non-spanning circuits $C_1,C_2$,
    then, as all non-spanning circuits of a sparse paving matroid are hyperplanes,
    we have 
    \begin{equation*}
        \rank_M(V(B)) \geq \rank_M(C_1 \cup C_2) \geq \rank_M(C_1) +1 = d,
    \end{equation*}
    and so $\rank_M(V(B))= d$.
    The result now follows from combining these observations with \Cref{lem: necessary for independent}.
\end{proof}

\begin{lemma}\label{lem: necessary for independent special case2}
    Let $M$ be a rank-$d$ sparse paving matroid on $[n]$ with realisation $\vf$.
    Let $B \subseteq \mathcal{C}(M)$ be a set of circuits satisfying
    \begin{enumerate}
        \item[\rm{(i)}] $|B| =|V(B)| - d$;
        \item[\rm{(ii)}] $|V(B)| < n$;
        \item[\rm{(iii)}] $V(B)$ contains a non-spanning circuit $c$ of $M$ that is not in $B$.
    \end{enumerate}
    Then $B$ is dependent in $\delta_{\mathrm{OW}}(\vf)$.
\end{lemma}

\begin{proof}
    Suppose for contradiction that $B$ is independent in $\delta_{\mathrm{OW}}(\vf)$.
    Without loss of generality, 
    we may suppose $c = [d]$ and $V(B) = [|B|+d]$.
    Choose a basis $A$ of $\delta_{\mathrm{OW}}(\vf)$ containing $B$.
    By a specific ordering of the rows and columns,
    we obtain the block matrix form:
    \begin{equation*}
        \mathcal{R}_{\vf}(M;A \cup\{c\})
        =
        \begin{pmatrix}
            X_1 & X_2 & \mathbf{0} \\
            Y_1 & Y_2 & Z \\
            w & 0 \cdots 0 & 0 \cdots 0 \\
        \end{pmatrix}.
    \end{equation*}
    The block columns of $\mathcal{R}_{\vf}(M;A \cup\{c\})$ are -- from left to right -- $c=[d]$, $V(B)\setminus c$ and $[n]\setminus V(B)$;
    and the block rows of $\mathcal{R}_{\vf}(M;A \cup\{c\})$ are -- from top to bottom -- $B$, $A\setminus B$ and $\{c\}$.
    The matrix $\mathcal{R}_{\vf}(M;A \cup\{c\})$ and the submatrix $\mathcal{R}_{\vf}(M;A)$ both have rank $n-d$ as $A$ is a basis of $\OW$.

    As $B$ is independent in $\delta_{\mathrm{OW}}(\vf)$ and $|B|=|V(B)|-d$,
    it follows from \Cref{lem: necessary for independent} that $\rank_M(V(B)) = d$.
    It also follows from \Cref{lem: necessary for independent} that $B \cup \{c\}$ is dependent in $\OW$ since 
    \begin{equation*}
        |B \cup \{c\}| = |B| + 1 = |V(B)| - d + 1 = |V(B \cup \{c\})| - \rank_M(V(B \cup \{c\})) + 1.
    \end{equation*}
    Thus there exists a non-zero left-kernel row vector $(\lambda , \mathbf{0}, \mu) \in \mathbb{C}^{B} \times \mathbb{C}^{A \setminus B} \times \mathbb{C} \cong \mathbb{C}^{A \cup \{c\}}$ of the matrix $\mathcal{R}_{\vf}(M;A \cup\{c\})$ with $\mu \neq 0$.
    Since
    \begin{equation*}
        \begin{pmatrix}
            \lambda & \mathbf{0} & \mu
        \end{pmatrix}
        \begin{pmatrix}
            X_1 & X_2 & \mathbf{0} \\
            Y_1 & Y_2 & Z \\
            w & 0 \cdots 0 & 0 \cdots 0 \\
        \end{pmatrix}
        =
        \begin{pmatrix}
            \lambda X_1 +\mu w & \lambda X_2 & \mathbf{0}
        \end{pmatrix}
        =
        \mathbf{0},
    \end{equation*}
    we see that the matrix $X_2$ is singular, since $X_2$ being non-singular would force both $\lambda$ and $\mu$ to be zero.
    Thus there exists a column vector $z \in \mathbb{C}^{V(B) \setminus c}$ for which $X_2 z = \mathbf{0}$.

    As observed in \Cref{rem:rowvcolmatroid}, the column matroid of
    $\mathcal{R}_{\vf}(M;A)$ is $M^*$.
    As $c$ is a non-spanning circuit of $M$ and $M$ is sparse paving,
    $[n] \setminus c$ is a non-spanning circuit of $M^*$.
    Since $[n]\setminus V(B)$ is contained in $[n]\setminus c$ as a proper subset,
    it is independent in $M^*$.
    It thus follows that $Z$ is invertible.
    As
    \begin{equation*}
        \begin{pmatrix}
            X_1 & X_2 & \mathbf{0} \\
            Y_1 & Y_2 & Z \\
            w & 0 \cdots 0 & 0 \cdots 0 \\
        \end{pmatrix}
        \begin{pmatrix}
            \mathbf{0} \\
            z \\
            Z^{-1}Y_2z
        \end{pmatrix}
        =
        \mathbf{0},
    \end{equation*}
    the set $V(B) \setminus c$ is dependent in $M^*$.
    Since $|V(B)|< n$,
    this implies that $[n] \setminus c$ contains a dependent set (namely, $V(B) \setminus c$) as a proper subset,
    which contradicts the circuit axioms of a matroid.
    This contradiction now concludes the proof.
\end{proof}

\begin{remark}
    It is important in \Cref{lem: necessary for independent special case2} that the union of the circuits in $B$ is not equal to the ground set.
    To see why, let $M$ be the sparse paving matroid of rank $3$ on $[6]$ with unique non-spanning circuit $\{1,2,3\}$. Then \[B=\{\{4,5,6,1\},\{4,5,6,2\},\{4,5,6,3\}\}\] is a set of circuits that satisfies $\size{V(B)}=\size{B}+d$ and $V(B)$ contains the non-spanning circuit $\{1,2,3\}$ of $M$. However, one can see that $B$ is independent in $\OW$ since each element of $B$ has an element of $[6]$ that is not present in the other elements of $B$. More precisely, if we denote $B:=\{c_1,c_2,c_3\}$, and we form the $3\times 6$ matrix with rows $r_{c_1},r_{c_2}$ and $r_{c_3}$, then the restriction of this matrix to the columns $1,2,3$ gives a diagonal invertible matrix. 
    Hence $B$ is independent.
\end{remark}

\begin{lemma}\label{lem:transversal_harder}
    Let $A = \{c_1,\ldots,c_{n-d}\}$ be a collection of subsets of $[n]$ with $|c_i| \in \{d,d+1\}$ that satisfies
    \begin{equation}\label{eq:transversal_harder}
        \size{B} \le \size{V(B)} - d  \quad \text{for all } B\subseteq A \text{ with } |B| \geq 2.
    \end{equation}
    If $D \subset [n]$ has size $n-d$ and $c_i \cap D \neq \emptyset$ for each set $c_i \in A$ of size $d$,
    then $A|_D := \{c_i \cap D : i \in [n-d]\}$ has a transversal.
\end{lemma}

\begin{proof}
    By \Cref{thm:hall},
    $A|_D$ has a transversal if and only if for every subset $B \subseteq A$ we have
    \begin{equation*}
        |B| \leq \left| \bigcup_{c_i \in B} (c_i \cap D)  \right| = \left| D \cap \bigcup_{c_i \in B} c_i    \right| = \left| V(B) \cap D \right|.
    \end{equation*}
    As $|[n]\setminus D| = d$, we have for each $B \subseteq A$ that
    \begin{equation}\label{eq:transversal_harder1}
        \Big(d - \big\lvert V(B) \cap ([n] \setminus D) \big\rvert \Big) \geq \Big(d - \big\lvert [n] \setminus D \big\rvert \Big) \geq 0.
    \end{equation}
    Combining \eqref{eq:transversal_harder1} and \eqref{eq:transversal_harder1} then shows that
    \[
    \lvert B \rvert \le \lvert V(B) \rvert - d = \lvert V(B) \cap D \rvert - \Big(d - \big\lvert V(B) \cap ([n] \setminus D) \big\rvert \Big) \le \lvert V(B) \cap D \rvert
    \]
    for every subset $B \subseteq A$ with at least 2 elements.
    Now suppose that $B =\{c_i\}$ for some $i \in [n-d]$.
    If $|c_i| = d+1$ then $c_i \cap D \neq \emptyset$ as $|[n]\setminus D| = d$,
    and if $|c_i| = d$ then $c_i \cap D \neq \emptyset$ by assumption;
    in either case, $|B| \leq | V(B) \cap D  |$.
    Hence $A|_D$ has a transversal, which concludes the proof.
\end{proof}

We will additionally require the following result of Frobenius describing an exact combinatorial condition for the determinant of a specific type of symbolic matrix to have an irreducible determinant.

\begin{theorem}[Frobenius \cite{frobenius1917}; see also {\cite[pg.~142]{Schneider}}]\label{thm:frobenius}
    Let $I \subset [n] \times [n]$ and let $M$ be an $n \times n$ symbolic matrix where each entry $M_{i,j}$ is either an independent variable $x_{i,j}$ or 0.
    Then the polynomial $\det M \in \mathbb{C}[x_{i,j}:(i,j) \in I]$ is irreducible if and only if for all $0<m<n$ and all subsets $S_1 \in \binom{[n]}{m}$ and $S_2 \in \binom{[n]}{n-m}$, there exists $i \in S_1$ and $j \in S_2$ such that $M_{i,j} \neq 0$.
\end{theorem}

\begin{remark}
    A symbolic matrix with irreducible determinant, as described in \Cref{thm:frobenius}, is said to be \emph{fully indecomposable}.
    We refer an interested reader to \cite{Schneider} for more details.
\end{remark}

We are now ready to prove \Cref{thm:F}, which we recall is the following.

\begin{theorem}\label{thm:maxint}
    Let $M$ be a rank-$d$ sparse paving matroid with maximally intersecting circuits.
    Then a set $A \subseteq \mathcal{C}(M)$ is a basis of the generic Oxley--Wang derived matroid for $M$ if and only if $|A|=n-d$ and for all proper subsets $B \subsetneq A$ with $|B| \geq 2$ we have
    \begin{equation}\label{eq:maxint}
    \size{B} \le
    \begin{cases}
        \size{V(B)} - d -1 &\text{if $V(B)$ contains a non-spanning circuit $c$ of $M$ not contained in $B$}, \\
        \size{V(B)} - d &\text{otherwise},
    \end{cases}
    \end{equation}
    where $V(B) \subseteq [n]$ denotes the union of the circuits in $B$.
\end{theorem}

\begin{proof}
    As $\rank (\delta M)= \rank (M^*) = n-d$ (see \Cref{rem:rowvcolmatroid}), the necessity of~\eqref{eq:maxint} follows from \Cref{lem: necessary for independent special case} and \Cref{lem: necessary for independent special case2}.
    
    For the converse, suppose \eqref{eq:maxint} holds for a chosen set $A \subseteq \mathcal{C}(M)$ with $|A|=n-d$. We need to show that $A$ is a basis of the Oxley--Wang derived matroid of $M$.
    
    Write $A = \{c_1,\ldots,c_{n-d}\}$ and set
    \[
    I := \bigcup_{i\in [n-d]} \{(i,j) : j \in c_i\} \subset [n-d] \times [n].
    \]
    By \Cref{thm: fixed support}, $A$ is a basis of $\delta M$ if and only if we can construct a matrix $\mathbf{w} \in \mathbb{C}^{(n-d)\times n}$ with $\supp(\mathbf{w}) = I$ that is also a realisation of $M^*$.

    Observe here that each non-spanning circuit of $M^*$ contains exactly $n-d$ elements.
    For each $\kappa \in \mathcal{C}^{\text{ns}}(M^*)$,
    let $I_{\kappa} := \{(i,j) \in I: j \in \kappa\}$.
    Now define $\mathbf{v}_{\kappa}$ to be the symbolic $[n-d] \times [n-d]$ matrix with columns indexed by $\kappa$ where entry $(i,j)$ is the monomial $x_{i,j}$ if $(i,j)\in I_{\kappa}$ and $0$ otherwise,
    and define $f_\kappa \in \mathbb{C}[x_{i,j}:(i,j)\in I]$ to be the polynomial $f_{\kappa}(x) := \det \mathbf{v}_{\kappa}$.
    With this we define the rationally-defined variety
    \begin{equation*}
        X := \Big\{ x \in \mathbb{C}^I : f_{\kappa}(x) = 0 \text{ for each } \kappa \in \mathcal{C}^{\text{ns}}(M^*) \Big\}.
    \end{equation*}
    Now fix $J := \{(i,j) : j \notin \kappa \text{ for any } \kappa \in \mathcal{C}^{\text{ns}}(M^*)\}$.
    Since the non-spanning circuits of $M^*$ are disjoint, the sets $\{I_{\kappa} : \kappa \in \mathcal{C}^{\text{ns}}(M^*)\}$ are pairwise disjoint and partition $I \setminus J$.
    Thus if we define for each $\kappa \in \mathcal{C}^{\text{ns}}(M^*)$ the rationally-defined variety
    \begin{equation*}
        X_\kappa := \Big\{ x \in \mathbb{C}^{I_\kappa} : f_{\kappa}(x) = 0\Big\},
    \end{equation*}
    then $X = \mathbb{C}^J \times \prod \{ X_\kappa : \kappa \in \mathcal{C}^{\text{ns}}(M^*)\}$.
    
    \begin{claim}\label{claim:maxint}
        Let $\kappa$ be a non-spanning circuit of $M^*$,
        and for any $t \in \kappa$ let $I_{\kappa\setminus t} := \{(i,j) \in I_{\kappa}: j \neq t\}$. 
        Then $X_{\kappa}$ is irreducible,
        and for any $t \in \kappa$,
        the map
        \begin{equation*}
            \pi_{t} : X_{\kappa} \longrightarrow \mathbb{C}^{I_{\kappa\setminus t}}, ~ (x_{i,j})_{(i,j) \in I_{\kappa}} \longmapsto (x_{i,j})_{(i,j) \in I_{\kappa\setminus t}}
        \end{equation*}
        is a dominant map.
        In particular, if $x$ is a generic point of $X_{\kappa}$, then $\pi_t(x)$ is a generic point of $\mathbb{C}^{I_{\kappa\setminus t}}$.
    \end{claim}

    \begin{claimproof}
        If $f_\kappa(x) = 0$ for all $x \in \mathbb{C}^{I_\kappa}$ then $X_{\kappa} = \mathbb{C}^{I_\kappa}$.        
        Similarly, if $[n] \setminus \kappa \in A$ then, as the row for $[n]\setminus \kappa$ in the symbolic matrix $\vf_\kappa$ contains only zeroes, $X_{\kappa} = \mathbb{C}^{I_\kappa}$.
        From this, it is immediate that we are done.
        Suppose instead that $X_{\kappa} \neq \mathbb{C}^{I_\kappa}$ and $[n] \setminus \kappa \notin A$.

        We first prove that $X_{\kappa}$ is irreducible, which is equivalent to showing $f_{\kappa}$ is an irreducible polynomial.
        Suppose for contradiction that $f_{\kappa}$ is reducible.
        Then by applying \Cref{thm:frobenius} to the symbolic matrix $\vf_\kappa$,
        we see that there exists a proper subset $\emptyset \neq B \subsetneq A$ and a set $S_2 \in \binom{\kappa}{n-d-|B|}$ where $(i,j) \notin I_\kappa$ for each $i \in B$ and each $j \in \kappa$;
        equivalently, $|V(B) \cap \kappa| \leq |B|$ (as we then choose $S_2 \subseteq \kappa \setminus V(B)$).
        Following the proof of \Cref{lem:transversal_harder} with $D = \kappa$ shows that we must have  $|V(B) \cap \kappa| = |B| = |V(B)|-d$.
        Thus $V(B)$ contains $[n] \setminus \kappa$.
        Since $M^*$ is a sparse paving matroid,
        $[n] \setminus \kappa$ is a non-spanning circuit of $M$.
        As $B$ is a proper subset of $A$ satisfying $|B|=|V(B)|-d$,
        we must have that $|V(B)|<n$,
        or else $|A|> |B| = n-d$.
        However, our choice of $B$ now contradicts \eqref{eq:maxint}, and thus $X_\kappa$ is irreducible.

        Suppose now for contradiction that $\pi_t$ is not dominant for some choice of $t \in \kappa$.
        Then there exists a polynomial $g \in \mathbb{C}[x_{i,j} : (i,j) \in I_{\kappa \setminus t}]$ for which $g(x) = 0$ for all $x \in X_{\kappa}$.
        As $X_{\kappa}$ is a hypersurface defined by $f_{\kappa}$, this implies $f_{\kappa} | g$.
        As $f_{\kappa}$ contains each variable $x_{i,j}$ with $(i,j) \in I$ in some monomial,
        it can only divide $g$ if $I_{\kappa} = I_{\kappa \setminus t}$. 
        This implies $f_{\kappa}(x) = 0$ for all $x \in \mathbb{C}^{I_\kappa}$ as the symbolic matrix $\vf_\kappa$ has only zero entries in column $t$.
        However, this now contradicts our earlier assumption that $X_{\kappa} \neq \mathbb{C}^{I_\kappa}$.
    \end{claimproof}
    
    Choose a generic point $x$ of $X$ and fix $\mathbf{w}$ be the $(n-d) \times n$-matrix where the $(i,j)$ entry is $x_{i,j}$ if $(i,j) \in I$, and $0$ otherwise.
    
    Choose any subset $D \subset [n]$ of size $n-d$ that does not contain a non-spanning circuit of $M^*$.
    It follows from \Cref{claim:maxint} that, given $I_D:= \{(i,j) \in I : j \in D\}$, the vector $(x_{i,j})_{(i,j)\in I_D}$ is a generic point of $\mathbb{C}^{I_D}$.
    Thus the columns of $\mathbf{w}$ corresponding to the set $D$ are linearly independent if and only if $D$ is a transversal of $A$ (see for example \textup{\cite[\S 2.2]{bonin2010introduction}} or \cite{Edmonds1967}).
    By \Cref{lem:transversal_harder}, $D$ is indeed a transversal of $A$, which concludes the proof.
\end{proof}

A direct application of \Cref{thm:B}, \Cref{thm:C} and \Cref{thm:F} provides the following characterisation for liftability and generic realisability for matroid which truncate to a sparse paving matroid with maximally intersecting circuits. 

\begin{corollary}
    Let $M$ be a rank-$d$ matroid whose truncation $T(M)$ is a sparse paving matroid with maximally intersecting circuits. Let $\mathcal{I}$ be the collection of all subsets $A$ of $\mathcal{C}(T(M))$ for which \eqref{eq:maxint} holds for all proper subsets $B\subsetneq A$ with $\size{B}\geq 2$.
    \begin{enumerate}
        \item There exists a matroid $\mathcal{M}$ with ground set $\mathcal{C}(T(M))$ and $\mathcal{I}(\mathcal{M}) = \mathcal{I}$.
        \item $M$ is liftable if and only if $\mathcal{C}^{\text{ns}}(M)$ does not contain a basis of $\mathcal{M}$.

        \item $M$ is generically realisable if and only if $\mathcal{C}^{\text{ns}}(M)$ is a flat of $\mathcal{M}$.
    \end{enumerate}
\end{corollary}

\subsection*{Acknowledgements}

The authors were supported by the KU Leuven grant iBOF/23/064 and the FWO grants G0F5921N (Odysseus) and G023721N.
E.\,L.\ was additionally supported by the FWO PhD fellowship 1126125N.

\printbibliography

\appendix

\section{\texorpdfstring{$M(K_n)$}{M(Kn)} is the unique element in its truncation's 1-elevation lattice}\label{app:kn}

Throughout the appendix we fix $M_K = M(K_n)$ for some $n \geq 5$.
We also fix $N_K$ to be the rank-3 sparse paving matroid on $E(K_n)$ with circuit hyperplanes being exactly the edge sets corresponding to 3-cycles.

For $k \geq 1$, we say that a matroid $M_2$ is a \emph{$k$-elevation} of a matroid $M_1$ if, given $T^k$ is the truncation operation applied $k$ times, $T^k M_2 = M_1$.
An \emph{elevation} of a matroid $M_1$ is then any $k$-elevation for any $k \geq 1$.
In contrast to the 1-elevations,
the set of all elevations of a matroid along with the original matroid itself is not always a lattice with respect to the weak order;
indeed, it may even lack a unique maximal element even though it is finite.
Fortunately, this latter problem is not an issue in the case we care about.

\begin{theorem}[{\cite[Theorem 4.3(a)]{jackson2024maximal}}]\label{thm:app:billshin}
    The matroid $M_K$ is the unique maximal elevation of $N_K$.
\end{theorem}

\begin{lemma}\label{lem:app:max}
    Let $N$ be a matroid, and suppose that $M \neq N$ is the unique maximal elevation of $N$.
    Then $M$ is the unique maximal element of $\Lat(TM)$.
\end{lemma}

\begin{proof}
    Choose any non-trivial 1-elevation $M_1$ of $TM$.
    We observe here that $TM$ is an elevation of $N$ since $M$ is an elevation of $N$, and $M \neq N$.
    Likewise, $M_1$ is an elevation of $N$ since $TM$ is an elevation of $N$.
    Thus $M_1 \leq M$,
    and so every 1-elevation of $TM$ is less than or equal to $M$ in the weak order.
    As $M$ is a 1-elevation of $TM$,
    it follows that $M$ is the maximal element of $\Lat(TM)$.
\end{proof}

For the following results, we fix an integer $n$ and consider the complete graph $K_n$.

\begin{lemma}\label{lem:app:cycle}
    Let $M$ be a 1-elevation of $T(M_K)$.
    If $C$ is a cycle of $K_n$,
    then $C$ is a circuit of $M$.
\end{lemma}

\begin{proof}
    As $T(M_K)$ is an elevation of $N_K$,
    every 3-cycle is a circuit of $M$.
    We now inductively suppose that the result holds for any $(k-1)$-cycle and select any $k$-cycle $C$.
    We label the vertices for $C$ by $\{v_1,\ldots,v_k\}$, where $v_i$ is adjacent to $v_{i+1}$ for each $i \in [k-1]$ and $v_k$ is adjacent to $v_1$.
    Fix $C'$ to be the cycle with vertices $\{v_1,\ldots,v_{k-1}\}$ where $v_i$ is adjacent to $v_{i+1}$ for each $i \in [k-2]$ and $v_{k-1}$ is adjacent to $v_1$,
    and fix $C'' = \{ v_{k-1}v_k, v_1 v_k, v_1 v_{k-1}\}$.
    Then both $C',C''$ are both circuits of $M$.
    Since $C'$ and $C''$ share exactly 1 edge and $C$ is the symmetric difference of the two, $C$ is also a circuit of $M$.
    Induction now implies the desired result.
\end{proof}

\begin{lemma}\label{lem:app:path}
    Let $M$ be a 1-elevation of $T(M_K)$.
    If $P$ is a path in $K_n$,
    then $P$ is an independent subset of $M$.
\end{lemma}

\begin{proof}
    Choose any path $P$ and fix $C$ to be the cycle formed by adding a single edge to connect the head and tail of the path.
    By \Cref{lem:app:cycle}, $C$ is a circuit.
    As $P \subsetneq C$, $P$ is independent in $M$. 
\end{proof}

\begin{lemma}\label{lem:app:triangleop}
    Let $T$ a tree with vertex set $[n]$.
    Then there exists a sequence of trees $T_1,\ldots,T_k$ such that: $T_1 = T$; $T_k$ is a path; and $T_{i+1}$ is formed from $T_i$ by selecting two adjacent edges $v_1 v_2, v_2 v_3$,
    removing the edge $v_2 v_3$,
    then adding the edge $v_1 v_3$.
\end{lemma}

\begin{proof}
    For a tree $T_i$, we set $a(T_i)$ to be the smallest distance from a leaf to a vertex of degree 3 or higher, and we set $b(T_i)$ to be the number of vertices $v$ of $T_i$ of degree 3 such that the distance from $v$ to a leaf is $a(T_i)$.
    The creation of our sequence $T_1,\ldots,T_k$ follows from a 2-step algorithm:
    \begin{enumerate}
        \item If $T_i$ contains a vertex $v$ adjacent to distinct vertices $u_1,u_2,w$ with $u_1,u_2$ being leaves:
        remove $u_1 v$ and add $u_1 u_2$.
        \item If $T_i$ is not a path and $a(T_i) = k \geq 2$:
        select a path $v_0,\ldots,v_k$ from a vertex $v_0$ with distinct neighbours $w_1,w_2,v_1$ to a leaf $v_k$,
        remove the edge $v_0 w_1$ and add the edge $v_1 w_2$.
    \end{enumerate}
    Both operations are illustrated in \Cref{fig:treesurgery}.
    The first operation can only be implemented when $T_i$ is not a path and $a(T_i)=1$, and either preserves or reduces the number of leaves of the tree.
    The second operation can only be implemented if $T_i$ is not a path and $a(T_i) \geq 2$. The operation reduces the value $a(T_i)$ whilst preserving the number of the leaves the tree.
    Thus following the algorithm will obtain a path from any tree.
\end{proof}

\begin{figure}[tp]
	\begin{center}
        \begin{tikzpicture}[scale=1]
        \begin{scope}[xshift=0]
			\node[vertex] (u2) at (-1,-0.5) {};
			\node[vertex] (w) at (1,-0.5) {};
			\node[vertex] (v) at (0,0) {};
			\node[vertex] (u1) at (0,1) {};

			\node (u2label) at (-1,-0.75) {$u_2$};
			\node (wlabel) at (1,-0.75) {$w$};
			\node (vlabel) at (0.2,0.18) {$v$};
			\node (u1label) at (0,1.25) {$u_1$};
			
			\draw[edge] (v)edge(u2);
			\draw[edge] (v)edge(w);
			\draw[edge] (v)edge(u1);
			
			\draw[dedge] (w)edge(1.6,-0.3);
			\draw[dedge] (w)edge(1.6,-0.7);
        \end{scope}
        \begin{scope}[xshift=60]
			\draw[thick,->]  (0,0) -- (1,0);
        \end{scope}
        \begin{scope}[xshift=150]
			\node[vertex] (u2) at (-1,-0.5) {};
			\node[vertex] (w) at (1,-0.5) {};
			\node[vertex] (v) at (0,0) {};
			\node[vertex] (u1) at (0,1) {};
            
			\node (u2label) at (-1,-0.75) {$u_2$};
			\node (wlabel) at (1,-0.75) {$w$};
			\node (vlabel) at (0.2,0.18) {$v$};
			\node (u1label) at (0,1.25) {$u_1$};
			
			\draw[edge] (v)edge(u2);
			\draw[edge] (v)edge(w);
			\draw[gedge] (v)edge(u1);
			\draw[oedge] (u1)edge(u2);
			
			\draw[dedge] (w)edge(1.6,-0.3);
			\draw[dedge] (w)edge(1.6,-0.7);
        \end{scope}
        
        \begin{scope}[xshift=-20,yshift=-75]
			\node[vertex] (v0) at (0,0) {};
			\node[vertex] (v1) at (1,0) {};
			\node[vertex] (v2) at (2,0) {};
			\node[vertex] (w1) at (-1,0.75) {};
			\node[vertex] (w2) at (-1,-0.75) {};

			\node (w1label) at (-1,1) {$w_1$};
			\node (w2label) at (-1,-1) {$w_2$};
			\node (v0label) at (0.2,-0.2) {$v_0$};
			\node (v1label) at (1,-0.2) {$v_1$};
			\node (v2label) at (2,-0.2) {$v_2$};
			
			\draw[edge] (v0)edge(w1);
			\draw[edge] (v0)edge(w2);
			\draw[edge] (v0)edge(v1);
			\draw[edge] (v1)edge(v2);
			
			\draw[dedge] (w1)edge(-1.6,0.65);
			\draw[dedge] (w1)edge(-1.6,0.85);
			\draw[dedge] (w2)edge(-1.6,-0.65);
			\draw[dedge] (w2)edge(-1.6,-0.85);
        \end{scope}
        \begin{scope}[xshift=60,yshift=-75]
			\draw[thick,->]  (0,0) -- (1,0);
        \end{scope}
        \begin{scope}[xshift=150,yshift=-75]
			\node[vertex] (v0) at (0,0) {};
			\node[vertex] (v1) at (1,0) {};
			\node[vertex] (v2) at (2,0) {};
			\node[vertex] (w1) at (-1,0.75) {};
			\node[vertex] (w2) at (-1,-0.75) {};

			\node (w1label) at (-1,1) {$w_1$};
			\node (w2label) at (-1,-1) {$w_2$};
			\node (v0label) at (0.2,-0.2) {$v_0$};
			\node (v1label) at (1,-0.2) {$v_1$};
			\node (v2label) at (2,-0.2) {$v_2$};
			
			\draw[gedge] (v0)edge(w1);
			\draw[edge] (v0)edge(w2);
			\draw[edge] (v0)edge(v1);
			\draw[edge] (v1)edge(v2);
			\draw[oedge] (v1)edge(w1);
			
			\draw[dedge] (w1)edge(-1.6,0.65);
			\draw[dedge] (w1)edge(-1.6,0.85);
			\draw[dedge] (w2)edge(-1.6,-0.65);
			\draw[dedge] (w2)edge(-1.6,-0.85);
        \end{scope}
        \begin{scope}[xshift=-100,yshift=10]
			\node (1) at (0,0) {Step 1};
        \end{scope}
        \begin{scope}[xshift=-100,yshift=-75]
			\node (1) at (0,0) {Step 2};
        \end{scope}
        \begin{scope}[xshift=250]
			\node (1) at (0,0) {};
        \end{scope}
		\end{tikzpicture}
	\end{center}
    \caption{The two possible steps for the algorithm described in \Cref{lem:app:triangleop}, with added edge in orange and removed edge in grey.
    Although both $v,v_0$ are pictured with exactly 3 neighbours, both vertices are allowed to have additional neighbours.}\label{fig:treesurgery}
\end{figure}
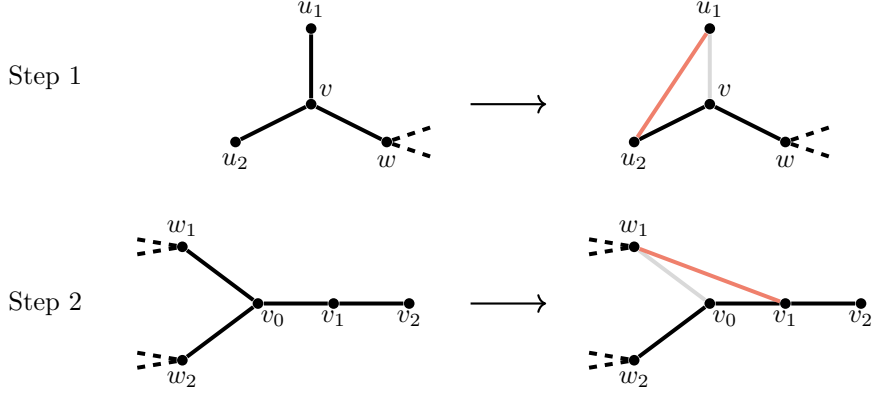

\begin{proposition}
    $M_K$ is the unique 1-elevation of $T(M_K)$.
\end{proposition}

\begin{proof}
    Choose any 1-elevation $M$ of $T(M_K)$.
    By \Cref{thm:app:billshin} and \Cref{lem:app:max},
    $M \leq M_K$.
    Every forest in $K_n$ that is either disconnected or not a spanning subgraph is independent in $T(M_K)$,
    and thus independent in $M$ also.
    If every spanning tree in $M$ is independent, then $M$ and $M_K$ share the same bases, and thus are equal.
    Suppose for contradiction that $T$ is the set of edges of a spanning tree of $M$ that is not independent in $M$.
    As every subgraph $T$ is independent in $M$,
    $T$ is a circuit of $M$.

    Let $T_1,\ldots,T_k$ be the sequence described in \Cref{lem:app:triangleop} with $T_1 = T$ and $T_k = P$ for some path $P$.
    We observe here that each $T_{i+1}$ is formed from $T_i$ by taking some 3-cycle $C$ that shares two edges with $T_i$, choosing an elements in $T_i\cap C$ to remove, then adding in the unique element in $C \setminus T_i$.
    Since $T$ is a circuit of $M$ and any 3-cycle of $M$ is also a circuit of $M$ (\Cref{lem:app:cycle}),
    it follows inductively, by the circuit elimination axiom, that each $T_i$ is a circuit.
    In particular, it follows that $P$ is a circuit.
    However, this now contradicts \Cref{lem:app:path}.
\end{proof}

\section{Paving matroids with disjoint non-spanning circuits}
\label{Sec:AppB}

This appendix is dedicated to proving \Cref{cor: Lat N}. First, recall that a cyclic flat is a flat that is a union of circuits.

\begin{lemma}\label{lem: collection of flats is admissible}
Let $M\in\Lat(N)$, and let $\mathcal{F}$ denote the collection of cyclic flats of $M$ of rank $d-1$. Then $\mathcal{F}$ is $N$-admissible.
\end{lemma}

\begin{proof}
We must show that $\mathcal{F}$ satisfies Conditions~\ref{ad1}-\ref{ad4}. Condition~\ref{ad1} follows from the fact that each $H\in\mathcal{F}$ has
rank $d-1$ and is dependent, and hence $|H|\geq d$. To verify Condition~\ref{ad2}, suppose, for contradiction, that there
exist distinct $H_1,H_2\in\mathcal{F}$ such that $H_1\cap H_2$ contains
a basis $B$ of $N$. Then $\rank_M(B)=d-1$. Since $H_1$ and $H_2$ are
flats of $M$ of rank $d-1$ containing $B$, we have $H_1=\textup{cl}_{M}(B)=H_2$, which contradicts $H_1$ and $H_2$ being distinct. To verify Condition~\ref{ad3}, consider $H\in \F$ and $j\in [r]$ such that $c_{j}\setminus \{x\}\subseteq H$ for some $x\in c_{j}$. Since $c_{j}$ is a circuit and $c_{j}\setminus \{x\}\subseteq H$, it follows that 
\[x\in \textup{cl}_{M}(c_j\setminus \{x\})\subseteq\textup{cl}_{M}(H)=H.\]
Finally, to verify that Condition~\ref{ad4} holds, note that any proper dependent subset of $H=c_j\cup \{x\}$ which is dependent in $M$, is also dependent in $N$. Since by assumption the $d-1$-circuits of $N$ are distinct, the only proper dependent set of $H=c_j\cup \{x\}$ is $c_j$. Hence, if $H=c_j\cup \{x\}$, then $x$ does not belong to a circuit which is contained in $H$. It follows that $H$ cannot be a union of circuits. 
\end{proof}

Recall that 

\[
\Ccal:=
\min\left(
\{c_1, \ldots, c_r\} \cup
\binom{H_1}{d} \cup \cdots\cup
\binom{H_k}{d}
\cup
\binom{[n]}{d+1}
\right)
\]

\begin{lemma}\label{prop: admissible subset gives elevation}
The collection of subsets $\Ccal$ forms the collection of circuits of a matroid $M(\HH)$. Moreover, $M(\HH)\in \Lat(N)$ and $\HH$ is the collection of cyclic flats of rank $d-1$ of $M(\HH)$. 
\end{lemma}

\begin{proof}
We show that $\Ccal$ satisfies the circuit-exchange axiom, i.e. that for any distinct $e_1,e_2\in \Ccal$ and any $x\in e_1\cap e_2$, there exists $e_3\in \Ccal$ such that $e_3\subseteq (e_1\cup e_2)\setminus \{x\}$. If either $e_1$ or $e_2$ belongs to $\textstyle \binom{[n]}{d+1}$ then $\size{(e_1\cup e_2)\setminus \{x\}}\geq d+1$. In this case, $(e_1\cup e_2)\setminus \{x\}$ contains a $(d+1)$-element subset $D$ and the set $D$ contains an inclusion-minimal member of the collection of subsets defining $\Ccal$. Hence $D$, and therefore $(e_1\cup e_2)\setminus \{x\}$, contains an element of $\Ccal$.
If both $e_1$ and $e_2$ belong to the set $\{c_1,\ldots,c_r\}$ then it is also clear since these circuits are pairwise disjoint. 

Suppose now that $e_1\in \{c_1,\ldots,c_r\}$ and $e_2\in \cup_{i=1}^{k}\textstyle \binom{H_i}{d}$. Let $e_1:=c_j$ and $e_2$ be a $d$-element subset of $H_i$. Since $e_2$ is a minimal element in~\eqref{mathcalC} we must have $c_j\not \subset e_2$. Moreover, by~\ref{ad3} we must have $\size{e_2\setminus c_j}\geq 2$. Hence, $\size{(e_1\cup e_2)\setminus \{x\}}\geq d+1$ and $e_1\cup e_2\setminus \{x\}$ contains a $(d+1)$-subset of $[n]$ and therefore an element of $\mathcal{C}$.

It remains to consider the case where $e_1,e_2\in \cup_{i=1}^{k}\textstyle \binom{H_i}{d}$. Let say that $e_1$ and $e_2$ are $d$-subsets of $H_i$ and $H_j$ respectively. We first treat the case $i=j$. If $\size{(e_1\cup e_2)\setminus \{x\}}\geq d+1$ then the claim follows as above. On the other hand, if $\size{(e_1\cup e_2)\setminus \{x\}}=d$, then $(e_1\cup e_2)\setminus \{x\}$ is a $d$-element subset of $H_i$ so it is either an element of $\Ccal$ or it contains one. 

Assume now that $i\neq j$. First, we show that $\size{e_1\cap e_2} \neq d-1$. Suppose for a contradiction that $\size{e_1\cap e_2} = d-1$. Since $e_1,e_2\in \Ccal$, the intersection $e_1\cap e_2$ is not equal to any of the circuits $\{c_1,\ldots,c_r\}$. Consequently, $e_1\cap e_2$ is a basis of $N$, but this contradicts Conditions~\ref{ad2}.

Hence, if $i \neq j$, then $\size{e_1\cap e_2}\leq d-2$, and $\size{(e_1\cup e_2)\setminus \{x\}}\geq d+1$, so we can argue as in the case where $i=j$. Thus, $\Ccal$ defines the circuits of a matroid $M(\mathcal{H})$.

Observe that a $(d-1)$-subset of $[n]$ is independent in $M(\HH)$ if and only if it is not equal to any of the circuits $\{c_1,\ldots,c_r\}$, which is equivalent to being a basis of $N$. Hence, the truncation of $M(\HH)$ gives $N$, which shows that $M(\HH)\in \Lat(N)$.

Let $\F$ be the collection of cyclic flats of rank $d-1$ of $M(\HH)$. It remains to show that $\F=\HH$. We first show that $\HH\subseteq \F$. Let $H\in \HH$. We begin by showing that $H$ is cyclic, that is, that $H$ is a union of circuits. To this end, it is enough to show that every element $x\in H$ is contained in a circuit of $M(\HH)$ that is itself contained in $H$. By Condition~\ref{ad4}, $H$ is not of the form $c_j\cup\{x\}$ for any $j\in[r]$. Hence, we may choose a subset $A\subseteq H\setminus\{x\}$ of size $d-1$ such that $A\neq c_j$ for every $j\in[r]$. Indeed, if every $(d-1)$-subset of $H\setminus \{x\}$ were equal to one of the $c_j$, then, using that the $c_j$'s are pairwise disjoint, we would necessarily have $H\setminus \{x\}=c_j$, for some $j$, contradicting~\ref{ad4}.
Consider the $d$-element set $A\cup\{x\}\subseteq H$. If $A\cup\{x\}$ contains none of the circuits $c_1,\ldots,c_r$, then, by definition, $A\cup\{x\}$ is a circuit of $M(\HH)$, and the claim follows. Otherwise, suppose that $c_j\subseteq A\cup\{x\}$ for some $j\in[r]$. Since $|A|=d-1$ and $A\neq c_j$, it follows that $x\in c_j$. Consequently, $
x\in c_j\subseteq H$,
and the claim follows in this case as well.

The above argument also shows that $H$ contains a basis of $N$, and hence $\rank(H)=d-1$.
We now show that $H$ is a flat. Suppose, for contradiction, that $\rank(H\cup \{y\})=d-1$ for some $y\in [n]\setminus H$. Then, there exists a circuit $c\in \Ccal$ of size at most $d$ such that $y\in c$ and $c\subset H\cup \{y\}$. We can not have $c\in \{c_1,\ldots,c_r\}$ because this would contradict Condition~\ref{ad3}. Suppose then that $c$ is a $d$-subset of some element $H'\in \HH$. Since $c\in \Ccal$, we must have that $c\setminus \{y\}$ is not equal to any of the subsets $\{c_1,\ldots,c_r\}$. Hence, $c\setminus \{y\}$ is a basis of $N$. However, since $c\setminus \{y\}\subseteq H\cap H'$, this contradicts Condition~\ref{ad2}. This proves that every $H\in \HH$ is a cyclic flat of rank $d-1$ of $M(\HH)$.

It remains to show that $\F\subseteq \HH$. Let $F\in \F$. Since $F$ is cyclic, it can not be of the form $c_j\cup\{x\}$ for any $j\in[r]$ and $x\notin c_j$. Since the circuits $\{c_1,\ldots,c_r\}$ are disjoint, this guarantees that we can find a subset $A\subseteq F$ of size $d$ that does not contain any $c_j$. Indeed, assume that this is not true. Consider any $d$-subset $A\subseteq F$. Then $c_j\subset A$ for some $j$, so $A=c_j\cup \{x\}$ for some $x\notin c_j$. Since $F$ can not be equal to $c_j\cup \{x\}$, there must exist some $y\in F\setminus A$. Replace $y$ by any $z\in c_j$ to form the $d$-subset $A':=(A\cup \{y\})\setminus \{z\}$. It does not contain $c_j$ and can not contain any other $c_l$. Indeed, if $c_l\subseteq A'$, then $c_l$ must be a subset of $\{x,y\}$ given that all the other elements of $A'$ belong to $c_j$, which is disjoint from $c_l$. However, since we assume $d\geq 4$, we have $\size{c_l}\geq 3$, and this is not possible. We conclude that $F$ contains a $d$-subset that does not contain any $c_j$. Let $A$ be such a susbet. 

Since $\rank(F)=d-1$, $A$ is dependent, so it must contain a circuit from $\Ccal$. Consequently, since it does not contain any $c_j$, $A$ must be a $d$-element subset of some member of $\HH$. Let say $A\in \textstyle \binom{H}{d}$ for some $H\in \HH$. Since $A$ is a circuit of size $d$ we have $\rank(A)=d-1$, and since $F$ is a flat of rank $d-1$ containing $A$ we must have $\textup{cl}_{M(\HH)}(A)=F$. However, since $A\subseteq H$ and $H$ is also a flat of rank $d-1$ of $M(\HH)$ we must have $\textup{cl}_{M(\HH)}(A)=H$ as well. Hence, $F=H$, which proves that $\F\subseteq \HH$, completing the proof.
\end{proof}

Finally, we prove \Cref{cor: Lat N}.

\begin{proof}[Proof of \Cref{cor: Lat N}]
The inclusion $\supseteq$ in~\eqref{latn} follows from~\Cref{prop: admissible subset gives elevation}.
To show the other inclusion, let $M\in\Lat(N)$, and let $\F$ denote the collection of cyclic flats of $M$ of rank $d-1$. By~\Cref{lem: collection of flats is admissible}, the collection $\F$ is $N$-admissible. Hence, by~\Cref{prop: admissible subset gives elevation}, the matroid $M(\F)$ belongs to $\Lat(N)$ and has $\F$ as its collection of cyclic flats of rank $d-1$.
It remains to show that $M=M(\F)$. Since both $M$ and $M(\F)$ truncate to $N$, they have the same circuits of size at most $d-1$, precisely the non-spanning circuits of $N$. Moreover, the circuits of size $d$ are determined by the collection of cyclic flats of rank $d-1$: a $d$-element set is a circuit precisely when it is contained in one of these cyclic flats and does not contain a smaller circuit. Thus $M$ and $M(\F)$ have the same circuits of size at most $d$. Moreover, the circuits of size $d+1$ are determined by the circuits of size at most $d$: a $d+1$-subset is a circuit if and only if it does not contain a circuit of size at most $d$. Therefore, $M$ and $M(\F)$ have the same collection of circuits, and hence $M=M(\F)$.
Thus $M$ belongs to the right-hand side of~\eqref{latn}, proving the inclusion $\subseteq$.
\end{proof}

\end{document}